\documentclass[11pt,reqno,letterpaper]{amsart}
\usepackage{color}
\usepackage[colorlinks=true, allcolors=blue,backref=page]{hyperref}
\usepackage{amsmath, amssymb, amsthm}
\usepackage{mathrsfs}
\usepackage{mathtools}
\usepackage[noabbrev,capitalize,nameinlink]{cleveref}
\crefname{equation}{}{}
\usepackage[noadjust]{cite}
\usepackage{graphics}
\usepackage{pifont}
\usepackage{tikz}
\usepackage{bbm}
\usepackage[T1]{fontenc}

\usetikzlibrary{arrows.meta}

\usepackage{environ}
\usepackage{framed}
\usepackage{url}
\usepackage[linesnumbered,ruled,vlined]{algorithm2e}
\usepackage[noend]{algpseudocode}
\usepackage[labelfont=bf]{caption}
\usepackage{cite}
\usepackage{framed}
\usepackage[framemethod=tikz]{mdframed}
\usepackage{appendix}
\usepackage{graphicx}
\usepackage[textsize=tiny]{todonotes}
\usepackage{tcolorbox}
\usepackage{enumerate}
\allowdisplaybreaks[1]
\usepackage{enumerate}
\usepackage{stmaryrd}
\usepackage[margin=1in]{geometry}

\usepackage[shortlabels]{enumitem}
\crefformat{enumi}{#2#1#3}
\crefrangeformat{enumi}{#3#1#4 to~#5#2#6}
\crefmultiformat{enumi}{#2#1#3}
{ and~#2#1#3}{, #2#1#3}{ and~#2#1#3}

\DeclareSymbolFont{symbolsC}{U}{pxsyc}{m}{n}
\SetSymbolFont{symbolsC}{bold}{U}{pxsyc}{bx}{n}
\DeclareFontSubstitution{U}{pxsyc}{m}{n}
\DeclareMathSymbol{\medcircle}{\mathbin}{symbolsC}{7}

\crefname{algocf}{Algorithm}{Algorithms}

\crefname{equation}{}{} 
\AtBeginEnvironment{appendices}{\crefalias{section}{appendix}} 

\usepackage[color,final]{showkeys} 

\colorlet{refkey}{orange!20}
\colorlet{labelkey}{blue!30}

\crefname{algocf}{Algorithm}{Algorithms}

\numberwithin{equation}{section}
\newtheorem{theorem}{Theorem}[section]
\newtheorem{proposition}[theorem]{Proposition}
\newtheorem{lemma}[theorem]{Lemma}
\newtheorem{claim}[theorem]{Claim}
\newtheorem{observation}[theorem]{Observation}
\crefname{claim}{Claim}{Claims}

\newtheorem{corollary}[theorem]{Corollary}

\newtheorem*{question*}{Question}

\theoremstyle{definition}
\newtheorem{definition}[theorem]{Definition}

\newtheorem*{definition*}{Definition}

\theoremstyle{remark}
\newtheorem*{remark}{Remark}

\newtheorem{assumption}[theorem]{Assumption}

\newcommand{\mb}{\mathbb}
\newcommand{\mbf}{\mathbf}

\newcommand{\mc}{\mathcal}
\newcommand{\mf}{\mathfrak}
\newcommand{\mr}{\mathrm}

\newcommand{\on}{\operatorname}

\newcommand{\wt}{\widetilde}

\let\originalleft\left
\let\originalright\right
\renewcommand{\left}{\mathopen{}\mathclose\bgroup\originalleft}
\renewcommand{\right}{\aftergroup\egroup\originalright}

\allowdisplaybreaks

\newcommand{\ignore}[1]{}

\title{On the chromatic number of random triangle-free graphs}
\author{Clayton Mizgerd}
\address{University of Illinois Chicago, Department of Mathematics, Statistics, and Computer Science}
\email{cmizge2@uic.edu}
\author{Will Perkins}
\address{Georgia Institute of Technology, School of Computer Science}
\email{math@willperkins.org}
\author{Yuzhou Wang}
\address{Georgia Institute of Technology, School of Mathematics}
\email{ywang3694@gatech.edu}

\newcommand{\paren}[1]{\left(#1\right)}
\newcommand{\eps}{\varepsilon}
\renewcommand{\Pr}{\mb{P}}
\newcommand{\TRUE}{\text{TRUE}}
\newcommand{\FALSE}{\text{FALSE}}
\DeclareMathOperator{\SAT}{SAT}
\newcommand{\squig}[1]{\underset{#1}{\rightsquigarrow}}
\newcommand{\notsquig}[1]{\underset{#1}{\not\rightsquigarrow}}

\begin{document}
	
	\begin{abstract}
        We study the chromatic number of typical triangle-free graphs with $\Theta \left( n^{3/2} (\log n)^{1/2}  \right)$ edges and   establish the width of the scaling window for the transitions from $\chi = 3$ to $\chi = 4$ and from $\chi = 4$ to $\chi = 5$.

        The transition from $3$- to $4$-colorability has scaling window of width $\Theta(n^{4/3} (\log n)^{-1/3})$.  To prove this, we show a high probability  equivalence of the $3$-colorability of a random triangle-free graph at this density and the satisfiability of an instance of bipartite random $2$-SAT, for which we establish the width of the scaling window following the techniques of Bollob{\'a}s,  Borgs,  Chayes,  Kim, and Wilson.

        The transition from $4$- to $5$-colorability has scaling window of width $\Theta(n^{3/2} (\log n)^{-1/2})$.  To prove this, we show a high probability equivalence of the $4$-colorability of a random triangle-free graph at this density and the simultaneous $2$-colorability of two independent Erd\H{o}s--R\'enyi random graphs.  For this transition, we  also  establish the limiting probability of $4$-colorability inside the scaling window.
	\end{abstract}
	
	\maketitle

\section{Introduction}

One of the main topics in probabilistic combinatorics is the study of the evolution of random graphs.  This study began with work of Erd\H{o}s and R\'enyi~\cite{erdds1959random,erdos1960evolution} who proposed a systematic study of the emergence of different structural properties of random graphs as the edge density changes.  Structural properties of interest include subgraph containment, connectivity, chromatic number, existence of a perfect matching or other factors, and component structure (see~\cite{bollobas1998random,alon2016probabilistic,janson2011random} for overviews).    
To be more specific, we focus for now on the random graph $G(n,p)$; a graph on $n$ vertices in which every potential edge is included independently with probability $p$, though most of the discussion will also apply to $G(n,m)$, a uniformly random graph on $n$ vertices with $m$ edges. 

For a given structural property $\mathcal P$ (formally a set of graphs on $n$ vertices, closed under isomorphism, usually assumed to be increasing, that is, closed under adding additional edges), there are several questions one can ask.  The first is to determine a \textit{threshold function} for the property; that is, a function $p^* = p^*(n)$ so that when $p \gg p^*$, $\Pr( G(n,p) \in \mathcal P) \to 1$, and when $p \ll p^*$, $\Pr( G(n,p) \in \mathcal P) \to 0$.  For instance the function $p^* = c \log n/n$ is a threshold function for the property of connectivity for any constant $c>0$.

Once a threshold function has been found, a finer understanding can be given by determining if a property has a \textit{sharp threshold} and finding the first-order asymptotics of this threshold.  For instance $p = \log n/n$ is a sharp threshold for connectivity, since for any fixed $\eps>0$, if $p \ge (1+\eps) \log n/n$, then $G(n,p)$ is connected with high probability (henceforth abbreviated `whp'); that is, with probability $1-o(1)$ as $n \to \infty$; and if $p \le (1-\eps) \log n/n$, then whp $G(n,p)$ is not connected.  On the other hand, the property of $G(n,p)$ containing a triangle does not have a sharp threshold since the probability is bounded away from $0$ and from $1$ when $p=c/n$ for any $c>0$. See~\cite{friedgut1999sharp} for a characterization of which properties admit sharp thresholds. 

An even finer understanding of a structural property $\mathcal P$ that admits a sharp threshold comes from an understanding of the \textit{scaling window}: the interval of values of $p$ on which the probability of the property increases from some small constant $\eps$ to $1-\eps$. For connectivity, the width of the scaling window is $\Theta(1/n)$.  Note that a sharp threshold is equivalent to the width of the scaling window being of smaller order than the threshold function itself. 

Beyond simply describing more precisely how the probability of a property depends on $p$, establishing the width of the scaling window very often comes with an explanation of \textit{why} the property typically occurs.  For instance, connectivity typically occurs when the last isolated vertex disappears, and so the scaling window of the global property of connectivity is the same as the scaling window of the local property of having no isolated vertices.  

See~\cite{friedgut2005hunting,park2023threshold,perkins2025searching} for more background on thresholds, sharp thresholds, and scaling windows.

While these questions about thresholds in random graphs have been pursued in most depth for the random graph models $G(n,p)$ and $G(n,m)$, they are also interesting to ask of other models.  Examples include random geometric graphs~\cite{penrose2003random}, models with prescribed degree sequences~\cite{molloy1995critical}, and models of large social networks~\cite{durrett2010random,newman2018networks,van2024random}.  

In this paper we will focus on two models of random triangle-free graphs: a graph chosen uniformly from $\mathcal T(n,m)$, the set of triangle-free graphs on $n$ vertices with $m$ edges, and the closely related model of $G(n,p)$ conditioned on the event $\mathcal T$ of being triangle-free. 

The study of triangle-free graphs (and more generally the study of graphs with a forbidden subgraph) has a long history in extremal, enumerative, and probabilistic combinatorics as well as in the study of non-linear large deviations in probability theory~\cite{chatterjee2017large}.  The story begins with perhaps the first result in extremal combinatorics, Mantel's Theorem~\cite{mantel1907problem}, that states that a balanced, complete bipartite graph has the most edges of any triangle-free graph on $n$ vertices.  Erd\H{o}s, Kleitman, and Rothschild~\cite{MR0463020} showed that this extremal construction is evident in the structure of a typical triangle-free graph: almost every triangle-free graph is bipartite.   This phenomenon continues to hold for much sparser triangle-free graphs: following~\cite{promel1996asymptotic,luczak2000triangle}, Osthus,  Pr{\"o}mel, and Taraz~\cite{osthus2003densities}
proved a sharp threshold for bipartiteness in a triangle-free graph:
\begin{theorem}[\cite{osthus2003densities}]
\label{thmOPT}
    Fix $\eps \in (0,1)$ and let $\mathbf G$ be drawn uniformly from $\mathcal T(n,m)$.  
    \begin{itemize}
        \item If $m \ge \frac{\sqrt{3} +\eps}{4}  n^{3/2} (\log n )^{1/2}$, then $\mathbf G$ is bipartite whp.
        \item  If $n/2 \le  m \le \frac{\sqrt{3} -\eps}{4}  n^{3/2} (\log n )^{1/2}$, then $\chi(\mathbf G) \ge 3$  whp.
    \end{itemize}
\end{theorem}
Note that a lower bound on $m$ in the second statement is necessary: when $m=o(n)$ then $\mathbf G$ is again bipartite whp (in fact it is a forest), and when $m = c n$ for constant $c \in (0, 1/2)$ then the probability of bipartiteness is bounded away from $0$ and $1$ (see, e.g.~\cite{bollobas1984evolutionrandom}).  An analogous statement for $G(n,p)$ conditioned on triangle-freeness holds around the threshold $p = \sqrt{3 \frac{\log n}{n}   } $.

Pr{\"o}mel, Steger, and Taraz~\cite{promel2001asymptotic} laid out a general program of understanding the `constrained evolution' of a random graph under a condition like triangle-freeness.
 Pr{\"o}mel and Taraz then asked the specific question~\cite{promel2001random}: for $n/2 \le m \le \frac{\sqrt{3}}{4} n^{3/2} \sqrt{\log n} $, what is the typical chromatic number of $\mathbf G$ drawn from $\mathcal T(n,m)$?

Jenssen, Perkins, and Potukuchi~\cite{jenssen2023evolution} established two more sharp thresholds for the chromatic number of random triangle-free graphs.
\begin{theorem}[\cite{jenssen2023evolution}]\label{thm:jpp-fixed-m}
Fix $0 < \eps < 1/14$.  Let $\mbf{G} \sim \mc{T}(n,m)$ be a uniform random sample.
\begin{itemize}
    \item If $(\sqrt2 + \eps) \frac14 n^{3/2} (\log n)^{1/2} \leq m \leq (\sqrt3 - \eps) \frac14 n^{3/2} (\log n)^{1/2}$ then $\chi(\mbf G) = 3$ whp.
    \item If $(1 + \eps) \frac14 n^{3/2} (\log n)^{1/2} \leq m \leq (\sqrt2 - \eps) \frac14 n^{3/2} (\log n)^{1/2}$ then $\chi(\mbf G) = 4$ whp.
    \item If $(\frac{13}{14} + \eps) \frac14 n^{3/2} (\log n)^{1/2} \leq m \leq (1-\eps) \frac14 n^{3/2} (\log n)^{1/2}$ then $\chi(\mbf G) \to \infty$ with $n$ whp.
\end{itemize}\end{theorem}
Here the constant $13/14$ in the last statement reflects limits of the methods of~\cite{jenssen2023evolution} rather than a phenomenon in the model itself.

Moreover, they showed that the scaling window for the sharp threshold for bipartiteness of \cref{thmOPT} is of width $\Theta(n^{3/2} (\log n)^{-1/2} )$.

Our main result is to establish the width of the next two scaling windows, from $\chi(\mbf{G}) = 3$ to $\chi(\mbf{G}) = 4$ and from $\chi(\mbf{G}) = 4$ to $\chi(\mbf{G}) = 5$.

To state the result for $3$-colorability, we first need to locate the center of the scaling window more accurately.  Recall the threshold in question occurs at $n^{3/2} \sqrt{\log n / 8}$ by \cref{thm:jpp-fixed-m}.  Denote by $\psi(n)$ the unique real number greater than 1 such that $\psi e^{-\psi} = 2/n$.  We may write this $\psi(n) = -W_{-1}(-2/n)$ for $W_{-1}$ the negative branch of the Lambert $W$ function (discussed in \cref{subsec:W-function}).  Note that $\psi(n) = (1+o(1)) \log n$.  Let
\[ m_c :=  n^{3/2} \sqrt{\frac{\psi(n)}{8}}. \]

This is the true center of the scaling window.  This is shifted from the more ``natural'' threshold $n^{3/2} \sqrt{\log n / 8}$ by terms of order $n^{3/2} \log \log n/\sqrt{\log n}$, which are significant with respect to the scaling window and must be included.

\begin{theorem}\label{theorem:main-fixed-m}
    Let $\mathbf G$ be drawn uniformly from $\mc{T}(n,m)$.  Then
    \begin{itemize}
	\item If $m = m_c + n^{4/3} (\log n)^{-1/3} \omega$ for $1 \ll \omega \ll n^{1/6} (\log n)^{5/6}$, then $\chi(\mbf G) = 3$ whp.
	\item If $m = m_c + n^{4/3} (\log n)^{-1/3} \omega$ for $\omega$ constant, then $\{ \chi(\mbf G) = 3 \}$ and $\{ \chi(\mbf G) = 4 \}$, each have probability bounded away from $0$ and $1$.
	\item If $m = m_c - n^{4/3} (\log n)^{-1/3} \omega$ for $1 \ll \omega \ll n^{1/6} (\log n)^{5/6}$, then $\chi(\mbf G) = 4$ whp.
    \end{itemize}
    In other words, the width of the scaling window is $\Theta \left( n^{4/3} (\log n)^{-1/3}   \right )$.
\end{theorem}

The center of the $4$- to $5$-colorability scaling window is also shifted from its ``natural'' center by terms of order $n^{3/2} \log\log n/\sqrt{\log n}$.  However, the scaling window is sufficiently coarse that only one additional term is relevant.  For this result, we are able to not only establish the width but compute the limiting probability inside the scaling window.

\begin{theorem}\label{thm:4-coloring-scaling-window-fixed-m}
    Fix $c \in (0,1)$ and let
    \[ m = \frac14 n^{3/2} \sqrt{\log n + \log\log n - 2\log(2c)}. \]

    Let $\mbf{G}$ be a uniform sample from $\mc{T}(n,m)$.  Then whp $4 \leq \chi(\mbf G) \leq 5$ and more precisely
    \[ \mb{P}[\chi(\mbf G) = 4] = \paren{\frac{1-c}{1+c}}^{1/2} \exp\{ -c - c^3/3 \} + o(1). \]

    Thus the width of the scaling window is $\Theta\left( \frac{n^{3/2}}{\sqrt{\log n}} \right)$.
\end{theorem}

 We also prove the analogs of both results for the Erd\H{o}s--R\'enyi binomial random graph conditioned on triangle-freeness.  Let
\[ p_c :=  \sqrt{\frac{2\psi(n)}{n}}. \]

\begin{theorem}\label{theorem:main-erdos-renyi}
    Let $\mbf G$ be distributed according to $G(n,p)$ conditioned on triangle-freeness.  Then
    \begin{itemize}
	\item If $p = p_c + n^{-2/3} (\log n)^{-1/3} \omega$ for $1 \ll \omega \ll n^{1/6} (\log n)^{5/6}$, then $\chi(\mbf{G}) = 3$ whp.
	\item If $p = p_c + n^{-2/3} (\log n)^{-1/3} \omega$ for $\omega$ constant, then $\chi(\mbf{G}) = 3$ and $\chi(\mbf{G}) = 4$, each with  probability bounded away from $0$ and $1$.
	\item If $p = p_c - n^{-2/3} (\log n)^{-1/3} \omega$ for $1 \ll \omega \ll n^{1/6} (\log n)^{5/6}$, then $\chi(\mbf{G}) = 4$ whp.
    \end{itemize}
    In other words, the width of the scaling window is $\Theta \left( n^{-2/3} (\log n)^{-1/3}   \right )$.
\end{theorem}

\begin{theorem}\label{thm:4-coloring-scaling-window}
    Fix $c \in (0,1)$ and let
    \[ p = \sqrt{\frac{\log n + \log\log n - 2\log(2c)}{n}}. \]

    Let $\mbf{G}$ be a sample from $G(n,p)$ conditioned on triangle-freeness.  Then whp $4 \leq \chi(\mbf G) \leq 5$ and more precisely
    \[ \mb{P}[\chi(\mbf G) = 4] = \paren{\frac{1-c}{1+c}}^{1/2} \exp\{ -c - c^3/3 \} + o(1). \]

    Thus the width of the scaling window is $\Theta \left ( \sqrt{\frac{1}{n \log n}}\right )$.
\end{theorem}

These theorems establish the widths of the respective scaling windows, but as in the examples described above, we also obtain an understanding of why the transition occurs.  Somewhat surprisingly, the transition from chromatic number $3$ to $4$ occurs due to a SAT to UNSAT transition in a random $2$-SAT formula implicitly encoded in the random triangle-free graph.  The transition from chromatic number $4$ to $5$ occurs due to the emergence of cycles in  binomial random graphs embedded (as `defect edges' in an otherwise bipartite graph) in $\mbf{G}$.  We make these notions precise in the following sections.  We remark that the widths of the scaling windows for $3$- and $4$-colorability in the usual random graphs $G(n,p)$ and $G(n,m)$ are unknown.

To explain the intuition for the width of the scaling window and how the transition occurs, we begin by summarizing some of the structural results on random triangle-free graphs from~\cite{luczak2000triangle,jenssen2023evolution}.

The starting point is a result of {\L}uczak, that can be thought of as a sparse, approximate  Erd\H{o}s--Kleitman--Rothschild result: typical triangle-free graphs at high enough density are almost bipartite. 
\begin{theorem}[\cite{luczak2000triangle}]
\label{thmLuczak{}}
    For every $\eps>0$ there is $C>0$ large enough so that when $m \ge C n^{3/2}$, almost every graph in $\mathcal T(n,m)$ has a cut containing at least  $(1-\eps)$-fraction of its edges.
\end{theorem}

The work~\cite{jenssen2023evolution} refines this result by determining very accurately the distribution of these `defect edges', edges inside the two parts of a max cut.

\begin{theorem}[\cite{jenssen2023evolution}]
    \label{thmJKPlight}
    Suppose $(\frac14+\eps) n^{3/2}(\log n)^{1/2} \le m \le \frac{1}{2} n^{3/2} (\log n)^{1/2}$ for some $\eps > 0$.  Then with $\mbf{G}$ drawn uniformly from $\mathcal T(n,m)$, the following hold whp:
    \begin{itemize}
        \item $\mbf{G}$ has a unique max cut $(A,B)$.
        \item Conditioned on the max cut being $(A,B)$, the edges in $A$ and $B$ are distributed as independent edges with probability $q$, up to $o(1)$ total variation distance, where $q$ is defined via
        \begin{equation}
            \frac{q}{1-q} = \lambda e^{-\lambda^2 n/2}
        \end{equation}
        and 
        $\lambda =(4+o(1))m/n^2 $, given precisely below in~\eqref{eq:def-of-lambda(m)}. 
        \item Conditioned on $(A,B)$ and the graphs $\mbf{G}[A]$, $\mbf{G}[B]$, the distribution of the crossing edges is that of an independent set model, defined precisely below in \cref{sec:JPP-summary}. 
    \end{itemize}
\end{theorem}

\subsection{\texorpdfstring{$3$}{3}- to \texorpdfstring{$4$}{4}-colorability}

Given this structural description of a triangle-free graph, we approach the question of the chromatic number when $m$ is near $\frac{\sqrt{2}}{4} n^{3/2} (\log n )^{1/2}$ by considered a special class of $3$-colorings we call `green-edge colorings'.  We describe a green-edge coloring of $\mbf{G}$ under the assumption that the graphs of defect edges, $\mbf{G}[A]$ and $\mbf{G}[B]$, both consist of isolated vertices and isolated edges, which holds whp (\cref{claim:disjoint-edges}).

\medskip

\textbf{Green-edge coloring of $\mbf{G}$}
\begin{enumerate}
    \item Color isolated vertices of $\mbf{G}[A]$ red.
    \item Color isolated vertices of $\mbf{G}[B]$ blue.
    \item Properly $2$-color edges of $\mbf{G}[A]$ green and red.
     \item Properly $2$-color edges of $\mbf{G}[B]$ green and blue.
\end{enumerate}
Since every edge has $2$ possible  $2$-colorings, the above description leaves room for many choices (which end point of each defect edge to color green).  

One step of our proof (\cref{cor:3-coloring-is-GEC}) is to show that whp in the relevant range of densities, $3$-colorable graphs must have green-edge colorings.

The next step is to determine if a valid green-edge coloring exists.  The only risk is of an edge crossing the partition $(A,B)$ joining two vertices colored green, so we must, if possible, choose which vertices to color green to avoid this.  We encode these constraints in a $2$-SAT formula.

\begin{enumerate}
    \item For each edge $e \in E(\mbf{G}[A])$, create a boolean variable $x_e$ and fix an (arbitrary) orientation $e = (u,v)$.  If $x_e = \TRUE$, set $\chi(u) = \text{red}, \chi(v) = \text{green}$.  Else set $\chi(u) = \text{green}, \chi(v) = \text{red}$.
    \item For each edge $e \in E(\mbf{G}[B])$, create a boolean variable $x_e$ and fix an (arbitrary) orientation $e = (u,v)$.  If $x_e = \TRUE$, set $\chi(u) = \text{blue}, \chi(v) = \text{green}$.  Else set $\chi(u) = \text{green}, \chi(v) = \text{blue}$.
    \item For each edge $f$ crossing the bipartition $A \cup B$, create a boolean disjunction.
\end{enumerate}

\begin{figure}
    \centering
    \begin{minipage}{0.45\textwidth}
        \centering
        \includegraphics[width=0.9\textwidth]{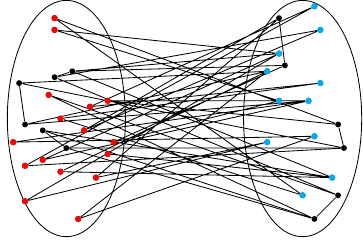} 
        \caption{\label{fig:coloring-isolated-vertices}An example of a graph $\mbf{G}$ separated along its max cut after steps (1) and (2) of the coloring.}
    \end{minipage}\hfill
    \begin{minipage}{0.45\textwidth}
        \centering
        \includegraphics[width=0.9\textwidth]{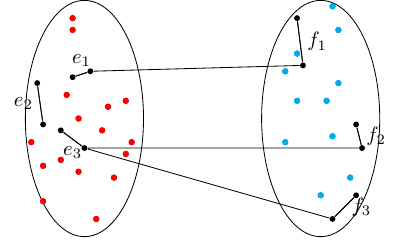} 
        \caption{\label{fig:coloring-defect-edges}With edges oriented from top to bottom, this gives the 2-SAT formula $(e_1 \vee \bar f_1) \wedge (\bar e_3 \vee \bar f_2) \wedge (\bar e_3 \vee \bar f_3)$.}
    \end{minipage}
\end{figure}

Each crossing edge gives a constraint that one boolean assignment of the edges containing its two endpoints is illegal.  This can be written as a disjunction.  The possibility of the properness of $\chi$ depends on the satisfiability of the conjunction of all of these disjunctions.

The remaining obstruction to determining the threshold is understanding the distribution of the resulting random boolean formula.  The formula is a bipartite 2-SAT formula, meaning we have two sets of variables $\{x_i\}$ and $\{y_i\}$ (corresponding to edges in $\mbf{G}[A]$ and $\mbf{G}[B]$), and all clauses must contain exactly one literal from each part.

\begin{theorem}\label{thm:coloring-equiv-to-SAT}
    Let $\mbf{G}$ be distributed according to $G(n,p)$ conditioned on triangle-freeness where $p = (1+o(1)) p_c$.  Let $\varphi$ be a random bipartite 2-SAT formula with $n^2 p \exp(-np^2/2)/8$ variables in each partition and each bipartite clause occurring independently with probability $p$.  Then
    \[ \mb{P}[\chi(\mbf{G}) = 3] = \mb{P}[\SAT(\varphi)] + o(1). \]
\end{theorem}

We prove this result via an explicit coupling between the two formulas.  Unfortunately we cannot hope for an exact coupling: the 2-SAT instance encoded in $\mbf{G}$ cannot have a pair of clauses of the form $(x \vee y) \wedge (\bar x \vee y)$ as this would correspond to a triangle in $\mbf{G}$ (both endpoints of a defect edge sharing a neighbor), and $\mbf{G}$ is conditioned to be triangle-free.  However, at the critical density, $\varphi$ has Poisson many such clauses.  We instead show that $\mbf{G}$ and $\varphi$ can be coupled in such a way to agree off these clauses, and that there are few enough of these as to not affect satisfiability.

$3$-colorability has now been completely reduced to an instance of random bipartite $2$-SAT, which we handle separately; see \cref{subsec:intro-2-SAT}.  Note that the question of $\mb{P}[\SAT(\varphi)]$ inside the critical window is open; a formula for this quantity would allow us to exactly compute $\mb{P}[\chi(\mbf{G}) = 3]$ inside the critical window.

\subsection{\texorpdfstring{$4$}{4}- to \texorpdfstring{$5$}{5}-colorability}

The obstruction to $4$-coloring a graph is easier to see than the obstruction to $3$-coloring.
Clearly, if $\mbf{G}[A]$ and $\mbf{G}[B]$ are both bipartite, then $\mbf{G}$ may be $4$-colored by assigning two distinct colors to each side.
In fact, due to expansion properties of $\mbf{G}$ that hold whp, we can show that this is the approximate structure of any $4$-coloring of $\mbf{G}$: two color classes must consist of $\Omega(n)$ vertices in $A$ and $o(n)$ vertices in $B$, and the other two vice versa.  Call a vertex $A$-colored (resp.~$B$-colored) if its color contains $\Omega(n)$ vertices in $A$ (resp.~$B$).  Call a vertex $v \in A$ (resp.~$v \in B$) \emph{miscolored} if $v$ is $B$-colored (resp.~$A$-colored).

\begin{proposition}
    Let $w \in V$.  Whp, there is no $4$-coloring of $\mbf{G}$ where $w$ is miscolored.
\end{proposition}

For a heuristic on why this should hold, consider the following procedure.  It is a simplification of \cref{def:mu-lambda-2}, which approximates the true distribution up to $o(1)$ total variation distance.

\begin{enumerate}
    \item Choose $A \in \binom{n}{n/2}$ uniformly at random and let $B = [n] \setminus A$.
    \item Sample $S \sim G(A,q)$ and $T \sim G(B,q)$ where $q = \Theta(1/n)$ is the solution to
    \[ \frac{q}{1-q} = \lambda e^{-\lambda^2 n/2} \]
    and $\lambda = 4m/n^2$.
    \item Sample $E_\mr{cr} \subset A \times B$ by including each element independently with probability $\lambda = 4m/n^2 = \Theta(\sqrt{\log n/n})$.
    \item Output $S \cup T \cup E_\mr{cr}$.
\end{enumerate}

Let $w \in A$.  In expectation, $N_B(w)$ contains $\Theta(\log n)$ many edges.  As $w$ is miscolored and so $B$-colored, these edges must each contain an $A$-colored vertex.  Thus one miscoloring propagates into logarithmically many miscolorings.  Their neighborhoods also contain $\Theta(\log n)$ many edges, and this propagation continues until each side has linearly many miscolored vertices, contradicting the structure of a $4$-coloring.

However, since $\mbf{G}$ is triangle-free, $N_B(w)$ cannot contain edges.  The true distribution of $E_\mr{cr}$ is an independent set model defined in \cref{sec:preliminaries}.  Instead, we count pairs $(u,v) \in N_B(w)$ connected by some path of length 3, of which we can locate $\Theta(\log n)$ many.  Each such pair results in a copy of $C_5$ consisting of $w$ and four vertices in $B$.  Since $w$ is $B$-colored and $\chi(C_5) = 3$, some vertex on this path must be $A$-colored and so miscolored.  Once again, one miscoloring propagates into logarithmically many, eventually resulting in linearly many miscolored vertices and a contradiction. 

Thus the question of $4$-colorability of $\mbf{G}$ reduces to the question of $2$-colorability of both $\mbf{G}[A]$ and $\mbf{G}[B]$, which are close in distribution to independent Erd\H{o}s--R\'enyi random graphs (\cref{prop:sandwiching-ERG-between-ERs}).  Understanding the bipartite to non-bipartite threshold in such graphs can be done following~\cite{pittel1988random}. 

\subsection{Bipartite \texorpdfstring{$2$}{2}-SAT}\label{subsec:intro-2-SAT}
To complete the analysis of the $3$- to $4$-colorability threshold, we need the scaling window of bipartite random $2$-SAT.  Non-bipartite random 2-SAT is a well-studied random constraint satisfaction problem (CSP) in theoretical computer science.  The sharp threshold for satisfiability of random $2$-SAT was established in the 1990's~\cite{goerdt1996threshold,chvatal1992mick,de2001random}.  Results of Goerdt~\cite{goerdt1999remark} and Verhoeven~\cite{verhoeven1999random} made progress on the size of the scaling window before it was fully determined by Bollob\'as, Borgs, Chayes, Kim, and Wilson~\cite{bollobas2001scaling}.

\begin{theorem}[{\cite[Theorem 1.1]{bollobas2001scaling}}]
    There are constants $\kappa_0$ large and $\eps_0$ small such that the following holds.  Let $F_{N,q}$ denote a random 2-SAT formula on $N$ variables with each clause occurring independently with probability
    \[ q := \frac{1+\eps}{2N} = \frac{1 + \kappa N^{-1/3}}{2N}, \]
    where $\kappa_0 < |\kappa|$ and $|\eps| < \eps_0$ (i.e.~$1 \ll |\kappa| \ll N^{1/3}$).  Then as $N \to \infty$,
    \[ \mb{P}(\SAT(F_{N,q})) = \begin{cases} \exp[-\Theta(|\kappa|^{-3})] & \kappa < 0, \\ \exp[-\Theta(|\kappa|^3)] & \kappa > 0. \end{cases} \]
\end{theorem}

Our result mirrors theirs in the bipartite setting.  

\begin{theorem}[bipartite 2-SAT]\label{theorem:bipartite-2SAT}
    For all $C > 0$, there exist $\kappa_0$ large and $\eps_0$ small such that the following holds.  Let $N,M$ be two integers with $N \to \infty$, $|N - M| \leq C N^{2/3}$.  Suppose we have $N+M$ variables $\{x_1,\ldots,x_N\}$ and $\{y_1,\ldots,y_M\}$.  Let $F_{N,M,q}$ be a random bipartite 2-SAT formula on $\{x_i\} \cup \{y_j\}$ with each bipartite clause occurring independently with probability
    \[ q := \frac{1 + \eps}{2(MN)^{1/2}} = \frac{1 + \kappa (MN)^{-1/6}}{2(MN)^{1/2}}, \]
    where $\kappa_0 < |\kappa|$ and $|\eps| < \eps_0$ (i.e.~$1 \ll |\kappa| \ll N^{1/3}$).  Then as $n \to \infty$,
    \[ \mb{P}(\SAT(F_{N,M,q})) = \begin{cases} \exp[-\Theta(|\kappa|^{-3})] & \kappa < 0, \\ \exp[-\Theta(|\kappa|^3)] & \kappa > 0. \end{cases} \]
\end{theorem}

Our proof of this theorem closely follows that of Bollob\'as, Borgs, Chayes, Kim, and Wilson~\cite{bollobas2001scaling} for random $2$-SAT, though we must keep track of more indices to handle the bipartiteness.  The lower bound is proven using a first- and second-moment computation.  The upper bound is proven using  structures known as ``hourglasses.''

  The biggest challenge to carrying through this proof strategy is a lack of enumerative results.  Their methods require control on $C(k,s)$ the number of connected graphs with $k$ vertices and excess $s$ (i.e.~$k-1+s$ edges).  In a series of papers, Wright~\cite{wright1977number,wright1980number} determined the asymptotics of $C(k,s)$ for any $s = o(k^{1/3})$ up to the leading-order coefficient using involved generating function arguments.  This was extended by Bender, Canfield, and McKay~\cite{bender1990asymptotic} to an accurate formula for any $s = s(k) \leq \binom k2 - (k-1)$.

However, the bipartite version $C(k,\ell,s)$ the number of connected (spanning) subgraphs of $K_{k,\ell}$ with excess $s$ is much less well-studied.  A recent paper of Clancy~\cite{clancy2024asymptotics} matches the result of Wright for $s = O(1)$, but does not hold for $s$ a function of $k,\ell$.  We instead use an upper bound on $C(k,\ell,s)$ by Do, Erde, Kang, and Missethan~\cite{do2021component} and prove a lower bound (\cref{sec:proof-of-lower-bound-bipartite}).  However, while $C(k,s)$ is known up to $(1+o(1))$ accuracy, these upper and lower bounds on $C(k,\ell,s)$ only agree up to $\exp(\Theta(s))$ factors, and so additional arguments are required to show that these estimates suffice.

Note that we do not have enough precision to determine  the limiting value of $\mb{P}(\SAT(\cdot))$ inside the critical window.  This remains open for the standard (non-bipartite) 2-SAT problem as well.  A formula for the correct limiting probability in the non-bipartite case was conjectured by Dovgal, de Panafieu, and Ravelomanana~\cite{dovgal2021exact}.

\subsection{Outline of paper}

In \cref{sec:preliminaries}, we gather some preliminary results for the first part of the paper.  In \cref{sec:structural-results}, we show that any $3$-coloring will be found by our algorithm whp.  
In \cref{sec:reduction-to-SAT}, we relate success probability of our algorithm to satisfiability of a random bipartite $2$-SAT formula.
In \cref{sec:4-coloring-scaling-window}, we prove \cref{thm:4-coloring-scaling-window} on $4$-colorability.  In \cref{sec:T(nm)}, we derive \cref{theorem:main-fixed-m,thm:4-coloring-scaling-window-fixed-m} on $\mc{T}(n,m)$ from \cref{theorem:main-erdos-renyi,thm:4-coloring-scaling-window} on binomial random graphs.

In \cref{sec:2-SAT-intro}, we establish some standard terminology and lemmas for the study of $2$-SAT formulas.  In \cref{sec:SAT-lower-bounds}, we prove the lower bounds in \cref{theorem:bipartite-2SAT}.  In \cref{sec:SAT-upper-bounds}, we prove the upper bounds in \cref{theorem:bipartite-2SAT}.

\section{Preliminaries for triangle-free graphs}\label{sec:preliminaries}

\subsection{Basic graph theory}

A graph is a pair $G = (V,E)$ where $E \subset \binom V2$.  We will write $V(G)$ and $E(G)$ to denote the vertex and edge sets respectively.  A set $I \subset V$ is called an independent set if $e \not\subset I$ for any $e \in E$.  We denote by $\mc{I}(G)$ the set of all independent sets in $G$.

A proper $k$-coloring of $G$ is a map $\chi : V(G) \to [k]$ such that $\chi^{-1}(i) \in \mc{I}(G)$ for all $i \in [k]$.  We say $G$ is $k$-colorable if there exists a proper $k$-coloring.  A graph is said to be bipartite if it is $2$-colorable.

Given two graphs $G$, $H$, the Cartesian product $G \on{\Box} H$ is defined by $V(G \on{\Box} H) = V(G) \times V(H)$ and
\begin{multline*} E(G \on{\Box} H) = \left\{ \{ (u,v),(u',v) \} : uu' \in E(G), v \in V(H) \right\} \\ \cup \left\{ \{ (u,v),(u,v') \} : u \in V(G), vv' \in E(H) \right\}. \end{multline*}

\subsection{Hard-core model}

The hard-core model is a one-parameter probability measure on the set of independent sets of a graph $G$.  Given a graph $G$ and a parameter $\lambda > 0$ called the activity or fugacity, we define the hard-core model $\mu$ by
\[ \mu(I) = \frac{1}{Z_G(\lambda)} \lambda^{|I|}, \hspace{30pt} Z_G(\lambda) := \sum_{I \in \mc{I}(G)} \lambda^{|I|}. \]

The function $Z_G(\lambda)$ is called the partition function.  The hard-core model satisfies some quasirandomness conditions.  Intuitively, for sufficiently small $\lambda$, the hard-core model exhibits similar behavior to independent random sampling of vertices with probability $\lambda/(1+\lambda)$.  We will use the following lemma for this notion.

\begin{lemma}[quasirandomness of the hard-core model, {\cite[Lemma 4.3]{jenssen2023evolution}}]\label{lem:hardcore-quasirandom}
    Let $G$ be a graph of maximum degree $\Delta$ and let $U \subset V(G)$.  Let $\lambda \leq \frac{1}{16e^2\Delta}$ and let $\mathbf{I}$ be a random sample from the hard-core model on $G$ at activity $\lambda$.  Then
    \[ \mb{P}(|\mathbf{I} \cap U| \geq 5\lambda |U|) \leq e^{-\lambda |U|}, \]
    and
    \[ \mb{P}(|\mathbf{I} \cap U| \leq \lambda |U|/10) \leq e^{-\lambda |U|/8}. \]
\end{lemma}

Note that this is neither optimal in range of $\lambda$ or constants, but is strong enough for our purposes.  

We will also need the following concentration result.

\begin{proposition}\label{prop:hardcore-size-concentration}
    Let $G$ be a triangle-free graph with $n$ vertices and maximum degree $\Delta$.  Let $\xi \leq 1/(4e\Delta)$ and let $\mbf{I}_\xi$ be a sample from the hard-core model on $G$ at fugacity $\xi$.  Then
    \[ \mb{E}|\mbf{I}_\xi| = \frac{\xi}{1+\xi} n - 2 e(G) \xi^2 + 3(P_2(G) + 2 e(G)) \xi^3 + O(n \Delta^3 \xi^4) \]
    and
    \[ \on{Var}|\mbf{I}_\xi| = \frac{\xi}{(1+\xi)^2} n - 4 e(G) \xi^2 + O(n \Delta^2 \xi^3). \]
\end{proposition}

\subsection{Lambert \texorpdfstring{$W$}{W} function}\label{subsec:W-function}

The Lambert $W$ function is the (multivalued) inverse of $x \mapsto x e^x$.  We will only be interested in the behavior over $x \in \mb{R}$, where there are two branches.  The upper branch is generally called $W_0 : [-1/e,\infty) \to [-1,\infty)$ while the lower branch is often denoted $W_{-1} : [-1/e,0) \to (-\infty,-1]$.

\begin{figure}
    \centering
    \includegraphics[width=0.5\linewidth]{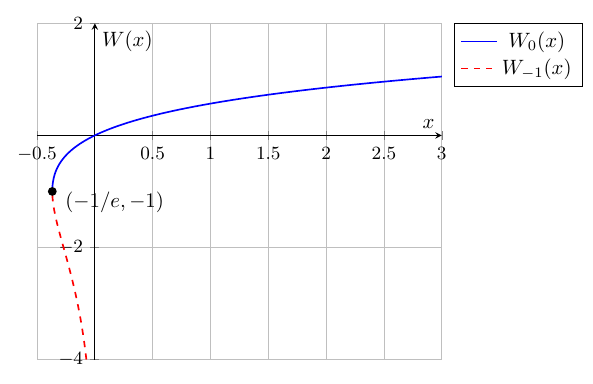}
    \caption{Real branches of the Lambert $W$ function}
\end{figure}

The function we will be interested in is the inverse of $2/y = x e^{-x}$, which, for all $y > 0$, has a unique solution in $[1,\infty)$, given by $x = -W_{-1}(-2/y)$; we will denote this $x = \psi(y)$ for brevity.  The function $\psi$ is continuous, increasing, differentiable, and has asymptotics
\[ \psi(y) = \log (y/2) + \log\log y + O\paren{\frac{\log\log y}{\log y}} \]
as $y \to \infty$.  See, e.g.~\cite{corless1996lambert} for more discussion of this function and proofs of these facts.

\subsection{Triangle-free graphs}\label{sec:JPP-summary}

Here we review the work of Jenssen, Perkins, and Potukuchi~\cite{jenssen2023evolution} on which we build.  Let $\mu_\lambda$ be the measure
\[ \mu_\lambda(G) = \frac{\lambda^{|G|}}{Z(\lambda)}, \qquad Z(\lambda) = \sum_{G \in \mc{T}(n)} \lambda^{|G|}. \]

Note that $\mu_\lambda$ is the distribution of $G(n,\lambda/(1+\lambda))$ conditioned on triangle-freeness.  

\begin{definition}[{\cite[Algorithm 2]{jenssen2023evolution}}]\label{def:mu-lambda-1}

Given $n$ and $\lambda$, define the measure $\mu_{\lambda,1}$ to be the law of the output of the following procedure for generating a random graph $\mbf{G}$ on $[n]$.

\begin{enumerate}
    \item Choose $\zeta \in \mb{Z}$ via
    \[ \mb{P}[\zeta = t] \propto (1+\lambda)^{-t^2}. \]
    \item Choose $A \in \binom{[n]}{\lfloor n/2 \rfloor - \zeta}$ uniformly at random, and let $B = [n] \setminus A$.  If $\zeta > \lfloor n/2 \rfloor$ or $\zeta < -\lceil n/2 \rceil$, abort and output the empty graph.
    \item Choose $S \subset \binom A2$ and $T \subset \binom B2$ according to independent Erd\H{o}s--R\'enyi random graphs on $A$ and $B$ with edge probability $q_0$ the unique solution in $(0,1)$ to
    \[ \frac{q_0}{1 - q_0} = \lambda e^{-\lambda^2n/2}. \]

    If $\mbf{G}[A]$ or $\mbf{G}[B]$ contains a triangle, then abort and output the empty graph.
    \item Choose $E_\mr{cr} \subset A \times B$ according to the hard-core model on $S \on{\Box} T$ at fugacity $\lambda$.
    \item Output $S \cup T \cup E_\mr{cr}$.
\end{enumerate}
\end{definition}

Note that selecting $E_\mr{cr}$ according to the hard-core model forbids the creation of triangles.

\begin{theorem}[{\cite[Theorem 2.1]{jenssen2023evolution}}]\label{thm:JPP-full-strength}
    Fix $\eps > 0$ and suppose $p \geq (1+\eps) \sqrt{\frac{\log n}{n}}$.  Then, with $\lambda = p/(1-p)$, we have
    \[ \| \mu_\lambda - \mu_{\lambda,1}\|_\mr{TV} = o(1), \]
    where $\mu_\lambda$ denotes the distribution of a sample from $G(n,p)$ conditioned on triangle-freeness.
\end{theorem}

This characterization implies several niceness conditions for $\mbf{G}$.

\begin{lemma}[\cite{jenssen2023evolution}]\label{lem:strongly-balanced}
     $\mbf G$ has a unique max cut $V(\mbf{G}) = A \cup B$ with $|A| = n/2 + O((n\log n)^{1/4})$ and $|B| = n/2 + O((n \log n)^{1/4})$ whp.
\end{lemma}

\begin{lemma}[\cite{jenssen2023evolution}]\label{lem:expander}
    Let $\mbf{G} \sim \mu_{\lambda,1}$ and let $(A,B)$ be the (unique) max cut of $\mbf{G}$.  Then whp, for all $v \in A$,
    \[ \frac{\lambda n}{30} \leq d_\mbf{G}(v,B) \leq 5 \lambda n \]
    and
    \[ X \subset A, \ |X| \geq \lambda n/100, \ Y \subset B, \ |Y| \geq n/9 \implies d_\mbf{G}(X,Y) > 0, \]
    and likewise with $A$ and $B$ exchanged.
\end{lemma}

\cref{lem:expander} is a minor modification of \cite[Lemma 5.3]{jenssen2023evolution}, where we have improved some constants and added the upper bound on degree.  We prove this version in \cref{app:expander}.

Throughout this paper, we will work with $\mbf{G}$ a sample from $\mu_{\lambda,1}$ and apply Theorem~\ref{thm:JPP-full-strength} to get our desired results on the random triangle-free graph.

\section{Structural results at critical \texorpdfstring{$\lambda$}{lambda}}\label{sec:structural-results}

We begin with several results about the typical $\mbf{G} \sim \mu_\lambda$ in the range of $\lambda$ we are interested in.

\begin{assumption}\label{assumption:3-to-4}
    Fix
\[ p_c := \sqrt{\frac{2\psi(n)}{n}} = (1+o(1)) \sqrt{\frac{2\log n}{n}} \]
and let
\[ p = p_c + n^{-2/3} (\log n)^{-1/3} \omega, \qquad |\omega| \ll n^{1/6} (\log n)^{5/6}. \]

\begin{itemize}
    \item Let $\lambda = p/(1-p)$ and let $q_0$ be the unique solution in $(0,1)$ to
    \[ \frac{q_0}{1 - q_0} = \lambda e^{-n \lambda^2/2}. \]
    \item $\mbf{G}$ is a random graph $\mbf{G} \sim \mu_{\lambda,1}$ on $n$ vertices (see \cref{def:mu-lambda-1}).
    \item $\mbf{G}$ has a unique max cut $V(\mbf{G}) = A \cup B$.  (This exists whp by \cref{lem:strongly-balanced}.)
    \item $S$ and $T$ are the edge sets of $\mbf{G}[A], \mbf{G}[B]$ respectively.  We will call $\mbf{G}[A]$ and $\mbf{G}[B]$ the defect graphs.
\end{itemize}
\end{assumption}

This will be our setting for \cref{sec:structural-results,sec:reduction-to-SAT}.  We may compute
\begin{align*}
    \frac{q_0}{1-q_0} & = \lambda e^{-\lambda^2n/2} = (1+\widetilde{O}(n^{-1/2})) p \exp\left\{-\left( \psi(n) + \widetilde{O}(n^{-1/6}) \right) \right\} \\
    & = (1+o(1)) \paren{\frac{2\log n}{n}}^{1/2} \exp\left\{ -\left( \log n + \log\log n + o(1) \right) \right\} \\
    & = (1+o(1)) \frac{\sqrt2}{n^{3/2} (\log n)^{1/2}} \ll 1 \\
    q_0 & = (1+o(1))\frac{\sqrt2}{n^{3/2} (\log n)^{1/2}}.
    \addtocounter{equation}{1}\tag{\theequation}\label{eq:q_0-leading-order}
\end{align*}

\begin{lemma}\label{lem:concentration-of-S}
    Let $\mbf{G}$ follow \cref{assumption:3-to-4}. Whp,
    \[ |S| = q_0 \frac{n^2}{8} + O(n^{1/4}) = \Theta\paren{n^{1/2} (\log n)^{-1/2}} \]
    and
    \[ |T| = q_0 \frac{n^2}{8} + O(n^{1/4}) = \Theta\paren{n^{1/2} (\log n)^{-1/2}}. \]
\end{lemma}

\begin{proof}
    By \cref{lem:strongly-balanced}, we may assume $|A| = n/2 + \widetilde O(n^{1/4})$ except with vanishing probability.  Conditioning in $|A|$, we have $|S| \sim \on{Binom}(\binom{|A|}{2}, q_0)$.  By a Chernoff bound,
    \[ \mb{P}\left[ \left| |S| - q_0 \binom{|A|}{2} \right| \geq n^{1/4} \right] \leq \exp\left\{-\frac12 n^{1/2} \left(q_0 \binom{|A|}{2}\right)^{-1} \right\} = \exp\{-\Omega((\log n)^{1/2}) \}. \]

    We can now evaluate
    \[ q_0 \binom{|A|}{2} = q_0 \paren{\frac{n^2}{8} + \widetilde{O}(n^{5/4})} = q_0 \frac{n^2}{8} + \widetilde{O}(n^{-1/4}), \]
    and so if $|A| = n/2 + \widetilde O(n^{1/4})$ then $|S| = q_0 n^2/8 + O(n^{1/4})$ except with probability $o(1)$.
\end{proof}

\begin{lemma}\label{claim:disjoint-edges}
    Let $\mbf{G}$ follow \cref{assumption:3-to-4}. Whp, $\mbf{G}[A]$ and $\mbf{G}[B]$ have no connected components of size $>2$.
\end{lemma}

    \begin{proof}
        This is a first moment calculation.  Let $X$ be the number of (not necessarily induced) paths of length 2 in $S \cup T$.  Then using \cref{eq:q_0-leading-order} and \cref{lem:strongly-balanced},
        \begin{align*}
            \mb{E}[X] & = \frac{1}{|\operatorname{Aut}(P_2)|} \left(\sum_{(u,v,w) \in A} q_0^2 + \sum_{(u,v,w) \in B} q_0^2 \right) \leq \frac12(1+o(1))n^3q_0^2 = (1+o(1)) \frac{1}{\log n}.
        \end{align*}

        Thus whp, $\mbf{G}[A] \cup \mbf{G}[B]$ has no copies of $P_2$ and so no connected components of size $>2$.
    \end{proof}

Thus \cref{lem:strongly-balanced,lem:expander,lem:concentration-of-S,claim:disjoint-edges} combine to imply that \cref{assumption:3-to-4} holds whp for $\mbf{G} \sim \mu_{\lambda,1}$, and \cref{thm:JPP-full-strength} implies that \cref{assumption:3-to-4} holds whp for $\mbf{G} \sim \mu_\lambda$.

\subsection{Reduction from general \texorpdfstring{$3$}{3}-colorability to ``green-edge coloring''}\label{sec:GEC-proof}

In this section, we demonstrate that under \cref{assumption:3-to-4}, whp the existence of any $3$-coloring of $\mbf{G}$ yields a $3$-coloring with a particular structure, which we described informally above as a green-edge coloring and formalize the definition below. 
Since the converse implication is immediate, this establishes the high probability equivalence between $\mbf{G}$ being $3$-colorable and $\mbf{G}$ admitting a green-edge coloring.

\begin{definition}\label{def:green_edge_coloring}
    Let $\sigma : V(\mbf{G}) \to \{\text{red, green, blue}\}$ be a proper coloring where $\mbf{G}$ has (unique) max cut $(A,B)$.  Then $\sigma$ is a \emph{green-edge coloring} if all isolated vertices of $\mbf{G}[A]$ are red, all isolated vertices of $\mbf{G}[B]$ are blue, no vertices of $A$ are blue, and no vertices of $B$ are red.
\end{definition}

We will show that whp in the given parameter regime, if a graph fails to be green-edge colorable, then it fails to be 3-colorable.  First, we show the following lemma.

\begin{lemma}\label{lem:no-bad-3-coloring}
    Let $\mbf{G}$ follow \cref{assumption:3-to-4}.  Whp, there is no 3-coloring of $\mbf{G}$ in which one part receives all three colors.
\end{lemma}

This will immediately give us the desired reduction.

\begin{corollary}\label{cor:3-coloring-is-GEC}
    For $\mbf{G}$ following \cref{assumption:3-to-4}, whp, if $\chi(\mbf{G}) = 3$, then $\mbf{G}$ has a green-edge coloring.
\end{corollary}

\begin{proof}
    Since $\chi(\mbf{G}) \leq 3$, $\mbf{G}$ has a proper $3$-coloring $\sigma : V(\mbf{G}) \to \{\text{red, blue, green}\}$.  By \cref{lem:no-bad-3-coloring}, whp, $\sigma$ cannot assign either side all three colors.  Since $\chi(\mbf{G}) > 2$, all three colors appear, so we do not use the same two on both sides.  Without loss of generality, $\sigma$ assigns red and green to $A$ and blue and green to $B$.  Define the coloring
    \[ \widetilde{\sigma} : \mbf{G} \to \{\text{red, blue, green}\}, \ \ \ \ \ \ \widetilde{\sigma}(v) = \begin{cases} \text{red} & v \in A \setminus \bigcup_{e \in S} e \\ \text{blue} & v \in B \setminus \bigcup_{e \in T} e \\ \sigma(v) & v \in \bigcup_{e \in S} e \cup \bigcup_{e' \in T} e'. \end{cases} \]

    Observe that $\widetilde{\sigma}$ is also a proper coloring.  The only changes are that vertices isolated in $\mbf{G}[A]$ have been changed to red (which cannot cause problems inside $\mbf{G}[A]$ since they are isolated or in the crossing edges since $B$ contains no red) and vertices isolated in $\mbf{G}[B]$ have been changed to blue.  Thus we have explicitly constructed $\widetilde\sigma$, a green-edge coloring of $\mbf{G}$.
\end{proof}

\begin{proof}[Proof of \cref{lem:no-bad-3-coloring}]
    We first establish some notation.  For $v \in V$, we write $N_B(v) := N_\mbf{G}(v) \cap B$ and similarly $N_A(v) := N_\mbf{G}(v) \cap A$.  We will use $\bar{S} := \bigcup_{e \in S} e$ to refer to the vertices contained in defect edges $S$ and similarly $\bar{T} := \bigcup_{e \in T} e$.  Define the color classes $A = \mf{r}_A \cup \mf{b}_A \cup \mf{g}_A$ and $B = \mf{r}_B \cup \mf{b}_B \cup \mf{g}_B$, where some parts may be empty.  Without loss of generality, suppose $A$ contains all three colors, i.e.~$\mf{r}_A,\mf{b}_A,\mf{g}_A$ are nonempty.

    The part $B$ cannot be monochromatic, as $A$ has vertices of all three colors and for all $v \in A$, we have $d_\mbf{G}(v,B) \geq \lambda n/30 > 0$.

    Suppose $B$ has only two colors.  WLOG let $B = \mf{b}_B \cup \mf{g}_B$.  Recall $S = \binom{A}{2} \cap E(\mbf{G})$ is the set of defect edges in $A$.  
    By \cref{lem:concentration-of-S}, whp, $\frac18 \sqrt{n/\log n} \leq |S| \leq 4\sqrt{n/\log n}$.  Note that each edge in $S$ must contain at least one non-red vertex.  
    By \cref{claim:disjoint-edges}, whp, all edges are disjoint.  
    By the pigeonhole principle, one of $\mf{b}_A$ and $\mf{g}_A$ (WLOG $\mf{g}_A$) contains at least one-quarter of the vertices in $\bar{S}$.  
    Thus there is $U \subset \bar S$ with $|U| = |\bar{S}|/4$ and $U \subset \mf{g}_A$.

    By assumption, there is some $v \in \mf{b}_A$.  Since this is a proper coloring and $\mf{r}_B$ is empty, $N_B(v) \subset \mf{g}_B$.  We claim that whp, there are edges between $\mf{g}_A$ and $\mf{g}_B$.  Crucially, we know something about not just the size but also the structure of $\mf{g}_A$ and $\mf{g}_B$ so we can union bound over a smaller set.
    \begin{align*}
        \Pr[d_\mbf{G}(\mf{g}_A,\mf{g}_B) = 0] & \leq \Pr\left[ \bigvee_{v \in A} \bigvee_{U \in \binom{\bar S}{|\bar S|/4}} d_\mbf{G}(N_B(v),U) = 0 \right] \\ 
        & \leq |A| \cdot \binom{|\bar S|}{|\bar S|/4} \cdot \exp\left\{ - \frac{\lambda}{8} \cdot \frac{\lambda n}{30} \cdot \frac{|\bar S|}{4}  \right\} & \text{\cref{lem:hardcore-quasirandom}} \\
        & \leq \frac{2n}{3} \cdot (4e)^{\sqrt{n/\log n}} \cdot \exp\left\{ -\frac{(1+o(1))\log n}{120} \cdot \frac{\sqrt n}{32 \sqrt{\log n}} \right\} & \text{\cref{lem:concentration-of-S}} \\
        & = \exp\left\{-\Omega(\sqrt{n \log n})\right\}.
    \end{align*}

    The second inequality is a union bound and the hard-core quasirandomness property \cref{lem:hardcore-quasirandom}.  The third is the concentration of $|S|$ by \cref{lem:concentration-of-S}.  Thus whp, there is no way to avoid monochromatic edges if $B$ only uses two colors.
    
    Finally, suppose part $B$ also contains all three colors $B = \mf{r}_B \cup \mf{b}_B \cup \mf{g}_B$ with all nonempty.  By the pigeonhole principle, one color contains at least $1/3$ of the vertices in $A$, so without loss of generality, $|\mf{r}_A| \geq |A|/3 \geq n/9$ since the partition is weakly balanced.  By \cref{lem:expander}, this means $|\mf{r}_B| < \lambda n/100$ as $d_\mbf{G}(\mf{r}_A,\mf{r}_B) = 0$.  Thus $|\mf{r}_B| \ll |B|/3$, so without loss of generality, $|\mf{b}_B| \geq |B|/3 \geq n/9$.  By the expansion property, this means $|\mf{b}_A| \leq \lambda n/100$.

    However, by assumption, $\mf{r}_B$ is not empty, so there is some $v \in \mf{r}_B$.  Since we have a proper coloring, $N_A(v) \subset \mf{b}_A \cup \mf{g}_A$.  By the first expansion property, $d_\mbf{G}(v,B) \geq \lambda n/30$, but $|\mf{b}_A| \leq \lambda n/100$, so $|\mf{g}_A| \geq \lambda n/30 - \lambda n/100 > \lambda n/50$.  By the same reasoning, we have $w \in \mf{b}_A$ with at least $\lambda n/50$ neighbors in $\mf{g}_B$, so $|\mf{g}_A|,|\mf{g}_B| \geq \lambda n/50$.  We now do a union bound here.
    \begin{align*}
        \Pr[d_\mbf{G}(\mf{g}_A,\mf{g}_B = 0)] & \leq \Pr\left[ \bigvee_{v \in A} \bigvee_{w \in B} \text{no edges between some pair of large subsets of }N_B(v), N_A(w) \right] \\ & \leq |A| \cdot |B| \cdot \binom{\max_{v \in A} d_\mbf{G}(v,B)}{\lambda n/50} \cdot \binom{\max_{w \in B} d_\mbf{G}(w,A)}{\lambda n/50} \cdot \exp\left\{ -\frac{\lambda}{8} \cdot \paren{\frac{\lambda n}{50}}^2 \right\} \\
        & \leq \frac{n^2}{4} \cdot \binom{5 \lambda n}{\lambda n/50} \cdot \binom{5 \lambda n}{\lambda n/50} \cdot \exp\left\{ - \Omega(\sqrt{n(\log n)^3}) \right\} \\
        & = \exp\left\{ - \Omega(\sqrt{n(\log n)^3}) \right\}.
    \end{align*}

    Hence whp there is no way to avoid crossing edges between $\mf{g}_A$ and $\mf{g}_B$.

    By a union bound over these three cases, whp $\mbf{G}$ has no coloring where either $A$ or $B$ receive all three colors.
\end{proof}

\section{Reduction to 2-SAT}\label{sec:reduction-to-SAT}

\subsection{Green-edge coloring to satisfiability}
Given $\mbf{G}$ following \cref{assumption:3-to-4}, let the event $\on{GEC}(\mbf{G})$ refer to the existence of a green-edge coloring of $G$.  We will control $\on{GEC}(G)$ by attempting to create a green-edge coloring.
First, recall that all isolated vertices  in $\mbf{G}_A$ are colored red and all isolated vertices  in $\mbf{G}_B$ are colored blue by the definition of green-edge coloring.  Thus the only edges that could cause failure to be a proper coloring are defect edges and crossing edges with both endpoints are incident to defect edges.  Let
\[ \Gamma = (S \times \{ \TRUE, \FALSE \}) \times (T \times \{ \TRUE, \FALSE \} ). \]

Place an arbitrary total order $\preceq$ on the set of all vertices in $A$ and in $B$.  As all edges are disjoint by \cref{claim:disjoint-edges}, we may create a canonical bijection $\bar{S} \times \bar{T} \to \Gamma$ by mapping a vertex $u$ to its unique edge $uv$ and the boolean valuation of $(u \preceq v)$.  A subset $H \subset \Gamma$ induces a bipartite 2-SAT formula $\varphi_H$ on variables $S$, $T$ via
\begin{align*}
    \varphi_H = {\bigwedge_{(e,\TRUE,f,\TRUE) \in H}} (e \vee f) &\wedge \bigwedge_{(e,\TRUE,f,\FALSE) \in H} (e \vee \bar{f}) \\ \wedge \bigwedge_{(e,\FALSE,f,\TRUE) \in H} (\bar{e} \vee f) &\wedge \bigwedge_{(e,\FALSE,f,\FALSE) \in H} (\bar{e} \vee \bar{f}).
\end{align*}

\begin{proposition}\label{prop:GEC-equals-SAT}
    Let $\mbf{G}$ follow \cref{assumption:3-to-4}.  Let $\varphi_\mbf{G}$ denote the 2-SAT formula $\varphi_H$ corresponding to $H = E_\mr{cr} \cap (\bar S \times \bar T)$.  Then we have an equality of events
    \[ \on{GEC}(\mbf{G}) = \SAT(\varphi_H). \]
\end{proposition}

\begin{proof}
    We will create a correspondence between green-edge colorings and satisfying assignments.  Since a green-edge coloring necessarily maps $A \setminus \bar S$ to red and $B \setminus \bar T$ to blue, we regard a green-edge coloring as a map
    \[ \chi : (\bar S \cup \bar T) \to \{\text{ red, green, blue } \} \]
    and an assignment as a map
    \[ \sigma : (S \cup T) \to \{ \TRUE, \FALSE \}. \]

    Given a green-edge coloring $\chi$, define $\sigma : (S \cup T) \to \{ \TRUE, \FALSE \}$ as follows.  For $e \in S \cup T$, let $e = uv$ where $u \preceq v$.  Then set $\sigma(e) = \FALSE$ if and only if $\chi(u) = \text{green}$.  
    We will show $\chi$ properly colors $\mbf{G}$ if and only if $\sigma$ satisfies $\varphi_H$.

    Let $uv \in S$ and $u'v' \in T$ with $u \preceq v$, $u' \preceq v'$.  Then
    \[ (uv \vee u'v') \in \varphi_H \iff (uv, \TRUE, u'v', \TRUE) \in H \iff uu' \in E(\mbf{G}). \]

    The first clause is violated if and only if $\sigma(uv) = \sigma(u'v') = \FALSE$, while $\chi$ is not a proper coloring if and only if $\chi(u) = \chi(u') = \text{green}$ (since $\chi$ must be a green-edge coloring regardless of properness).  By our definition of $\sigma$, these are the same event.  The other cases follow similarly.
\end{proof}

\subsection{Reduction to independent random \texorpdfstring{$2$}{2}-SAT clauses}
We now consider the distribution of $\varphi_H$ for $\mbf{G} \sim \mu_{\lambda,1}$.  We will also call this $\varphi_{E_\mr{cr}}$.  We wish to compare to a random bipartite 2-SAT formula on $S \cup T$ with each clause occurring independently.  We will couple $E_\mr{cr}$ with an independent random subset $X \subset S \on{\Box} T$, and then use \cref{prop:GEC-equals-SAT} to convert both $E_\mr{cr}$ and $X$ into bipartite 2-SAT formulas $\varphi_{E_\mr{cr}}$, $\varphi_X$.

Let \begin{equation}\label{eq:lambda_0-def} \lambda_0 = 1 - \left(1 + 4\lambda + 2\lambda^2\right)^{-1/4} = \lambda - 2\lambda^2 + O(\lambda^3), \end{equation}
and let $X$ be a random subset of $S \on{\Box} T$ where each vertex is included independently with probability $\lambda_0$. 

Notice that, once the isolated vertices of $\mbf{G}[A]$ and $\mbf{G}[B]$ have been removed, $S \on{\Box} T$ consists of disjoint copies of $C_4$ as $S$ and $T$ are both matchings by \cref{claim:disjoint-edges}.  
Fix $C \subset S \on\Box T$ a $4$-cycle.  We say a coupling $(E_\mr{cr},X)$ \textit{succeeds on $C$} if
\begin{align*}
    E_\mr{cr} \cap C = \emptyset & \iff X \cap C = \emptyset \\
    E_\mr{cr} \cap C = \ ^{\bullet \hspace{8pt}} & \iff X \cap C \in \{ ^{\bullet \hspace{8pt}}, ^{\bullet \hspace{4pt} \bullet} \} \\
    E_\mr{cr} \cap C = \ ^{\hspace{8pt} \bullet} & \iff X \cap C \in \{ ^{\hspace{8pt}\bullet }, ^{\hspace{8pt} \bullet}_{\hspace{8pt} \bullet} \} \\
    E_\mr{cr} \cap C = \ _{\hspace{8pt} \bullet} & \iff X \cap C \in \{ _{\hspace{8pt}\bullet }, _{\bullet \hspace{4pt} \bullet} \} \\
    E_\mr{cr} \cap C = \ _{\bullet \hspace{8pt}} & \iff X \cap C \in \{ _{\bullet \hspace{8pt}}, _{\bullet \hspace{8pt}}^{\bullet \hspace{8pt}}  \} \\
    E_\mr{cr} \cap C = \ ^\bullet_{\hspace{8pt} \bullet} & \iff X \cap C = \ ^\bullet_{\hspace{8pt} \bullet} \\
    E_\mr{cr} \cap C = \ _\bullet^{\hspace{8pt} \bullet} & \iff X \cap C = \ _\bullet^{\hspace{8pt} \bullet}.
\end{align*}

\begin{claim}\label{claim:coupling-succeeds}
    Let $\mbf{G}$ follow \cref{assumption:3-to-4}.  Then there exists a coupling $(E_\mr{cr},X)$ that succeeds on all cycles $C \subset S \on{\Box} T$  whp.
\end{claim}

\begin{proof}
As both $X$ and $E_\mr{cr}$ behave independently on different connected components, we will explicitly couple on a copy of $C_4$, and then union bound over the number of copies.  Fix $C \subset S \on{\Box} T$ inducing $C_4$.  For $F \subset C$, let $\Pr_X[F] := \Pr[X \cap C = F]$ and $\Pr_E[F] = \Pr[E_\mr{cr} \cap C = F]$.  Let the event $B_C$ denote the coupling failing on $C$.  We will show $\Pr[B_C] = O(\lambda^3) = \widetilde{O}(n^{-3/2})$.  Note that $\Pr_X[\emptyset] = \Pr_E[\emptyset]$ by choice of $\lambda_0$.  Thus by a union bound
\[ \Pr[B_C] \leq 4\big|\Pr_X[^{\bullet \hspace{8pt}}] + \Pr_X[^{\bullet \hspace{4pt} \bullet}] - \Pr_E[^{\bullet \hspace{8pt}}]\big| + 2\big|\Pr_X[^\bullet_{\hspace{8pt} \bullet}] - \Pr_E[^\bullet_{\hspace{8pt} \bullet}]\big| + \Pr[|X \cap C| \geq 3]. \]

We now bound these term-by-term.
\begin{align*}
    \Pr_X[^{\bullet \hspace{8pt}}] + \Pr_X[^{\bullet \hspace{4pt} \bullet}] - \Pr_E[^{\bullet \hspace{8pt}}] & = \lambda_0(1-\lambda_0)^3 + \lambda_0^2(1-\lambda_0)^2 - \frac{\lambda}{1+4\lambda+2\lambda^2} \\
    & = (\lambda - 2\lambda^2)(1 - 3\lambda) + \lambda^2(1) - \lambda(1-4\lambda) + O(\lambda^3) = O(\lambda^3). \\
    \Pr_X[^\bullet_{\hspace{8pt} \bullet}] - \Pr_E[^\bullet_{\hspace{8pt} \bullet}] & = \lambda_0^2(1-\lambda_0)^2 - \frac{\lambda^2}{1+4\lambda+2\lambda^2} \\
    & = \lambda^2(1-O(\lambda)) - \lambda^2(1+O(\lambda)) = O(\lambda^3). \\
    \Pr[|X \cap C| \geq 3] & = 4\lambda_0^3(1-\lambda_0) + \lambda_0^4 = O(\lambda^3).
\end{align*}

Thus $\Pr[B_C] = O(\lambda^3)$. By a union bound  and \cref{lem:concentration-of-S},
\[ \Pr\left[ \bigvee_{C \subset S \on{\Box} T} B_C \right] \leq |S| \cdot |T| \cdot \Pr[B_C] = \widetilde{O}(n^{1/2}) \cdot \widetilde{O}(n^{1/2}) \cdot \widetilde{O}(n^{-3/2}) = \widetilde{O}(n^{-1/2}). \]

Thus our coupling succeeds on all cycles $C$ whp.
\end{proof}

Notice that if the coupling succeeds, then $E_\mr{cr} \subseteq X$.  Monotonicity of satisfiability immediately gives the following.

\begin{corollary}\label{cor:comparison-random-to-HCM-for-SAT}
    For $\mbf{G}$ following \cref{assumption:3-to-4},
    \[ \mb{P}[\SAT(\varphi_X)] \leq \mb{P}[\SAT(\varphi_{E_\mr{cr}})] + o(1). \]
\end{corollary}

We now bound in the other direction.
\begin{lemma}\label{lem:comparison-random-to-HCM-for-SAT}
    For $\mbf G$ following \cref{assumption:3-to-4},
    \[ \mb{P}[\SAT(\varphi_X)] \geq \mb{P}[\SAT(\varphi_{E_\mr{cr}})] - o(1). \]
\end{lemma}

\begin{proof}
For a cycle $C \subset S \on{\Box} T$, if the coupling succeeded, notice that $X \cap C \supsetneq E_\mr{cr} \cap C$ if and only if $X \cap C$ is two adjacent vertices.  Thus we have
\[ \Pr[\on{UNSAT}(\varphi_X)] \leq \Pr[\on{UNSAT}(\varphi_{E_\mr{cr}})] + \Pr[\on{UNSAT}(\varphi_X)\text{ due to }^{\bullet \hspace{4pt} \bullet} \text{ configurations}] + \Pr \left[ \bigvee_{C \subset S \on{\Box} T} B_C \right]. \]

By \cref{claim:coupling-succeeds}, the last probability is $o(1)$.  Thus to show \cref{lem:comparison-random-to-HCM-for-SAT}, it suffices to show that whp
, the $^{\bullet \hspace{4pt} \bullet}$ configurations do not cause satisfiability to fail.  

\begin{claim}\label{claim:variable-partial-assignment}
Under \cref{assumption:3-to-4}, $^{\bullet \hspace{4pt} \bullet}$-type configurations have $o(1)$ impact on satisfiability of $\varphi_X$.
\end{claim}

This claim is almost entirely a corollary of \cref{prop:spine-expectation} and uses the language of \cref{sec:2-SAT-intro,sec:SAT-lower-bounds,sec:SAT-upper-bounds}.  We leave the proof to \cref{app:variable-partial-assignment}.  The heuristic explanation is that in the scaling window, there are $\Theta_\mr{p}(1)$ many $^{\bullet \hspace{4pt} \bullet}$-type configurations while the ``giant component'' of $G[\bar S \cup \bar T]$ contains $o(1)$ proportion of the vertices.  The giant component is the dominant term in determining satisfiability (which is equivalent to $3$-colorability), and by a union bound, none of these $^{\bullet \hspace{4pt} \bullet}$-type configurations will be in the giant component.

Thus $\mb{P}[\on{UNSAT}(\varphi_X)] \leq \mb{P}[\on{UNSAT}(\varphi_{E_\mr{cr}})] + o(1)$, and subtracting each side from 1 gives the result.
\end{proof}

Thus
\[ \mb{P}[\SAT(\varphi_X)] = \mb{P}[\SAT(\varphi_{E_\mr{cr}}] + o(1). \]

\subsection{Parameter comparison}\label{sec:parameter-comparison}

\begin{proof}[Proof of \cref{theorem:main-erdos-renyi}]

We have now completely reduced to the bipartite random 2-SAT.  By \cref{cor:3-coloring-is-GEC}, \cref{prop:GEC-equals-SAT}, \cref{cor:comparison-random-to-HCM-for-SAT}, and \cref{lem:comparison-random-to-HCM-for-SAT},
\[ \mb{P}[\chi(\mbf{G}) = 3] = \mb{P}[\SAT(\varphi_X)] + o(1) \]
where $\varphi_X$ is a random bipartite 2-SAT formula on $S \cup T$ where each clause is included independently with probability $\lambda_0$.  Notice that $\varphi_X$ is equal in distribution to $F_{|S|,|T|,\lambda_0}$ as in the statement of \cref{theorem:bipartite-2SAT}.  Thus to learn the scaling window, we need to evaluate
    \begin{equation}\label{eq:kappa} \kappa := \paren{2 (|S||T|)^{1/2} \lambda_0 - 1} \cdot (|S| |T|)^{1/6}. \end{equation}
To simplify calculations, instead of working with $\omega$, let $\gamma$ be such that
    \[ \lambda = \paren{\frac{2\psi(n)}{n}(1+\gamma)}^{1/2}. \]
    We will not assume that $\gamma$ is either positive or negative, but will assume $|\gamma| = o(1)$ but may decay arbitrarily slowly.

    We begin by computing that under \cref{assumption:3-to-4}
    \begin{align*}
        & \qquad 2 \lambda_0 (|S| |T|)^{1/2} = (1+\widetilde{O}(n^{-1/2})) 2 \lambda (|S||T|)^{1/2} & \text{by \cref{eq:lambda_0-def} (def of $\lambda_0$)} \\
        & = (1+\widetilde{O}(n^{-1/4})) 2\lambda \paren{q_0 \frac{n^2}{8} + \widetilde{O}(n^{1/4})} & \text{by \cref{lem:concentration-of-S}} \\
        & = (1+\widetilde{O}(n^{-1/4})) \frac{n^2}{4} \lambda^2 \exp\left\{ - \frac n2 \lambda^2 \right\} & \text{plug in }q_0 \\
        & = (1+\widetilde{O}(n^{-1/4})) \frac n2 \left(\psi(n) ( 1 + \gamma) \right) \exp\left\{-\left(\psi(n) + \gamma \psi(n) \right) \right\} & \text{plug in }\lambda \\
        & = (1+\widetilde{O}(n^{-1/4})) (1+\gamma) \exp\{-\gamma \psi(n) \} & \text{definition of }\psi(n) \\
        & =  (1+\widetilde{O}(n^{-1/4})) \exp\{ - (1+O(1/\log n)) \gamma \psi(n) \}. & \text{Taylor expand $\log(1+\gamma)$}
    \end{align*}
    At this point, we will assume $|\gamma| = o(1/\log n)$.  Monotonicity of satisfiability will allow us to extend to $|\gamma| = \Omega(1/\log n)$.  Thus by Taylor's theorem, we get
    \[ 2\lambda_0 (|S| |T|)^{1/2} - 1 = (1+o(1)) \gamma \psi(n) + \widetilde{O}(n^{-1/4}) = \frac{2^{-1/6}+o(1)}{(|S||T|)^{1/6}} \paren{ \gamma n^{1/6} (\log n)^{5/6} + \widetilde{O}(n^{-1/12})  }, \]
    recalling $(|S| |T|)^{-1/6} = (1+o_\mr{p}(1))2^{-1/6} n^{-1/6} (\log n)^{1/6}$.  Returning to \cref{eq:kappa}, we have $\kappa = (2^{-1/6}+o(1)) \gamma n^{1/6} (\log n)^{5/6}$ for $|\gamma| \gg n^{-1/4}$ (to dominate the error term).  Thus \cref{theorem:bipartite-2SAT} gives us the scaling for
    \[ n^{-1/6} (\log n)^{-5/6} \ll |\gamma| \ll (\log n)^{-1}. \]
    With this in mind, let $\omega' = \gamma n^{1/6} (\log n)^{5/6}$ so that $\kappa = \Theta(\omega')$.  Then
    \begin{align*}
        p & = (1+\widetilde{O}(n^{-1/2})) \lambda = (1+\widetilde{O}(n^{-1/2})) \paren{\frac{2\psi(n)}{n} (1+\gamma)}^{1/2} \\
        & = \paren{\frac{2\psi(n)}{n}}^{1/2} + \frac{\gamma}{2} \paren{\frac{2\psi(n)}{n}}^{1/2} + O\paren{\gamma^2 \paren{\frac{\psi(n)}{n}}^{1/2}} \\
        & = \paren{\frac{2\psi(n)}{n}}^{1/2} + (1+o(1)) \gamma \paren{\frac{\log n}{2n}}^{1/2} \\
        & = \paren{\frac{2\psi(n)}{n}}^{1/2} + (1+o(1)) 2^{-1/2} \omega' (\log n)^{-1/3} n^{-2/3}.
    \end{align*}

    Thus we have $\omega = \Theta(\omega') = \Theta(\kappa)$, and so for any
    \[ 1 \ll \omega \ll n^{1/6} (\log n)^{-1/6}, \]
    we can apply \cref{theorem:bipartite-2SAT} to get the correct scaling window.  For $|\omega| = \Omega(n^{1/6} (\log n)^{-1/6})$ but $|\omega| = o(n^{1/6} (\log n)^{5/6})$, by monotonicity of satisfiability, we retain (un)satisfiability whp.
    
\end{proof}

\section{From \texorpdfstring{$4$}{4}-colorability to \texorpdfstring{$5$}{5}-colorability}\label{sec:4-coloring-scaling-window}

In this section, we prove \cref{thm:4-coloring-scaling-window} on the transition from $4$-colorability to $5$-colorability.

\subsection{Preliminaries at critical \texorpdfstring{$p$}{p}}\label{subsec:preliminaries-for-small-p}
Through the previous sections, we have worked with $\mbf{G} \sim \mu_{\lambda,1}$ (see \cref{def:mu-lambda-1}) and converted our results to $\mbf{G} \sim \mu_\lambda$ the desired distribution by using \cref{thm:JPP-full-strength}.  Note that \cref{thm:JPP-full-strength} only holds when $\lambda \geq (1+\eps)\sqrt{\log n/n}$ for $\eps>0$ a constant, while under \cref{assumption:4-to-5} below, $\lambda = (1 + o(1)) \sqrt{\log n/n}$.  

We will instead use the following result for smaller $p$.
\begin{definition}[{\cite[Algorithm 3]{jenssen2023evolution}}]\label{def:mu-lambda-2}

Given $n$ and $\lambda$, define the measure $\mu_{\lambda,2}$ to be the law of the output of the following procedure for generating a random graph $\mbf{G}$ on $[n]$.

\begin{enumerate}
    \item Choose $\zeta \in \mb{Z}$ via
    \[ \mb{P}[\zeta = t] \propto (1+\lambda)^{-t^2}. \]
    \item Choose $A \in \binom{[n]}{\lfloor n/2 \rfloor - \zeta}$ uniformly at random, and let $B = [n] \setminus A$.  If $\zeta > \lfloor n/2 \rfloor$ or $\zeta < -\lceil n/2 \rceil$, abort and output the empty graph.
    \item\label{item:algorithm-step-with-ERG} Choose $S \subset \binom A2$ and $T \subset \binom B2$ according to independent exponential random graphs $G(A,q_2,\psi)$ and $G(A,q_2,\psi)$ respectively conditioned on triangle-freeness.  Formally, define the parameters
    \begin{align*}
        \frac{q_0}{1-q_0} & := \lambda e^{-n\lambda^2/2}, \\
        \frac{q_1}{1-q_1} & := \lambda e^{-\lambda^2n/2+\lambda^3n - 7\lambda^4n/4}, \\
        \mu & := \binom{n/2}{2} q_1 e^{\lambda^3 n^2 q_0}, \\
        \frac{q_2}{1-q_2} & := \frac{q_1}{1-q_1} e^{4\lambda^3\mu}, \\
        \psi & := \lambda^3n/2,
    \end{align*}
    and then define the probability measure $G(A,q_2,\psi)$ via
    \[ \mb{P}[G(A,q_2,\psi) = H] \propto \paren{\frac{q_2}{1-q_2}}^{e(H)} \exp\{\psi P_2(H)\} \mbf{1}_{H \in \mc{T}} \mbf{1}_{\Delta(H) \leq 50 \max\{q_2n,\log n\}}, \]
    where $P_2(H)$ denotes the number of copies of $P_2$ in $H$.  Note that $q_2 = (1+o(1))q_0$ and $G(A,q_2,\psi)$ is well-approximated by $G(A,q_2 \mid \mc{T})$ (see \cref{prop:sandwiching-ERG-between-ERs}).
    \item Choose $E_\mr{cr} \subset A \times B$ according to the hard-core model on $S \on{\Box} T$ at fugacity $\lambda$.
    \item Output $S \cup T \cup E_\mr{cr}$.
\end{enumerate}
\end{definition}

Replacing $\mu_{\lambda,1}$ with $\mu_{\lambda,2}$ gives equivalent statements to those in \cref{sec:preliminaries} but allowing for smaller $p$.

\begin{theorem}[{\cite[Theorem 2.4]{jenssen2023evolution}}]\label{thm:JPP-full-strength-small-p}
    Fix $0 < \eps \leq 1/14$ and suppose $p \geq (1-\eps) \sqrt{\frac{\log n}{n}}$.  Then, with $\lambda = p/(1-p)$, we have
    \[ \| \mu_\lambda - \mu_{\lambda,2}\|_\mr{TV} = o(1), \]
    where $\mu_\lambda$ denotes the distribution of a sample from $G(n,p)$ conditioned on triangle-freeness.
\end{theorem}

\begin{lemma}[\cite{jenssen2023evolution}]\label{lem:strongly-balanced-small-p}
    Whp, $|A| = n/2 + O((n\log n)^{1/4})$ and $|B| = n/2 + O((n \log n)^{1/4})$.
\end{lemma}

\begin{lemma}[\cite{jenssen2023evolution}]\label{lem:expander-small-p}
    Let $\mbf{G} \sim \mu_{\lambda,2}$ and let $(A,B)$ be the (unique) max cut of $\mbf{G}$.  Then whp, for all $v \in A$,
    \[ \frac{\lambda n}{30} \leq d_\mbf{G}(v,B) \leq 5 \lambda n \]
    and
    \[ X \subset A, \ |X| \geq \lambda n/100, \ Y \subset B, \ |Y| \geq n/9 \implies d_\mbf{G}(X,Y) > 0, \]
    and likewise with $A$ and $B$ exchanged.
\end{lemma}

The main new difficulty is that the exponential graphs in step \ref{item:algorithm-step-with-ERG} of \cref{def:mu-lambda-2} appear much more complicated than the Erd\H{o}s--R\'enyi random graphs of \cref{def:mu-lambda-1}.  The following `sandwiching' result clarifies the behavior.

\begin{proposition}[{\cite[Proposition 10.11]{jenssen2023evolution}}]\label{prop:sandwiching-ERG-between-ERs}
    Suppose $\lambda \geq \frac{13}{14} \sqrt{\log n/n}$.  Let $q_\ell = q_2 (1-n^{-2/5})$ and $q_u = q_2(1+n^{-2/5})$.  Then for $A \subset [n]$, there is a coupling of $G(A,q_\ell \mid \mc{T}),G(A,q_2,\psi),G(A,q_u \mid \mc{T})$ so that whp,
    \[ G(A,q_\ell \mid \mc{T}) \subset G(A,q_2,\psi) \subset G(A,q_u \mid \mc{T}). \]
\end{proposition}

We will show below that the difference $q_u - q_\ell = 2n^{-2/5}$ is insignificant in our calculations regarding the chromatic number, so it will suffice to study the Erd\H{o}s--R\'enyi random graph conditioned on triangle-freeness.

Note that conditioning on triangle-freeness is now meaningful.  In previous sections, $q$ was sufficiently small that $G(A,q) \in \mc{T}$ whp.  
This is no longer true for the parameter regime of this section.  
At the critical $p$, we have $q = 2c/n = (1+o(1)) c/|A|$, so $G(A,q)$ has a Poisson distribution of triangles.  
We will show that when $c$ is close to $0$, the probability of $4$-colorability is close to 1, and when $c$ is close to $1$, the probability of $4$-colorability is small.  Translating these values of $q$ to values of $p$ will give the desired scaling window.

\begin{observation}\label{obs:graph-property}
    Let $\mc{A}$ be a monotone increasing graph property and $q_\ell = 2c/n$ for $c \in \mb{R}$.  If $|A| = (1+o(1)) \frac n2$ and $G(A,q_\ell) \in \mc{A}$ whp, then $\mbf{G}[A] \in \mc{A}$ whp.
\end{observation}

\begin{proof}
    Note that $\mb{P}[G(A,q_\ell) \in \mc{T}]$ is bounded away from both 0 and 1 for $q_\ell = 2c/n$.  By Bayes' law,
    \[ \mb{P}[G(A,q_\ell \mid \mc{T}) \notin \mc{A}] = \frac{\mb{P}[G(A,q_\ell) \in \overline{\mc{A}} \cap \mc{T}]}{\mb{P}[G(A,q_\ell) \in \mc{T}]} \leq \frac{\mb{P}[G(A,q_\ell) \in \overline{\mc{A}}]}{\mb{P}[G(A,q_\ell) \in \mc{T}]} = \frac{o(1)}{\Omega(1)} = o(1). \]

    Thus $G(A,q_\ell \mid \mc{T}) \in \mc{A}$ whp
    .
    Since $\mc{A}$ is an increasing property, the first half of the coupling in \cref{prop:sandwiching-ERG-between-ERs} gives the result.
\end{proof}

\subsection{Proof of \texorpdfstring{\cref{thm:4-coloring-scaling-window}}{theorem}}

\begin{assumption}\label{assumption:4-to-5}
    Fix $c \in (0,1)$.  Let $q_2 = 2c/n$.

\begin{itemize}
    \item Let $p = p(q_2)$ be the unique choice of $p$ resulting in our choice of $q_2$.  Let $\lambda = p/(1-p)$.
    \item $\mbf{G}$ is a random graph $\mbf{G} \sim \mu_{\lambda,2}$ on $n$ vertices (see \cref{def:mu-lambda-2}).
    \item $\mbf{G}$ has a unique max cut $V(\mbf{G}) = A \cup B$.  (This exists whp by \cref{lem:strongly-balanced}.)
    \item $S$ and $T$ are the edge sets of $\mbf{G}[A], \mbf{G}[B]$ respectively.  We will call $\mbf{G}[A]$ and $\mbf{G}[B]$ the defect graphs.
    \item Let $q_\ell = (1-n^{-2/5}) q_2$ and $q_u =  (1+n^{-2/5}) q_2$.
\end{itemize}

\end{assumption}

Note that
\[ p = \sqrt{\frac{\log n + \log\log n - 2\log(2c) + o(1)}{n}} \,. \]

\begin{proposition}\label{lem:odd-cycle-no-4-color}
    Let $\mbf{G}$ follow \cref{assumption:4-to-5}.  Suppose $\mbf G[A]$ contains an odd cycle.  Then whp, $\mbf G$ cannot be $4$-colored.
\end{proposition}
\begin{proof}
Let $\mbf{G} \sim \mu_{\lambda,1}$ and suppose that $\mbf{G}$ follows the balance conditions in \cref{lem:strongly-balanced} and the expander conditions in \cref{lem:expander}, both of which hold whp.

We begin by showing that $\mbf G[B]$ has a matching of size $\Theta(n)$ whp.  This is a second moment calculation; for $M$ the number of components of $G(B,q_\ell)$ of size exactly 2, we have
\begin{align*}
    \mb{E}[M] & = \binom{|B|}{2} q_\ell (1-q_\ell)^{2(n/2-2)} = (1+O(n^{-3/4})) \frac{c e^{-2c}}{4} n, \\
    \mb{E}[M^2] & = \mb{E}[M] + \frac{(|B|)_4}{4} q_\ell^2 (1-q_\ell)^{2(n/2-2) + 2(n/2-4)} = (1+O(1/n)) \mb{E}[M]^2.
\end{align*}

Thus by Chebyshev's inequality, whp, we have a matching of size $\Theta(n)$ in $G(B,q_\ell)$ and so in $\mbf{G}[B]$ by \cref{obs:graph-property} (note that the matching in $\mbf{G}[B]$ need not be induced).  By almost analogous reasoning, $\mbf{G}[A]$ also has a matching of size $\Theta(n)$.  There is the slight complication that $\mbf{G}[A]$ has also been conditioned to contain an odd cycle; since this is also an event that occurs with $\Omega(1)$ probability, the same computation as in \cref{obs:graph-property} shows that high-probability events transfer to $\mbf{G}[A]$ as well.

Suppose $\chi : V \to \{ \text{red, blue, green, yellow} \}$ is a proper $4$-coloring of $\mbf G$.  Due to the large matching, there must be at least two colors with $\Omega(n)$ many vertices in $A$ and similarly at least two colors with $\Omega(n)$ many vertices in $B$.  By the expansion properties of \cref{lem:expander-small-p}, these two pairs of colors must be all distinct; without loss of generality, red and blue for $A$ and green and yellow for $B$.  We say that a vertex $u \in A$ (resp.~$v \in B$) is \emph{miscolored} if $\chi(u) \in \{\text{green, yellow}\}$ (resp.~$\chi(v) \in \{\text{blue, red}\}$).

By assumption, there is an odd cycle $C_\mr{odd} \subset A$.  This cannot be red/blue-colored and so requires a miscolored vertex $w \in C_\mr{odd}$.  We analyze how much one miscoloring in $A$ propagates to cause further miscolorings in $B$.

Let $\mc{P}_A$ (resp.~$\mc{P}_B$) be the collection of all components of $\mbf{G}[A]$ (resp.~$\mbf{G}[B]$) isomorphic to a path of length 3.  Fix a map $\varphi : \mc{P}_A \cup \mc{P}_B \to A \cup B$ where for all paths $\gamma$, we have $\varphi(\gamma) \in \gamma$.  This will indicate a vertex we choose to miscolor if $\gamma$ cannot be $2$-colored due to a miscolored vertex across the way.  Define the random directed bipartite graph $\Gamma_\varphi$ with $V(\Gamma_\varphi) = \mc{P}_A \cup \mc{P}_B$ and
\[ (\gamma,\gamma') \in E(\Gamma_\varphi) \iff \varphi(\gamma) \text{ connected in } \mbf G \text{ to both endpoints of } \gamma'. \]

Add a third component to $\Gamma_\varphi$ consisting of the vertices $w \in C_\mr{odd}$ with
\[ (w, \gamma) \in E(\Gamma_\varphi) \iff w \text{ connected in }\mbf G \text{ to both endpoints of }\gamma. \]

\begin{claim}\label{claim:miscolorings-propagate}
    Whp, for all maps $\varphi$ and all $w \in C_\mr{odd}$, the component $C(w)$ in $\Gamma_\varphi$ has $|C(w)| = \Omega(n)$.
\end{claim}

We show how this finishes the proof before proving the claim.  Suppose $\mbf{G}$ follows the condition of \cref{claim:miscolorings-propagate}.  As noted, there must be some $w \in C_\mr{odd}$ which is miscolored.  Consider those $\gamma \in N(w)$.  This means that $w$ is $\mbf G$-adjacent to both endpoints of $\gamma \in \mc{P}_B$, meaning $\mbf G[\gamma \cup \{w\}] \cong C_5$.  Since $w$ is miscolored and all 5 vertices cannot be green/yellow-colored, this forces there to be a miscolored vertex $v_\gamma \in \gamma$.  Assign $\varphi(\gamma) = v_\gamma$.  Iterating this process, we build some partial map $\varphi$, and may arbitrarily assign the remaining values of $\varphi$.

Notice that for each $\gamma \in C_{\Gamma_\varphi}(w)$, we have $\varphi(\gamma)$ is miscolored.  Thus we have $|C(w)|$ miscolored vertices, and so by pigeonhole, $|C(w)|/2$ miscolored vertices on the same side and so by pigeonhole again $|C(w)|/4$ vertices on the same side miscolored to the same color.

Thus this color has $\Omega(n)$ vertices in $A$ and $\Omega(n)$ vertices in $B$.  This contradicts the expander property in \cref{lem:expander-small-p}.  Since we only conditioned on high probability events, this means there is no such $4$-coloring $\chi$ whp.
\end{proof}

\begin{proof}[Proof of \cref{claim:miscolorings-propagate}]
We need to understand the structure of $\Gamma_\varphi$.  First, we lower bound $|\mc{P}_A|$.

Without conditioning on triangle-freeness, for $\mc{P}_A^\ell$ the number of isolated induced $P_3$s in $G(A,q_\ell)$ we compute
    \begin{align*}
        \mb{E}[|\mc{P}_A^\ell|] & = \frac{1}{|\on{Aut}(P_3)|} \sum_{(v_1,v_2,v_3,v_4) \in A^4} \mb{P}(\mbf G[\{v_1,v_2,v_3,v_4\}] \cong P_3) \\
        & = \frac{(|A|)_4}{2} q_\ell^3 (1-q_\ell)^{4(|A|-4) + 3} = (1+O(n^{-3/4})) \frac{c^3 n}{4} \exp\left\{ - 4c + O(1/n) \right\}\\
        \mb{E}[|\mc{P}_A^\ell|^2] & = \frac{1}{|\on{Aut}(P_3)|^2} \sum_{(v_1,\ldots,v_8) \in A^8} \mb{P}(\mbf G[\{v_1,v_2,v_3,v_4\}] \cong \mbf G[\{v_5,v_6,v_7,v_8\}] \cong P_3) \\
        & = \frac{(|A|)_8}{4} q_\ell^6 (1-q_\ell)^{8(|A|-8) + 22} + O(n) = (1+O(1/n)) \mb{E}[|\mc{P}_A^\ell|]^2.
    \end{align*}

    Thus by Chebyshev's inequality,
    \[ |\mc{P}_A^\ell| = \frac{c^3 e^{-4c}}{4} n + O_\mr{p}(\sqrt n). \]

    By Bayes' law, $G(A,q_\ell \mid \mc{T})$ has $\Omega(n)$ many components isomorphic to $P_3$ as well.

    We now consider how many of these remain induced isolated $P_3$s in $G(A,q_u \mid \mc T)$ under the coupling in \cref{prop:sandwiching-ERG-between-ERs}.  Note that whp
    , $G(A,q_\ell \mid \mc{T})$ and $G(A,q_u \mid \mc{T})$ each have $q_2 \binom{n}{2} + O(n^{3/5})$ edges.  Thus, under the coupling, $|G(A,q_\ell) \setminus G(A,q_u)| = O(n^{3/5})$, and so at most $O(n^{3/5})$ of these $P_3$s can be touched.  Thus there are $\Omega(n)$ copies of $P_3$ isolated and induced in both $G(A,q_\ell \mid \mc{T})$ and $G(A,q_u \mid \mc{T})$, and so in $\mbf{G}[A]$ by \cref{prop:sandwiching-ERG-between-ERs}.

Thus $|\mc{P}_A| = \Theta(n)$ and similarly $|\mc{P}_B| = \Theta(n)$ both whp.

Recall that the edges $(\gamma,\gamma') \in E(\Gamma_\varphi)$ occur if $\varphi(\gamma)$ is connected to both endpoints of $\gamma'$ in $G$.  The crossing edges in $\mbf G$ are according to the hard-core model at $\lambda$ on $S \on{\Box} T$.  Since each $\gamma \in \mc{P}_A, \gamma' \in \mc{P}_B$ is a connected component, their product $\gamma \on{\Box} \gamma'$ is a connected component of $S \on{\Box} T$ isomorphic to $P_3 \on{\Box} P_3$.  For any choice of $\varphi(\gamma)$, we have $\mb{P}[(\gamma,\gamma') \in E(\Gamma_\varphi)] = (1+O(\lambda))\lambda^2$.  Furthermore, for any fixed $\varphi$, a collection of edges $F \subset E$ are mutually independent unless there is some $\gamma,\gamma'$ with $(\gamma,\gamma') \in F$ and $(\gamma',\gamma) \in F$.

    We now iteratively grow the component containing $w$ using the exploration method.  Let $Y_0 = \{w\}$.  We will then build up a sequence $Y_{k+1} \subset N(Y_k)$ of ``layers'' of the component, showing by a union bound that $|Y_{k+1}|$ is large for all $\varphi$, until we have $\Omega(|\mc{P}_A|)$ vertices.

    Let $N = \frac12\min\{|\mc{P}_A|, |\mc{P}_B|\}$ and fix arbitrary total orders on $\mc{P}_A,\mc{P}_B$.  We will always explore a vertex's out-neighborhood among only the first $N$ vertices not yet in the component, setting the remainder aside as a reservoir and refilling so there are always $N$ vertices to be explored.  Let $\mu = \lambda^2 N$.  Notice $\mu = \Theta(\log n)$.  We will show that $|Y_k| \geq (\mu/20)^k$.

    The base case is $Y_1 \subset N(w)$ has size at least $\mu/20$.  By a Chernoff bound, each $w \in C_\mr{odd}$ has $\deg_{\Gamma_\varphi} w \geq \mu/4$ except with probability $n^{-\Omega(1)}$.  Further, whp, $|C_\mr{odd}| \leq \log n$.  This follows from a first moment computation.  The expected number of cycles of length $k$ is $(n)_k q_2^k/k \leq c^k/k$, and so the expected number of cycles longer than $k$ is
    \begin{equation}\label{eq:no-large-cycles} \sum_{s > k} \frac1s c^s \leq \frac{1}{k} \sum_{s=0}^\infty c^s = \frac{1}{(1-c)k}. \end{equation}

Thus for $k = \log n$, we have no cycles longer than $\log n$ whp, and so by a union bound, all vertices $w \in C_\mr{odd}$ have $|Y_1| \geq \mu/4$ whp.

    We describe the exploration process to get from a ``layer'' $Y_k =: Y$ to the next layer $Y_{k+1}$.

    Given a vertex $\gamma \in Y$, we walk through the paths $\gamma'$ one by one checking if $(\gamma,\gamma') \in \Gamma_\varphi$.  We continue until we have either (a) found $\mu/2$ neighbors, succeeding, or (b) checked $N$ vertices of $\Gamma_\varphi$ without finding $\mu/2$ neighbors.  By a Chernoff bound, the probability of failure is $\exp\{-\Omega(\mu)\} $.

    We will now go through the whole layer $Y$ checking each $\gamma \in Y$ as above.  We want at least $|Y|/10$ successes to guarantee $|Y| \mu/20$ new neighbors.  Recall that the successes are independent.  By Hoeffding's inequality, we have
    \begin{multline*}
     \mb{P}\left[ \on{Binom}(|Y|,e^{-\Omega(\mu)}) \leq |Y|/10\right] \leq \exp\left\{ - \frac{2 ((1-e^{-\Omega(\mu)})|Y| - |Y|/10)^2}{|Y|} \right\} \\ = \exp\{-(1.62+e^{-\Omega(\mu)}) |Y|\}. \end{multline*}

    We can now show inductively that as long as the reservoir is not exhausted, $|Y_k| \geq (\mu/20)^k$ for any assignment $\varphi$.  We will then discard everything but $(\mu/20) |Y_k|$ vertices before continuing to the next step.  As long as this process succeeds, $|Y_{k+1}| \geq (\mu/20) |Y_k|$ and so $|Y_i| \geq (\mu/20)^i$ for all $i \leq k$.

    We now need to union bound over all possible $\varphi$.  Note that each $\varphi$ can take $4$ distinct values since $\varphi(\gamma) \in \gamma$.  We have at most $4^{\sum_{i \leq k} |Y_i|}$ possible choices of $\varphi$ thus far.  The probability of success in constructing layer $k+1$ conditioned on success for all $\varphi$ in all the earlier layers $Y_1,\ldots,Y_k$ is then
    \begin{align*}
        & \qquad \exp\{-(1.62+e^{-\Omega(\mu)}) |Y_k| \} \cdot \exp\left\{ \paren{ \sum_{i=1}^k |Y_i| } \log 4 \right\} \\ & = \exp\left\{ - (1.62+e^{-\Omega(\mu)}) \paren{\frac\mu{20}}^k + \log 4 \sum_{i=1}^k \paren{\frac\mu{20}}^i \right\} \\
        & = \exp\left\{ - (1.62+e^{-\Omega(\mu)}) \paren{\frac\mu{20}}^k + \log 4 \frac{(\mu/20)^{k+1} - (\mu/20)}{(\mu/20) - 1} \right\} \\
        & = \exp\left\{ - (1.62 - \log 4 + O(1/\mu)) \paren{\frac\mu{20}}^k \right\} = \exp\{-\Omega((\mu/{20})^k) \}.
    \end{align*}

    The probability that all rounds up through round $k$ have been successful is then
    \[ \sum_{i=1}^k \exp\{ -\Omega((\mu/20)^i) \} = O\paren{\exp\{-\Omega(\mu) \} } \]
    for any $k$.  Notice that we have not assumed any control on $k$; the above holds for any $k = k(\mu)$.

    Thus whp, we always have $|Y_k| \geq (\mu/20)^k$ until the reservoir has run out.  At that point, at least $N$ vertices have been exhausted.  Note that at each step, we discard at most half of the vertices viewed, meaning the final component $\bigcup_i Y_i$ contains at least $N/2 = \Omega(n)$ vertices on one side for all choice $w \in C_\mr{odd}$ and $\varphi$.
\end{proof}

\begin{proof}[Proof of \cref{thm:4-coloring-scaling-window}]
    If $\mbf G[A]$ and $\mbf G[B]$ have no odd cycles, then certainly $\mbf G$ is $4$-colorable.  By \cref{lem:odd-cycle-no-4-color}, if $\mbf G[A] \cup \mbf G[B]$ has an odd cycle, then $G$ is not $4$-colorable whp.  Thus since $S$ and $T$ are independent (conditioned on $(A,B)$), it comes down to
    \begin{equation}\label{eq:chi=4 chi=2^2} \mb{P}[\chi(\mbf G) = 4] = \mb{P}[\chi(\mbf G[A]) = \chi(\mbf G[B]) = 2] + o(1) = \mb{P}[\chi(\mbf G[A]) = 2]^2 + o(1). \end{equation}
    
    By \cref{eq:no-large-cycles}, for $q_2 = 2c/n$ with $c \in (0,1)$,
    \[ \lim_{k\to\infty} \lim_{n\to\infty} \mb{P}[\text{cycle of length }\geq k] = 0. \]
    
    By classical results of Bollob\'as~\cite{bollobas1984evolutionrandom} and Karo\'nski and Ruci\'nski~\cite{karonski1983number}, the number of copies of $C_k$ in $G(n/2,2c/n)$ is Poisson with parameter $c^k/2k$, and these are independent for different $k$.  Thus conditioning on $C_3$-freeness does not affect the other distributions, and so we may compute
    \begin{multline}\label{eq:cycle-counting} \mb{P}[\chi(\mbf G[A]) = 2] = \lim_{k\to\infty}\lim_{n\to\infty} \mb{P}[\text{no odd cycles with length}\leq k] \\ = \lim_{k\to\infty}\exp\left\{ \sum_{s=5}^k \frac{1-(-1)^s}{2} \frac{c^s}{2s} \right\} = \paren{\frac{1-c}{1+c}}^{1/4} \exp\{-c/2 - c^3/6\}. \end{multline}

    We now do a computation of $q_2 = q_2(\lambda)$ to find the scaling window in $p$.  Fix some $c \in (0,1)$.  Let
    \[ \lambda = \sqrt{\frac{-W_{-1}(-(2c)^2/n)}{n}} = \sqrt{\frac{\log n + \log\log n - 2 \log(2c) + o(1)}{n}}. \]

    Then we may compute $p = (1+\widetilde{O}(n^{-1/2})) \lambda$, $q_0 = 2c/n$, and $q_2 = (1+o(1))q_0$ and so by \cref{eq:chi=4 chi=2^2,eq:cycle-counting}
    \[ \mb{P}[\chi(\mbf G) = 4] = \paren{\frac{1-c}{1+c}}^{1/2} \exp\{-c - c^3/3\} + o(1). \]

    Finally, we show $\chi(\mbf G) \leq 5$ whp.  We first bound the number of cycles in $\mbf G[A]$ via a simple first moment calculation.  The expected number of $k$-cycles is $c^k/2k$, and so the expected total number of cycles is $\Theta(1)$, and so by Markov's inequality, whp there are at most $\log n$ many cycles in $A$ and in $B$.  Choose a set $C$ consisting of a random vertex in each cycle.  Since the crossing edges are stochastically dominated by independent choice with probability $\lambda$, the expected number of crossing edges in $C$ is at most $(\log n)^2 \lambda = O(n^{-1/2} (\log n)^{5/2}) = o(1)$.  Thus whp, we may color $C$ all one color.  Now, $\mbf{G}[A \setminus C]$ and $\mbf{G}[B \setminus C]$ are both bipartite and so can be $2$-colored with distinct sets of colors, giving a $5$-coloring of $\mbf G$.
\end{proof}

\section{Results for \texorpdfstring{$\mc{T}(n,m)$}{T(n,m)}}\label{sec:T(nm)}

In this section, we prove the following.

\begin{proposition}\label{prop:lambda-to-fixed-m}
    Let $\mbf{G}_m$ be a uniform sample from $\mc{T}(n,m)$.  Let $\lambda_0 = 4m/n^2$ and let
    \begin{equation}\label{eq:def-of-lambda(m)} \lambda = \lambda(m) := \lambda_0 + \lambda_0^2 + (\lambda_0^2n - 1) \lambda_0 e^{-\lambda_0^2n/2}. \end{equation}

    Let $\mbf{G}_\lambda$ be a sample from $\mbf{G}(n,\lambda/(1+\lambda))$ conditioned on triangle-freeness.  Then
    \[ \mb{P}[\chi(\mbf{G}_m) = 3] = \mb{P}[\chi(\mbf{G}_\lambda) = 3] + o(1) \]
    and
    \[ \mb{P}[\chi(\mbf{G}_m) = 4] = \mb{P}[\chi(\mbf{G}_\lambda) = 4] + o(1). \]
\end{proposition}

\begin{proof}
The following algorithm of Jenssen, Perkins, and Potukuchi approximates a uniform sample from $\mc{T}(n,m)$ up to $o(1)$ total variation distance~\cite[Theorem 1.7]{jenssen2023evolution}.

\begin{definition}[{\cite[Algorithm 1]{jenssen2023evolution}}]
    Given $n$ and $m$, let $\lambda_0 = 4m/n^2$ and let $\lambda$ be as in \cref{eq:def-of-lambda(m)}.
    Define the measure $\mu_{m,1}$ to be the law of the output of the following procedure for generating a random graph on $V$.
    \begin{enumerate}
        \item Choose $\zeta \in \mb{Z}$ via
        \[ \mb{P}[\zeta = t] \propto (1+\lambda)^{-t^2}. \]
        \item Choose $A \in \binom{[n]}{\lfloor n/2 \rfloor - \zeta}$ uniformly at random, and let $B = [n] \setminus A$.  If $\zeta > \lfloor n/2 \rfloor$ or $\zeta < -\lceil n/2 \rceil$, abort and output the empty graph.
        \item Choose $S \subset \binom A2$ and $T \subset \binom B2$ according to independent Erd\H{o}s--R\'enyi random graphs on $A$ and $B$ with edge probability $q_0$ the unique solution in $(0,1)$ to
        \[ \frac{q_0}{1 - q_0} = \lambda e^{-\lambda^2n/2}, \]
        conditioned on triangle-freeness.
        \item Choose $E_\mr{cr} \subset A \times B$ a uniform independent set in $S \on{\Box} T$ of size $m - |S| - |T|$.
        \item Output $S \cup T \cup E_\mr{cr}$.
    \end{enumerate}
\end{definition}

Consider also the following modification of $\mu_{\lambda,1}$ (see \cref{def:mu-lambda-1}).
\begin{definition}
    Given parameters $n$, $\lambda$, and $\xi$, define $\mu_{\lambda,1}^{(\xi)}$ to be the law of the output of the following procedure for generating a random graph on $V$.
    \begin{enumerate}
        \item Choose $\zeta \in \mb{Z}$ via
        \[ \mb{P}[\zeta = t] \propto (1+\lambda)^{-t^2}. \]
        \item Choose $A \in \binom{[n]}{\lfloor n/2 \rfloor - \zeta}$ uniformly at random, and let $B = [n] \setminus A$.  If $\zeta > \lfloor n/2 \rfloor$ or $\zeta < -\lceil n/2 \rceil$, abort and output the empty graph.
        \item Choose $S \subset \binom A2$ and $T \subset \binom B2$ according to independent Erd\H{o}s--R\'enyi random graphs on $A$ and $B$ with edge probability $q_0$ the unique solution in $(0,1)$ to
        \[ \frac{q_0}{1 - q_0} = \lambda e^{-\lambda^2n/2}, \]
        conditioned on triangle-freeness.
        \item Choose $E_\mr{cr} \subset A \times B$ according to the hard-core model on $S \on{\Box} T$ at fugacity $\xi$.
        \item Output $S \cup T \cup E_\mr{cr}$.
    \end{enumerate}
\end{definition}

Note that for $\xi = \lambda$ we have $\mu_{\lambda,1}^{(\lambda)} = \mu_{\lambda,1}$ is our approximation for sampling from $G(n,\lambda/(1+\lambda))$ conditioned on triangle-freeness.  Fix some $m$ in the critical range and let $\lambda = \lambda(m)$ as in \cref{eq:def-of-lambda(m)}.  
We will attempt to couple a sample $\mbf{G}_m \sim \mu_{m,1}$ with $\mbf{G}_\xi \sim \mu_{\lambda,1}^{(\xi)}$ for some $\xi$ to be determined later.  
Begin with the first three steps coupled perfectly, as they are identical.  

Let $\mbf{I}_\xi$ denote a sample from the hard-core model on $S \on{\Box} T$.  Let $A_{\xi}$ denote the event that $\mbf{G}_\xi$ is $3$-colorable and $A_m$ denote the event that $\mbf{G}_m$ is $3$-colorable.  Then
\[ \mb{P}[A_\xi] = \sum_{m=0}^N \mb{P}[|\mbf{I}_\xi| = m]\mb{P}[A_m]. \]

By monotonicity of the chromatic number, for any $m$ and $\xi$,
\[ \mb{P}[A_m] - \mb{P}[|\mbf{I}_\xi| > m] \leq \mb{P}[A_\xi] \leq \mb{P}[A_m] + \mb{P}[|\mbf{I}_\xi| < m]. \]

Thus for any $m,\xi,\xi'$, we rearrange this to get
\begin{equation}\label{eq:xi-m-xi'} \mb{P}[A_\xi] - \mb{P}[|\mbf{I}_\xi| < m] \leq \mb{P}[A_m] \leq \mb{P}[A_{\xi'}] + \mb{P}[|\mbf{I}_{\xi'}| > m]. 
\end{equation}

Applying \cref{prop:hardcore-size-concentration} with $\xi = (1+o(1))\lambda$ gives $\mb{E}|\mbf{I}_\xi| = \xi n^2/4 + \widetilde{O}(n)$ and $\on{Var}(|\mbf{I}_\xi|) = \widetilde{O}(n^{3/2})$.  Let $\xi = \lambda - n^{-1+\eps}$ and $\xi' = \lambda + n^{-1+\eps}$ for $0 < \eps < 1/100$ a constant.  These are sufficiently small shifts that the analysis in \cref{sec:parameter-comparison} carries through and so in particular $|\mb{P}[A_\xi] - \mb{P}[A_{\xi'}]| = o(1)$.  By \cref{prop:hardcore-size-concentration}, we have that $\mb{P}[|\mbf{I}_\xi| < m] = o(1)$ and $\mb{P}[|\mbf{I}_{\xi'}| > m] = o(1)$.  Thus by these facts, by monotonicity of chromatic number and $\xi < \lambda < \xi'$, we get
\[ \mb{P}[A_m] = \mb{P}[A_\lambda] + o(1). \]

Instead letting $A_m$ denote the event of $4$-colorability of $\mbf{G}_m$ and $A_\xi$ the event of $4$-colorability of $\mbf{G}_\xi$, the analysis in \cref{sec:4-coloring-scaling-window} is similarly unaffected by a shift of order $n^{-1+\eps}$ and so the result for $4$-colorability is proven identically.
\end{proof}

\cref{theorem:main-fixed-m} follows from \cref{prop:lambda-to-fixed-m,theorem:main-erdos-renyi}, and \cref{thm:4-coloring-scaling-window-fixed-m} follows from \cref{prop:lambda-to-fixed-m,thm:4-coloring-scaling-window}.

\section{Bipartite random \texorpdfstring{$2$}{2}-SAT}\label{sec:2-SAT-intro}

We will spend the remainder of the paper establishing the scaling window for bipartite random $2$-SAT.  Note that the notation for these sections will be disjoint from the notation in the previous sections of the paper for consistency with the notation of the $2$-SAT literature, namely the paper of Bollob\'as, Borgs, Chayes, and Wilson~\cite{bollobas2001scaling}.  Notably $n$, $p$, and $\lambda$ refer to different quantities than they have thus far.  We begin by restating \cref{theorem:bipartite-2SAT} with more standard variable names.

\begin{theorem}
    For all $C > 0$, there exist $\lambda_0$ large and $\eps_0$ small such that the following holds.  Let $n,m$ be two integers with $n \to \infty$, $|n - m| \leq C n^{2/3}$.  Suppose we have $n+m$ variables $\{x_1,\ldots,x_n\}$ and $\{y_1,\ldots,y_m\}$.  Let $F_{n,m,p}$ be a random bipartite 2-SAT formula on $\{x_i\} \cup \{y_j\}$ with each clause $x_i \vee y_j$ occurring independently with probability
    \[ p := \frac{1 + \eps}{2(mn)^{-1/2}} = \frac{1 + \lambda (mn)^{-1/6}}{2(mn)^{-1/2}}, \]
    where $\lambda_0 < |\lambda|$ and $|\eps| < \eps_0$ (i.e.~$1 \ll |\lambda| \ll n^{1/3}$).  Then as $n \to \infty$,
    \[ \mb{P}(\SAT(F_{n,m,p})) = \begin{cases} \exp[-\Theta(|\lambda|^{-3})] & \lambda < 0 \\ \exp[-\Theta(|\lambda|^3)] & \lambda > 0. \end{cases} \]
\end{theorem}

Throughout the following sections, there will be little exposition as the arguments are largely the same as those of~\cite{bollobas2001scaling}; we refer the reader there for more discussion.

The main difference in the proof is that where the authors of~\cite{bollobas2001scaling} control a value $P_{n,p}(k)$ regarding a vertex having out-neighborhood of size $k$, we instead must control  a value $P_{n,m,p}(k,\ell)$ the probability a vertex has bipartite out-neighborhood with $k$ vertices on one side and $\ell$ on the other.  We then fix $j$ and show that $\sum_\ell P_{n,m,p}(j-\ell,\ell)$ is asymptotic to their quantity $P_{n,p}(j)$.  

The other difficulty that arises is a lack of sharp enumerative bounds on the number of connected sparse bipartite graphs.  We prove a lower bound (\cref{prop:lower-bound-on-bipartite-connected-graphs}) which may be independently useful.  However, having an exponential gap between upper and lower bounds on these quantities leads to some additional hurdles in the proof of \cref{lem:Pnp-m0}.

\subsection{Digraph representation}

Given a 2-SAT formula $F$ on variables $\{x_i\}_{i \in [n]} \cup \{y_i\}_{i \in [m]}$, we will create a directed graph $G = G_F$.  Define
\[ V(G) = X \cup Y, \qquad X = \{x_i\}_{i \in [n]} \cup \{\bar x_i\}_{i \in [n]}, \qquad Y = \{y_i\}_{i \in [m]} \cup \{\bar y_i\}_{i \in [m]}, \] 
\[ E(G) = \{ zz' : (\bar z \vee z') \in F \}. \]

The vertices of $G$ are the literals, and the arcs indicate implication.  Note that $G$ is bipartite if and only $F$ is a bipartite 2-SAT formula.  Note also that $zz' \in E(G)$ if and only if $\bar z' \bar z \in E(G)$; thus $G$ is not simply a directed Erd\H{o}s--R\'enyi random graph.  We write $z \xrightarrow[F]{} z'$ to mean the arc $zz' \in G$ and $z \squig{F} z'$ to mean there is a (directed) path from $z$ to $z'$ in $G$.  When the formula is clear from context, we may simply write $z \to z'$ or $z \squig{} z'$.

A set $S \subset X \cup Y$ is said to be strictly distinct if there does not exist a variable $z$ with $z \in S$ and $\bar z \in S$.

For $z \in X \cup Y$, we will let $L_F^+(z)$ denote the \textit{trimmed outgraph} of $z$.  We produce it as follows.  Let $S = \{z\}$ and iterate the following steps until $S = \emptyset$.
\begin{enumerate}
    \item Choose an arbitrary literal $t \in S$ and remove it from $S$.  If $t \in L_F^+(z)$ or $\bar t \in L_F^+(z)$, skip step (2).
    \item Add $t$ to $L_F^+(z)$.  Add $\{ t' \in X \cup Y : t \xrightarrow[F]{} t' \}$ to $S$.
\end{enumerate}

The final result will be $L_F^+(z)$.  Note that $L_F^+(z)$ is always strictly distinct.

We define the spine of $F$ to be $S(F) = \{ z \in X \cup Y : z \squig{F} \bar z \}$.  Note that $F$ is satisfiable if and only if $S(F)$ is strictly distinct.  This is the main power of this graphical formulation of a 2-SAT formula.  We will use this notion to bound satisfiability.

\subsection{A note on asymptotic notation}
We must be careful with our asymptotic notation, as different parameters are controlled in very different ways.
Throughout this section, we will be assuming $n \to \infty$, and asymptotic notation will generally be with respect to $n$.  As $m = (1+o(1))n$, we are also simultaneously taking $m \to \infty$.  For simplicity, when we do not need lower order terms, we will remain in terms of $n$, e.g.~$(mn)^{-1/2} = O(1/n)$.  Our asymptotic notation will also be assuming $\lambda \gg 1$ and $\eps \ll 1$, so, e.g., $\lambda^2 \ne O(\lambda)$ but $\eps^2 = O(\eps)$.  The parameters $\lambda$ and $\eps$ are not as well controlled as $n$, as we are first fixing $\lambda_0$ and $\eps_0$ as absolute constants and then taking $n \to \infty$.  When we use $o(\cdot)$ notation, we will always explicitly denote the parameter(s) with which the quantity tends to 0, e.g.~$o_n(1)$ or $o_{\lambda,\eps}(1)$.  We will treat $C$ as an absolute constant and absorb it into $O(\cdot)$ notation freely.

\section{Lower bounds}\label{sec:SAT-lower-bounds}

In this section, we show the lower bounds on the probability of satisfiability.  The lower bounds are achieved via first and second moment control on the size of the spine.

\subsection{Satisfiability bounds}

We will begin by demonstrating how the bound on satisfiability follows from the control on the moments of the spine.

\begin{theorem}\label{thm:SAT-lower-bound}
    We have
    \[ \mb{P}[\SAT(F_{n,m,p})] = \begin{cases} \exp\left[ -O(|\lambda|^{-3}) \right] & \lambda < 0, \\ \exp\left[ -O(\lambda^3) \right] & \lambda > 0. \end{cases} \]
\end{theorem}

\begin{proof}

Fix $n$ and $m$ and create the literals $X \cup Y$.  We will couple all formulas $F_p := F_{n,m,p}$.  For each of the $4nm$ possible bipartite disjunctions $c$, assign an i.i.d.~random variable $\xi_c \sim \on{Unif}([0,1])$.  Let
\[ F_p = \bigwedge_{c : \xi_c < p} c. \]

Define the reduced formula process $\Phi_p$ as follows. Start from $\emptyset$, and iteratively, as we hit the ``birthday'' $\xi_c$ of a clause $c$, add $c$ to $\Phi_p$ if and only if $\Phi_p \wedge c$ is satisfiable.

We now define $H_p$ as follows.  For each clause $c \notin \Phi_1$, there is some minimum $p(c)$ so that $\Phi_{p(c)} \wedge c$ is not satisfiable but $\Phi_{p(c) - \delta} \wedge c$ is for all $\delta > 0$.  Choose $\xi'_c \sim \on{Unif}([p(c),1])$ independently for all $c \notin \Phi_1$, and let
\[ H_p = \Phi_p \wedge \bigwedge_{c : \xi'_c \leq p} c. \]

Note that $H_p$ becomes unsatisfiable the first time it diverges from $\Phi_p$, and that $H_p$ has the same distribution as $F_p$.  For a formula $F$, let $U(F)$ denote the total number of clauses $c$ so that $F \wedge c$ is unsatisfiable.  We may then compute
\begin{align*}
    \mb{P}[\SAT(F_p)] & = \mb{P}[\SAT(H_p)] = \mb{E}_{{\Phi}}[\mb{P}[\SAT(H_p)\,|\,{\Phi}]] \\
    & = \mb{E}_{{\Phi}}\left[ \exp\left[ -\int_0^p \frac{U(\Phi_s)}{1-s} \mr ds \right] \right] \\
    & \geq \exp\left[ \mb{E}_{{\Phi}} \left[ - \int_0^p \frac{U(\Phi_s)}{1-s} \mr ds \right] \right] \\
    & = \exp\left[ - \int_0^p \frac{\mb{E}_{{\Phi}}[U(\Phi_s)]}{1-s} \mr ds \right] \\
    & = \exp\left[ - \int_0^p \frac{4mn \mb{P}[x \squig{F_s} \bar x \text{ and } y \squig{F_s} \bar y \text{ for }x \in X, \ y \in Y]}{1-s} \mr ds \right].
\end{align*}

The third line is Jensen's inequality, the next line is the Fubini--Tonelli theorem, and the last is by definition of $U(\Phi_s)$ as the number of clauses such that $F \wedge C$ is unsatisfiable.  Noting that $p \leq 1/2$, we may rewrite the above as
\[ \mb{P}[\SAT(F_p)] \geq \exp\left[ - O(n^2) \int_0^p \mb{P}[x \squig{F_s} \bar x \text{ and } y \squig{F_s} \bar y \text{ for }x \in X, \ y \in Y] \ \mr ds \right]. \]

We now change variables via
\[ s = s(t) = \frac{1 + t n^{-1/3}}{2n} \]
to get
\[ \mb{P}[\SAT(F_p)] \geq \exp\left[ - O(n^{2/3}) \int_{-n^{1/3}}^\lambda \mb{P}[x \squig{F_{s(t)}} \bar x \text{ and } y \squig{F_{s(t)}} \bar y \text{ for }x \in X, \ y \in Y] \ \mr dt \right]. \]

In subsequent sections, we will show the following lemma.

\begin{lemma}\label{lemma:second-moment-bound}
    For arbitrary literals $x \in X$ and $y \in Y$ in $F_{n,m,p}$,
    \[ \mb{P}(x \rightsquigarrow \bar x, y \rightsquigarrow \bar y) = \begin{cases} O(n^{-2/3} \lambda^{-4}) & \lambda < 0, \\ O(n^{-2/3} \lambda^2) & \lambda > 0. \end{cases} \]
\end{lemma}

Proving this lemma is the content of the remainder of this section.  We will first finish the proof of \cref{thm:SAT-lower-bound} assuming \cref{lemma:second-moment-bound}.

\subsubsection*{Subcritical regime (\texorpdfstring{$\lambda < 0$}{negative})}

For $-n^{1/3} \leq t \leq -\lambda_0$, by \cref{lemma:second-moment-bound},
\begin{align*}
    \mb{P}[\SAT(F_p)] & \geq \exp\left[ -O \paren{ \int_{-n^{1/3}}^\lambda t^{-4} \mr dt } \right] = \exp\left[ - O(|\lambda|^{-3}) \right].
\end{align*}

\subsubsection*{Supercritical regime (\texorpdfstring{$\lambda > 0$}{positive})}

We break up
\[ \int_{-n^{1/3}}^\lambda \mb{P}(x \rightsquigarrow \bar x, y \rightsquigarrow \bar y) \mr dt = \left[ \int_{-n^{1/3}}^{-\lambda_0} + \int_{-\lambda_0}^{\lambda_0} + \int_{\lambda_0}^{\lambda} \right] \mb{P}(x \rightsquigarrow \bar x, y \rightsquigarrow \bar y) \mr dt. \]

By the subcritical case, the first integral is $O(n^{-2/3} \lambda_0^{-3})$.  In the last interval, the integrand is $O(n^{-2/3}t^2)$ by \cref{lemma:second-moment-bound}, and so the integral is $O(n^{-2/3}\lambda^3)$.  By monotonicity, throughout the range of the middle integral, the integrand is $O(n^{-2/3} \lambda_0^2)$ and so the middle integral is $O(n^{-2/3} \lambda_0^3)$.  Thus
\[ \mb{P}[\SAT(F_p)] \geq \exp\left[ - O(\lambda^3) \right], \]
provided $1 \ll \lambda \ll n^{1/3}$.
\end{proof}

\subsection{Random graph theory}

We spend the remainder of \cref{sec:SAT-lower-bounds} proving \cref{lemma:second-moment-bound}.  We first introduce some notation.

\begin{itemize}
    \item For $v \in X \cup Y$, recall $L^+_{n,m,p}(v)$ the distinct out-neighborhood defined in \cref{sec:2-SAT-intro}.
    \item Define $Q_{n,m,p}(k,\ell)$ to be the probability that, for a generic $x \in X$, $|L^+_{n,m,p}(x) \cap X| = k$ and $|L^+_{n,m,p}(y) \cap Y| = \ell$.
    \item Define $P_{n,m,p}(k,\ell)$ to be the probability that $|L^+_{n,m,p}(x) \cap X| = k$, $|L^+_{n,m,p}(x) \cap Y| = \ell$, and $x \notsquig{} \bar x$ (i.e.~$L^+_{n,m,p}(x)$ is the full out-neighborhood).
    \item Define $G_{k,\ell,p}$ to be the random subgraph of the complete bipartite graph $K_{k,\ell}$ where each edge is removed with probability $1-p$.
\end{itemize}

The goal of this section is to control the moments of $L^+_{n,m,p}(v)$ using the two probabilities $Q_{n,m,p}$ and $P_{n,m,p}$ by relating them to standard random graph theoretic probabilities.

\begin{lemma}\label{lem:Pnp-exact}
    For all $n,k,\ell,p$, we have
    \[ P_{n,m,p}(k,\ell) = 2^{k+\ell-1} \binom{n-1}{k-1} \binom m\ell (1-p)^{k(2m-\ell) + \ell(2n-k) - k\ell} \mb{P}(G_{k,\ell,p} \text{ is connected}). \]
\end{lemma}

\begin{proof}
    Fix  $X_0 \in \binom Xk$, $Y_0 \in \binom Y\ell$ sets of strictly distinct literals with $x \in X_0$.  Then
    \[ \mb{P}(L^+_{n,m,p}(x) = X_0 \cup Y_0) = (1-p)^{k(2m-\ell) + \ell(2n-k) - k\ell} \mb{P}(G_{k,\ell,p} \text{ is connected}), \]
    as all edges from $X_0$ to $Y \setminus Y_0$ and all edges from $Y_0$ to $X \setminus X_0$ cannot be present.  This is $k(2m-\ell) + \ell(2n-k)$ edges, but for each $x_0 \in X_0$, $y_0 \in Y_0$, we check both $x_0 \to \bar{y}_0$ and $y_0 \to \bar{x}_0$, which are the same implication.  We also must have connectivity inside $X_0 \cup Y_0$, hence the connectedness term.

    Finally, we must count the sets $X_0,Y_0$.  As a strictly distinct set, each can be chosen by first picking the variables and then choosing the signs for the literals.  Thus we have $2^{k-1} \binom{n-1}{k-1}$ choices for $X_0$ as we must include $x$ and then $2^\ell \binom m\ell$ choices for $Y_0$.
\end{proof}

We now use some enumerative results to move this into a more workable equation.

\begin{corollary}
    For all $n,k,\ell,p$, we have
    \begin{equation}\label{eq:Pnp-Sp} P_{n,m,p}(k,\ell) = \frac{1}{2np\ell} \binom nk \binom m\ell (2pk)^\ell (2p\ell)^{k} (1-p)^{2n\ell + 2mk - (k+\ell) - 2k\ell +1} S_p(k,\ell) \end{equation}
    where
    \begin{equation}\label{eq:Sp-summation-form} S_p(k,\ell) = \sum_{s=0}^{k\ell - (k+\ell-1)} \frac{C({k,\ell,s})}{k^{\ell-1} \ell^{k-1}} \paren{\frac{p}{1-p}}^s \end{equation}
    for $C(k,\ell,s)$ the number of connected (spanning) subgraphs of $K_{k,\ell}$ with exactly $k+\ell-1+s$ edges (excess $s$).

    In particular, if $k/\ell$ converges to a positive real number and $(k+\ell)^{3/2}p \to 0$, then
    \begin{equation}\label{eq:Sp-leading-order} S_p(k,\ell) = 1 + \sqrt{\frac\pi8} \frac{(k\ell(k+\ell))^{1/2} p}{1-p} \left[ 1 + O\paren{(k+\ell)^{-3/2} + (k+\ell)^{3/2} p} \right]. \end{equation}
\end{corollary}

\begin{proof}
    Trivially
    \begin{align*}
        \mb{P}(G_{k,\ell,p}\text{ is connected}) & = \sum_{s=0}^{k\ell - (k+\ell-1)} C(k,\ell,s) p^{k+\ell-1+s} (1-p)^{k\ell - (k + \ell - 1 + s)},
    \end{align*}
    and \cref{eq:Sp-summation-form} comes from reorganization.

    To motivate the definition of $S_p$, by a result of Clancy~\cite{clancy2024asymptotics}, for constant $s$ and $k,\ell \to \infty$ with $k/\ell$ converging to a positive real,
\[ C(k,\ell,s) = (1+O(k^{-3/2})) \paren{k\ell (k+\ell)}^{s/2} \rho_s k^{\ell-1} \ell^{k-1} \]
for explicit constants $\rho_s$ defined in terms of Brownian excursions (see, e.g.~\cite{janson2007brownian} for discussion and evaluation of these coefficients).  In particular, $\rho_0 = 1$ and $\rho_1 = \sqrt{\pi/8}$.

Equation \cref{eq:Sp-leading-order} comes from these two evaluations together with
\[ \frac{C(k,\ell,s)}{k^{\ell-1} \ell^{k-1} (k+\ell)^{3s/2}} \leq \exp[O(s)], \]
which is a corollary of results in \cite{do2021component}.  We will explicitly cite their theorem later as \cref{prop:upper-bound-on-bipartite-connected-graphs}; we defer its statement until we require the full strength.
\end{proof}

We now seek a comparison result between $P_{n,m,p}(k,\ell)$ and $Q_{n,m,p}(k,\ell)$, since the latter is easier to control.  For small $k+\ell$, we use the following.

\begin{lemma}\label{lem:Qnp-minus-Pnp}
    For $(k+\ell)^{3/2}p$ bounded, we have
    \begin{equation}\label{eq:Qnp-minus-Pnp} Q_{n,m,p}(k,\ell) - P_{n,m,p}(k,\ell) = \sqrt{\frac\pi8} P_{n,m,p}(k,\ell) \frac{(k\ell(k+\ell))^{1/2}p}{1-p} \left[ 1 + O\paren{(k+\ell)^{-1/2} + {(k+\ell)^{3/2}p}} \right]. \end{equation}
\end{lemma}

\begin{proof}
Note that, for $v$ a vertex on the $n$-side of $G_{n,m,2p-p^2}$,
\[ Q_{n,m,p}(k,\ell) = \mb{P}[ C(v) \text{ contains exactly } k \text{ vertices on the }n\text{-side and }\ell\text{ on the }m\text{-side} ]. \]

By the same arguments as for \cref{lem:Pnp-exact},
\begin{align*}
    Q_{n,m,p}(k,\ell) & = \binom{n-1}{k-1} \binom{m}{\ell} (1-2p+p^2)^{k(m-\ell) + \ell(n-k)} \mb{P}(G_{k,\ell,2p-p^2} \text{ is connected}) \\
    & = \frac 1{2np\ell} \binom nk \binom m\ell (k(2p-p^2))^{\ell}  (\ell(2p-p^2))^k ((1-p)^2)^{n\ell + mk - (k+\ell) - k\ell + 1} S_{2p-p^2}(k,\ell). \stepcounter{equation} \tag{\theequation} \label{eq:Qnp-Sp}
\end{align*}
Dividing this by \cref{eq:Pnp-Sp},
\begin{equation}\label{eq:Pnp-Qnp-ratio} \frac{Q_{n,m,p}(k,\ell)}{P_{n,m,p}(k,\ell)} = \paren{1-p/2}^{k+\ell} (1-p)^{-(k+\ell-1)} \frac{S_{2p-p^2}(k,\ell)}{S_p(k,\ell)} = (1+O((k+\ell)p) \frac{S_{2p-p^2}(k,\ell)}{S_p(k,\ell)}. \end{equation}
By \cref{eq:Sp-leading-order},
\[ \frac{S_{2p-p^2}(k,\ell)}{S_p(k,\ell)} = 1 + \sqrt{\frac\pi8} \frac{(k\ell(k+\ell))^{1/2}p}{1-p} \left[ 1 + O\paren{(k+\ell)^{-1/2} + \frac{(k\ell(k+\ell))^{1/2}p}{1-p}} \right], \]
giving \cref{eq:Qnp-minus-Pnp}.
\end{proof}

For large $k + \ell$, we will need a different comparison between $P_{n,m,p}(k,\ell)$ and $Q_{n,m,p}(k,\ell)$.

\begin{lemma}\label{lem:Pnp-m0}
    There is an absolute constant $c_0 > 0$ (independent of $\lambda_0,\eps_0$) such that the following holds.
    For $(k+\ell)^{3/2}p \geq 1$ and $k/\ell \in (1/5,5)$, we have
    \begin{equation} \label{eq:Pnp-m0} P_{n,m,p}(k,\ell) = O\paren{s_0 2^{-s_0} Q_{n,s,p}(k,\ell)} \end{equation}
    where
    \[ s_0 = s_0(k+\ell) = \min\left\{ c_0 \frac{(k+\ell)^3p^2}{(1-p)^2}, n^{1/5} \right\}. \]
\end{lemma}

The proof of this lemma will require the following enumerative bounds on bipartite connected graphs.
\begin{proposition}\label{prop:lower-bound-on-bipartite-connected-graphs}
    There is an absolute constant $c_1>0$ such that for any $k,\ell,s$ subject to $s \leq \frac14(k+\ell)$ and $k/\ell \in (1/5,5)$,
    \[ C(k,\ell,s) \geq k^{\ell-1} \ell^{k-1} (k+\ell)^{\frac32s} \paren{\frac{c_1}s}^{s/2}. \]
\end{proposition}

    \begin{proposition}[{\cite[Theorem 3.5\footnote{While they prove this only for $k/\ell \in (1/2,2)$, their proof techniques do not depend on $2$ and can handle any absolute constant.}]{do2021component}}]\label{prop:upper-bound-on-bipartite-connected-graphs}
        There is an absolute constant $c_2>0$ such that for any $k,\ell,s \in \mb{N}$ with $s \leq k\ell - (k+\ell-1)$ and $k/\ell \in (1/5,5)$,
        \[ C(k,\ell,s) \leq k^{\ell-1} \ell^{k-1} (k+\ell)^{\frac32s} \paren{\frac{c_2}s}^{s/2}. \]
    \end{proposition}

The upper bound is proved by Do, Erde, Kang, and Missethan in \cite{do2021component}.  We will prove the lower bound in \cref{sec:proof-of-lower-bound-bipartite}.

\begin{proof}[Proof of \cref{lem:Pnp-m0}]
    Recall by \cref{eq:Pnp-Qnp-ratio},
    \[ \frac{Q_{n,m,p}(k,\ell)}{P_{n,m,p}(k,\ell)} \leq \frac{S_{2p-p^2}(k,\ell)}{S_p(k,\ell)}, \]
    where
    \[ S_p(k,\ell) = \sum_{s=0}^{k\ell - (k+\ell-1)} \frac{C(k,\ell,s)}{k^{\ell-1} \ell^{k-1}} \paren{\frac{p}{1-p}}^s. \]

    Let $s^*$ be the index of a maximum summand in $S_p(k,\ell)$.  Then
    \begin{align*}
        S_p(k,\ell) & = \sum_{s=0}^{s^*-1} \frac{C(k,\ell,s)}{k^{\ell-1} \ell^{k-1}} \paren{\frac{p}{1-p}}^s + \sum_{s=s^*}^{k\ell - (k+\ell-1)} \frac{C(k,\ell,s)}{k^{\ell-1} \ell^{k-1}} \paren{\frac{p}{1-p}}^s \\
        & \leq s^* \frac{C(k,\ell,s^*)}{k^{\ell-1} \ell^{k-1}} \paren{\frac{p}{1-p}}^{s^*} + \sum_{s=s^*}^{k\ell - (k+\ell-1)} \frac{C(k,\ell,s)}{k^{\ell-1} \ell^{k-1}} \paren{\frac{p}{1-p}}^s \\
        & \leq (s^*+1) \sum_{s=s^*}^{k\ell - (k+\ell-1)} \frac{C(k,\ell,s)}{k^{\ell-1} \ell^{k-1}} \paren{\frac{p}{1-p}}^s \\
        & \leq 2^{-s^*} (s^*+1) \sum_{s=s^*}^{k\ell - (k+\ell-1)} \frac{C(k,\ell,s)}{k^{\ell-1} \ell^{k-1}} \paren{\frac{2p}{1-p}}^s \\
        & = 2^{-s^*} (s^*+1) \sum_{s=s^*}^{k\ell - (k+\ell-1)} \frac{C(k,\ell,s)}{k^{\ell-1} \ell^{k-1}} \paren{\frac{2p-2p^2}{(1-p)^2}}^s \\
        & \leq 2^{-s^*} (s^*+1) \sum_{s=s^*}^{k\ell - (k+\ell-1)} \frac{C(k,\ell,s)}{k^{\ell-1} \ell^{k-1}} \paren{\frac{2p-p^2}{1 - (2p-p^2)}}^s \\
        & \leq 2^{-s^*} (s^*+1) S_{2p-p^2}(k,\ell).
    \end{align*}

    Thus
    \[ P_{n,m,p}(k,\ell) \leq 2^{-s^*} (s^*+1) Q_{n,m,p}(k,\ell). \]
Note that for $s^* > 1$, the RHS is a decreasing function of $s^*$.  Thus we must only show $s^* \geq s_0$ to get \cref{eq:Pnp-m0}.

    Recalling $c_1$ from \cref{prop:lower-bound-on-bipartite-connected-graphs}, let
    \[ s_1 = c_1 \frac{p^2(k+\ell)^3}{e(1-p)^2}, \]
    so that by \cref{prop:lower-bound-on-bipartite-connected-graphs}, the $s_1$th summand is bounded below by
    \begin{equation}\label{eq:summand-s1} \frac{C(k,\ell,s_1)}{k^{\ell-1} \ell^{k-1}} \paren{\frac{p}{1-p}}^{s_1} \geq \paren{\frac{c_1 p^2 (k+\ell)^3}{s_1(1-p)^2}}^{s_1/2} = \exp\left\{ c_1 \frac{p^2(k+\ell)^3}{2e(1-p)^2} \right\}. \end{equation}

    Define
    \[ c_0 = \frac{c_1^2}{c_2}, \hspace{20pt} s_0 = c_0 \frac{p^2(k+\ell)^3}{e(1-p)^2}. \]

    We will show that the first $s_0$ terms are all at most \cref{eq:summand-s1}.  First, we compute the logarithmic derivative of the upper bound from \cref{prop:upper-bound-on-bipartite-connected-graphs}
    \[ \frac{\mr d}{\mr ds} \log\paren{\frac{c_2 p^2 (k+\ell)}{s(1-p)^2}}^{s/2} = \frac12\paren{\log\paren{\frac{c_2 p^2 (k+\ell)^3}{e (1-p)^2}} - \log s}, \]
    showing that the upper bound is increasing for all $s < s_0$ since $c_0 < c_2$.  Thus we need only check that term $s_0$ is at most the bound from \cref{eq:summand-s1}.  Indeed, using \cref{prop:upper-bound-on-bipartite-connected-graphs}
    \begin{align*}
        \frac{C(k,\ell,s_0)}{k^{\ell-1} \ell^{k-1}} \paren{\frac{p}{1-p}}^{s_0} & \leq \paren{\frac{c_2 p^2 (k+\ell)^3}{s_0(1-p)^2}}^{s_0/2} \\
        & = \exp\left\{ c_0 \frac{p^2 (k+\ell)^3}{2e (1-p)^2} \log\paren{\frac{c_2 e}{c_0}} \right\} \\
        & = \exp\left\{ c_1 \frac{p^2(k+\ell)^3}{2e(1-p)^2} \cdot 2 \frac{c_1}{c_2} \log\paren{e^{1/2} \ \frac{c_2}{c_1}} \right\}.
    \end{align*}

    The result follows from the numerical inequality $2x \log(e^{1/2}/x) < 1$, valid for all $x \in (0,1)$.
\end{proof}

\subsubsection{Moment estimates}

The goal of this next subsection is to bound $\sum_k \sum_\ell k^a P_{n,m,p}(k,\ell)$ and $\sum_k \sum_\ell \ell^a P_{n,m,p}(k,\ell)$ for several key values of $a$.  These quantities are difficult to handle, so we will instead control $\sum_k \sum_\ell (k+\ell)^a P_{n,m,p}(k,\ell)$, which is obviously an upper bound for both.

\begin{lemma}\label{lem:Qnmp-expanded-formula}
    For $j n^{-2/3}$ bounded and $j \geq 400$, we may write
    \begin{multline*} \sum_{k + \ell = j} Q_{n,m,p}(k,\ell) = \frac{1}{pn\sqrt{\pi} j^{3/2}} \exp \bigg\{ -\frac{j \eps^2}{2} \paren{1 + O(\eps) + O\paren{\frac1\lambda}} \\ + O\paren{\frac{(\log j)^{3/2}}{j^{1/2}}} + O\paren{\frac{j^3}{n^2}} + O\paren{\frac{j \log j}{n}} \bigg\}. \end{multline*}
\end{lemma}

\begin{remark}
    This approximation holds for $Q_{n,m,p}(k,\ell)$ as written.  It also holds for $P_{n,m,p}(k,\ell)$ in this range, as by \cref{eq:Qnp-minus-Pnp}, the ratio $Q_{n,m,p}(k)/P_{n,m,p}(k) = 1+o(1)$.  Let $R_{n,m,p}(k,\ell)$ be the probability that $L^+_{n,m,p}(x)$ is strictly distinct, a tree, and $|L^+_{n,m,p}(x) \cap X| = k$, $|L^+_{m,n,p}(x) \cap Y| = \ell$.  Then this approximation also holds for $R_{n,m,p}(k,\ell)$, as $R_{n,m,p}(k,\ell)$ satisfies \cref{eq:Pnp-Sp} replacing $S_p(k,\ell)$ with $1$, and by \cref{eq:Sp-leading-order}, $S_p(k,\ell) = 1+o(1)$ in this range.
\end{remark}

\begin{proof}
    We will use \cref{eq:Pnp-Sp}.  First, using Stirling's formula, $2p(mn)^{1/2} = 1+\eps$, and Taylor expanding $\log(1+\eps)$,
    \begin{align*}
        \binom nk (2p\ell)^k & = (2pne)^k \frac{(\ell/e)^k}{k!} \prod_{i=0}^{k-1} \paren{1-\frac in} \\
        & = \frac{\exp[-\Theta(1/k)]}{\sqrt{2 \pi k}} (2pne \ell/k)^k \prod_{i=0}^{k-1} \paren{1-\frac in} \\
        & = \frac{1}{\sqrt{2\pi k}} \exp \left\{ k\paren{1 + \log\frac\ell k + \eps - \frac12\eps^2 + O(\eps^3)} - \frac{k^2}{2n} - \frac{k^3}{6n^2} + O(k/n) \right\},
    \end{align*}
    and similarly,
    \[ \binom m\ell (2pk)^\ell = \frac{1}{\sqrt{2\pi \ell}} \exp \left\{ \ell \paren{1 - \log\frac\ell k + \eps - \frac12\eps^2 + O(\eps^3)} - \frac{\ell^2}{2m} - \frac{\ell^3}{6m^2} + O(\ell/m) \right\}. \]

    We can also evaluate, using $\log(1-p) = -p + O(1/n^2)$ and $2p(mn)^{1/2} = 1+\eps$,
    \begin{align*} (1-p)^{2n\ell + 2mk - (k+\ell) - 2k\ell + 1} & = \exp\left\{ -(1+\eps)(k \sqrt{m/n} + \ell \sqrt{n/m}) + \frac{(1+\eps)k\ell}{\sqrt{mn}} + O\paren{\frac{k+\ell}{n}} \right\}.
    \end{align*}
    When we put this all together, we will not get perfect cancellation.  Let us handle one term now.  Define $\beta = \frac12\log(n/m)$; note $\beta$ may be negative. 
    \begin{align*}
        & \qquad (1+\eps)(k+\ell) - (1+\eps)(k \sqrt{m/n} - \ell \sqrt{n/m}) \\
        & = (1+\eps) \left[ k \paren{1 - e^{-\beta}} + \ell \paren{1 - e^{\beta}} \right] \\
        & = (1+\eps) \left[ k \paren{\beta - \frac{\beta^2}{2} + O(\beta^3)} + \ell \paren{-\beta - \frac{\beta^2}{2} + O(\beta^3)}  \right] \\
        & = (1+\eps)\left[ \beta(k-\ell) - \frac{\beta^2}{2} (k+\ell) + O\paren{\frac{k+\ell}{n}} \right].
    \end{align*}

    The last line came as $\beta = O(n^{-1/3})$ as $|m-n| \leq C n^{2/3}$ by assumption.  Putting this all together with \cref{eq:Pnp-Sp} and noting that $\eps k \ell/\sqrt{mn} = O(\eps^2 (k+\ell)^2/(n^{2/3} \lambda)) = \eps^2(k+\ell) O(1/\lambda)$,
    \begin{multline}\label{eq:multiline-k-ell} P_{n,m,p}(k,\ell) = \frac{S_p(k,\ell)}{4\pi \ell \sqrt{nmk\ell}} \exp\bigg\{ -\frac{(k+\ell)\eps^2}{2}\paren{1+O(\eps) + O\paren{\frac1\lambda}} - \frac{k^3}{6n^2} - \frac{\ell^3}{6m^2} \\ - \frac{(k\sqrt m-\ell \sqrt n)^2}{2mn} + (1+\eps)\beta (k-\ell) - \frac{\beta^2}{2} (k+\ell) + (k-\ell)\log\frac\ell k + O\paren{\frac{k+\ell}{n}}  \bigg\} . \end{multline}

    Note that \cref{eq:multiline-k-ell} holds for any $k,\ell$; we have not yet assumed any bounds.  Assuming $(k+\ell)/n^{2/3}$ is bounded, many of these terms are easily controlled.
    We now reparametrize $i = (k+\ell)/2$, $d = (k-\ell)/2$ to get
    \[ \sum_{k + \ell = 2i} P_{n,m,p}(k,\ell) = \sum_{d = -10\sqrt{i \log i}}^{10\sqrt{i \log i}} P_{n,m,p}(i+d,i-d) + {\sum_{|d| > 10 \sqrt{i \log i}}} P_{n,m,p}(i+d,i-d). \]

    We will begin by showing that the second sum does not contribute.  We will go through \cref{eq:multiline-k-ell} term-by-term.  We will trivially upper bound the leading fraction by 1, and the first two terms in the exponential by 0.  This gives
    \begin{equation}\label{eq:Pnmp-trivial-upper-bound} P_{n,m,p}(k,\ell) \leq \exp\left[ (1+\eps)2 \beta d + 2d \log\paren{1 - \frac{2d}{i-d}} + O\paren{\frac{i}{n^{2/3}}}  \right]. \end{equation}

    Applying $\partial/\partial d$ to the logarithm of \cref{eq:Pnmp-trivial-upper-bound}, we get
    \[ 2\beta(1+\eps) + 2 \log\paren{1 - \frac{2d}{i-d}} - \frac{4id}{i^2-d^2}. \]

    This is symmetric and decreases as $|d|$ grows.  Further, evaluating at $d = 10\sqrt{i \log i}$,
    \[ 2\beta(1+\eps) - 4\log i + O\paren{\frac{(\log i)^{3/2}}{i^{1/2}}} \]
    is already negative as $2\beta(1+\eps) = O(n^{-1/3})$ by assumption.  By the same logic, by $d = -10\sqrt{i \log i}$, the derivative is positive and so further decreasing $d$ will decrease the function.  Thus the maximum of \cref{eq:Pnmp-trivial-upper-bound} for $|d| \geq 10\sqrt{i \log i}$ is when $d = \pm10\sqrt{i \log i}$.  We bound for large $d$
    \begin{align*}
        & \qquad P_{n,m,p}(i+d,i-d) \\ & \leq \exp\left\{ (1+\eps)2\beta O(\sqrt{i \log i}) + 20 \sqrt{i \log i} \log\paren{1 - \frac{20 \sqrt{i \log i}/i}{1 - 10\sqrt{i \log i}/i}} + O \paren{\frac{i}{n^{2/3}}} \right\} \\
        & \leq \exp\left\{ - \frac{400 \log i}{1 - 10\sqrt{(\log i)/i
        }} + O\paren{\frac{i^{1/2}}{n^{1/3}} (\log i)^{1/2}} \right\} \leq i^{-200}.
    \end{align*}

    The second line used the numerical inequality $\log(1+x) \leq x$ and the assumption $\beta = O(n^{-1/3})$.  The final inequality holds for $i > 200$.  Thus
    \begin{equation}\label{eq:large-d-no-contribution} \sum_{10\sqrt{i \log i} \leq |d| \leq i} P_{n,m,p}(i+d,i-d) \leq 2i^{-199}. \end{equation}

    We now deal with the main term
    \[ \sum_{d = -10\sqrt{i \log i}}^{10\sqrt{i \log i}} P_{n,m,p}(i+d,i-d). \]

    For this range of $d$, \cref{eq:multiline-k-ell} becomes
    \begin{multline*} P_{n,m,p}(i+d,i-d) = \frac{1+O(i^{-1/2})}{4n\pi i^2} \exp \bigg\{ -i \eps^2 \paren{1 + O(\eps) + O\paren{\frac1\lambda}} \\ - O\paren{\frac{i^2}{n^{5/3}}} + (1+\eps)2 \beta d - \frac{4d^2}{i} + O\paren{\frac{(\log i)^{3/2}}{i^{1/2}}} + O\paren{\frac{i^3}{n^2}} + O\paren{\frac{i \log i}{n}} \bigg\}. \end{multline*}

    We now estimate the sum by an integral. By, e.g., \cite[Theorem 4.3]{spencer2014asymptopia} and a standard $u$-substitution,
    \begin{align*}
        \sum_{d = -10\sqrt{i \log i}}^{10 \sqrt{i \log i}} \exp\left\{ (1+\eps)2\beta d -\frac{4d^2}{i} \right\} & = \int_{-10\sqrt{i \log i}}^{10 \sqrt{i \log i}} \exp\left\{ \alpha d -\frac{4d^2}{i} \right\} \mr{d}d + O\paren{\max \exp\left\{ \alpha d - \frac{4d^2}{i} \right\}} \\
        & = \paren{1+O(i^{-50})} \sqrt{\frac{\pi i}{2}} + \exp[n^{-\Omega(1)}].
    \end{align*}
    Thus
    \begin{multline*} \sum_{d = -10 \sqrt{i \log i}}^{10 \sqrt{i \log i}} P_{n,m,p}(i+d,i-d) = \frac{1}{2n(2\pi)^{1/2} i^{3/2}} \exp \bigg\{ -i \eps^2 \paren{1 + O(\eps) + O\paren{\frac1\lambda}} \\ + O\paren{\frac{(\log i)^{3/2}}{i^{1/2}}} + O\paren{\frac{i^3}{n^2}} + O\paren{\frac{i \log i}{n}} \bigg\}.  \end{multline*}

    Combining with \cref{eq:large-d-no-contribution} and changing variables to $j = 2i$ gives the desired result.
\end{proof}

In addition, we will want to control the case where $k$ and $\ell$ are far apart.

\begin{lemma}\label{lem:k/ell-bounded}
    Suppose $k/\ell \notin (1/5,5)$.  Then $P_{n,p}(k,\ell) = \exp[-\Omega(k+\ell)]$.
\end{lemma}

\begin{proof}
    Recall \cref{eq:multiline-k-ell} holds for any $k,\ell$.  Applying some trivial inequalities
    \begin{equation}\label{eq:k/ell-bound-with-k-ell}
        P_{n,m,p}(k,\ell) \leq S_p(k,\ell) \exp\left\{ (1+\eps)\beta (k-\ell) + (k-\ell) \log \frac\ell k + O(1) \right\}.
    \end{equation}

    We now note
    \begin{align*}
        S_p(k,\ell) & = p^{-(k+\ell-1)} (1-p)^{-k\ell + (k+\ell-1)} \frac{1}{k^{\ell-1} \ell^{k-1}} \mb{P}[G_{k,\ell,p} \text{ is connected}] \\
        & \leq p^{-(k+\ell-1)} (1-p)^{-k\ell + (k+\ell-1)} \frac{1}{k^{\ell-1} \ell^{k-1}} p^{k+\ell-1} C(k,\ell,0) \\
        & \leq \exp\left\{ p k\ell \right\}.
    \end{align*}

    The second line comes as all connected graphs must contain a tree.  The last line comes as $C(k,\ell,0) = k^{\ell-1} \ell^{k-1}$.

    Combining this with \cref{eq:k/ell-bound-with-k-ell} and reparametrizing with $i = (k+\ell)/2$, $d = (k-\ell)/2$, we get
    \[ P_{n,m,p}(i+d,i-d) \leq \exp\left\{ 2(1+\eps) \beta d + 2d \log \paren{\frac{i - d}{i + d}} + (i^2 - d^2) p + O(1) \right\}. \]

    By assumption, $|d| \geq 3i/5$.  This gives the bound
    \[ P_{n,m,p}(i+d,i-d) \leq \exp\left\{ 2(1+\eps) |\beta| i - \frac{6i}{5} \log 4 + \frac{7(1+\eps)}{8} \frac{i^2}{(mn)^{1/2}} + O(1) \right\}. \]

    Note that $i \leq (m+n)/2 \leq \frac87(mn)^{1/2}$ as $|m-n| = O(n^{2/3})$.  Thus as $\beta = O(n^{-1/3})$,
    \[ P_{n,m,p}(i+d,i-d) \leq \exp\left\{ \paren{- \frac{6 \log 4}{5} + (1+\eps) + O(n^{-1/3}) } i + O(1) \right\} = \exp[-\Omega(i)] \]
    for sufficiently large $n$ and small $\eps_0$.
\end{proof}

We will now bound the moments of $Q_{n,m,p}$, towards bounding the moments of $P_{n,m,p}$.

\begin{lemma}\label{lem:Qnmp-sum-k/ell-bounded}
    There exist constants $c,\eps_0,\lambda_0$ such that the following hold for $\lambda_0 < |\lambda| < \eps_0 n^{1/3}$.
    \begin{enumerate}
        \item If $\lambda < 0$, then
        \begin{equation} \label{eq:moment-large-k} \sum_{k + \ell \geq n^{2/3}/|\lambda|} (k+\ell)^a Q_{n,m,p}(k,\ell) = O\paren{\left[ \frac{2}{\eps^2} \right]^{a-1/2} e^{-c|\lambda|}}, \end{equation}
        where the implicit constant depends on $a$.
        \item For both positive and negative $\lambda$, we have
        \begin{equation}\label{eq:moment-small-k} \sum_{k + \ell \leq n^{2/3}/|\lambda|} (k+\ell)^a Q_{n,m,p}(k,\ell) = (1 + o_{\eps,\lambda}(1)) \left[ \frac{2}{\eps^2} \right]^{a-1/2}, \end{equation} 
        where the implicit constants again depend on $a$.
    \end{enumerate}
\end{lemma}

We will only need to apply this lemma for $a=1,a=3/2,a=2$, so we have no worries about $a$ dependence.

\begin{proof}
    To show the first bound, notice that the cluster size in $G(n,m,\widetilde{p})$ is stochastically dominated by a Poisson birth process with parameter $\max\{m,n\} \log(1/(1-\widetilde{p}))$, giving
    \begin{align*}
        & \qquad \sum_{k+\ell \geq n^{2/3}/|\lambda|} (k+\ell)^a Q_{n,m,p}(k,\ell) \\ & \leq \frac{1+O(1/n)}{2p\max\{m,n\}\sqrt{2\pi}} \sum_{k+\ell \geq n^{2/3}/|\lambda|} (k+\ell)^{a-3/2} \exp\left[ -\frac{(k+\ell)\eps^2}{2}(1+o_n(1)) \right] \\
        & = O\paren{\left[ \frac2{\eps^2} \right]^{a-1/2} e^{-c|\lambda|}},
    \end{align*}
    using \cite[eq.~(6.12)]{bollobas2001scaling}.

    For \cref{eq:moment-small-k}, we will need to break up the sum into two regions of ``very small'' and ``intermediate'' $k+\ell$.  In the very small region, the sum is
    \begin{equation}\label{eq:small-j-no-contribution} \sum_{k+\ell < 1/|\eps|} (k+\ell)^a Q_{n,m,p}(k,\ell) = O((1/\eps)^{a-1/2}). \end{equation}

    In the ``intermediate'' region, we now write
    \[ \sum_{j = 1/|\eps|}^{n^{1/3}/|\eps|} j^a \sum_{k+\ell = j} Q_{n,m,p}(k,\ell). \]

    Inside this range of $j$, we will use \cref{lem:Qnmp-expanded-formula}.  The additive errors inside the exponential can be controlled as can the multiplicative.  This gives, for $\delta = o_{\eps,\lambda}(1)$,
    \begin{equation}\label{eq:B(delta)} \frac{1+o_{\eps,\lambda}(1)}{2pn \sqrt{2\pi}} B(2-\delta) \leq \sum_{j = 1/|\eps|}^{n^{1/3}/|\eps|} j^a \sum_{k+\ell=j} Q_{n,m,p}(k,\ell) \leq \frac{1+o_{\eps,\lambda}(1)}{2pn \sqrt{2\pi}} B(2+\delta), \end{equation}
    where
    \[ B(t) = \sum_{k=1/|\eps|}^{n^{1/3}/|\eps|} j^{a-3/2} \exp\left\{- j\eps^2 \right\}. \]

    We may approximate the sum with an integral (see, e.g.~\cite[Theorem 3.4]{spencer2014asymptopia})
    \begin{align*}
        B(t) & = \int_{1/|\eps|}^{n^{1/3}/|\eps|} j^{a-3/2} \exp\left\{ -\frac{j\eps^2}{t} \right\} \mr{d}j + O\paren{\max j^{a-3/2} \exp\left\{ -\frac{k \eps^2}{t} \right\} } \\
        & = \left[ \frac{t}{\eps^2} \right]^{a - 1/2} \int_{|\eps|/t}^{|\lambda|/t} u^{a-3/2} \exp\{-u\} \mr{d}u + O\paren{\left[ \frac{1}{\eps^2} \right]^{a - 3/2}} \\
        & = \left[ \frac{t}{\eps^2} \right]^{a - 1/2} \paren{\Gamma(a - 1/2) + o_{\eps,\lambda}(1)}.
    \end{align*}

    Combining with \cref{eq:B(delta),eq:small-j-no-contribution} completes the proof of \cref{eq:moment-small-k}.
\end{proof}

We will now need the following standard results on the asymptotics of $Q_{n,m,p}(k,\ell)$.

\begin{lemma}\label{lem:Qnp-bounds}
    There are constants $c,\eps_0,\lambda_0$ such that for $\lambda_0 < \lambda < \eps_0 n^{1/3}$ and $p = (1+\lambda n^{-1/3})/(2n)$, we have
    \[ \sum_{k+\ell \geq \lambda n^{2/3}} Q_{n,m,p}(k,\ell) \leq (2+o_{\lambda,\eps}(1)) \eps \] 
    and
    \[ \sum_{k + \ell \in (n^{2/3}/\lambda, \lambda n^{2/3})} Q_{n,m,p}(k) = \exp\{-\Omega(\lambda) \} n^{-1/3}. \]
\end{lemma}

Note that this result is only for the weakly supercritical regime $\lambda > 0$.  While various papers~\cite{bollobas1984evolutionrandom,janson1993birth,luczak1990component} have proved sharper results for non-bipartite random graphs for smaller regimes of $\lambda$, this statement for the whole range of $\lambda$ is not in the literature.  As the proof is rather technical and closely follows that of \cite[Lemma 5.5]{bollobas2001scaling}, we leave the proof to \cref{app:branching-process}.

\begin{lemma}\label{lem:moments-of-Pnp}
    There are constants $\eps_0,\lambda_0$ such that for any fixed $a > 1/2$ and $\lambda_0 < |\lambda| < \eps_0 n^{1/3}$, we have
    \[ \sum_k \sum_\ell (k+\ell)^a P_{n,m,p}(k,\ell) = \frac{1+o(1)}{2pn} \frac{\Gamma(a-1/2)}{\sqrt{2\pi}} \left[ \frac{2}{\eps^2} \right]^{a-1/2} \]
    and
    \[ \sum_{k + \ell \geq n^{2/3}/|\lambda|} (k+\ell)^a P_{n,m,p}(k,\ell) = \left[ \frac{1}{\eps^2} \right]^{a-1/2} \exp\left\{ - \Omega(|\lambda|^{3/5}) \right\}. \]
\end{lemma}

\begin{proof}
    We will consider several regions.

    First, for any $\lambda$, for $k + \ell \leq 2n^{2/3}/|\lambda|$, combining \cref{eq:Qnp-minus-Pnp} and \cref{eq:moment-small-k} gives
    \[\sum_{k + \ell \leq n^{2/3}/|\lambda|} (k+\ell)^a P_{n,m,p}(k,\ell) = \frac{1 + o_{\eps,\lambda}(1)}{2pn} \left[ \frac{2}{\eps^2} \right]^{a-1/2}. \]

    For $\lambda < 0$, we will show there is no contribution outside this interval.  We use the trivial bound $P_{n,m,p}(k,\ell) \leq Q_{n,m,p}(k,\ell)$ and \cref{eq:moment-large-k} to get $o_{\lambda,\eps}(1)$ contribution from larger $k+\ell$.

    For $\lambda > 0$, the situation is more delicate and we must break into two further regions.  First, we have
    \begin{align*}
        \sum_{n^{2/3}/\lambda \leq k+\ell \leq n^{2/3}\lambda} (k+\ell)^a P_{n,m,p}(k,\ell) & \leq (n^{2/3}\lambda)^a \sum_{n^{2/3}/\lambda \leq k+\ell \leq n^{2/3}\lambda} Q_{n,m,p}(k,\ell) \\
        & \leq (n^{2/3}\lambda)^a \exp[-\Theta(\lambda)] n^{-1/3} \\
        & = \left[ \frac{1}{\eps^2} \right]^{a-1/2} \times \exp[-\Theta(\lambda)].
    \end{align*}

    The second line uses \cref{lem:Qnp-bounds}.  Thus we have $o_\lambda(1)$ contribution from this middle region.

    Finally, we must handle $k+\ell$ large for $\lambda>0$.  We begin by rewriting
    \[ \sum_{j \geq n^{2/3}\lambda} j^a \left[ \sum_{\substack{k + \ell = j \\ k/\ell \in (1/5,5)}} + \sum_{\substack{k+\ell = j \\ k/\ell \notin (1/5,5)}} \right] P_{n,m,p}(k,\ell). \]

    By \cref{lem:k/ell-bounded},
    \[ \sum_{k + \ell \geq n^{2/3}|\lambda|} (k+\ell)^a P_{n,m,p}(k,\ell) = \sum_{j \geq n^{2/3}|\lambda|} j^a \paren{ \exp[-\Omega(j)] + \sum_{\substack{k + \ell = j \\ k/\ell \in (1/5,5)}} P_{n,m,p}(k,\ell)}. \]

    By \cref{eq:Pnp-m0},
    \[  \sum_{j \geq n^{2/3}|\lambda|} j^a \sum_{\substack{k + \ell = j \\ k/\ell \in (1/5,5)}} P_{n,m,p}(k,\ell) \leq \sum_{j \geq n^{2/3} |\lambda|} j^a (s_0+1) 2^{-s_0} \sum_{\substack{k + \ell = j \\ k/\ell \in (1/5,5)}} Q_{n,m,p}(k,\ell). \]
    
    Note that we can factor it like this as $s_0$ is only a function of $j$.  Call $j^a s_0 2^{-s_0}$ the prefactor; we will give this a universal upper bound.  Provided $s_0 (j+1) < n^{1/5}$, recalling the definition of $s_0$, we can differentiate
    \[ \frac{\mr{d}}{\mr{d}j} \log\paren{ j^a \paren{c_0 \frac{j^3p^2}{(1-p)^2} + 1} 2^{-c_0 j^3p^2(1-p)^{-2}} } = \frac{a+3}{j} - 3j^2 c_0 p^2 (1-p)^{-2} \log 2. \]

    We will choose $\lambda_0 = \lambda_0(c_0)$ so that this derivative is nonpositive for $j>n^{2/3} |\lambda_0|$.  Thus the largest value of the prefactor occurs at either $j = \lambda_0$ or when $s_0 = n^{1/5}$.  In this range, we can see that the prefactor is $j^a n^{1/5} 2^{-n^{1/5}}$.  This is clearly increasing with $j$, so the max is at $j = m+n$, where we will get $n^{1/5+a} 2^{-n^{1/5}}$.  Thus the prefactor is at most
    \[ \max\left\{ c_0 (n^{2/3}\lambda)^a 2^{-\Theta(\lambda^3)}, n^{a+1/5} 2^{-n^{1/5}} \right\}. \]

    Since $|\lambda| \leq n^{1/3}$, we may rewrite this
    \[ \max\left\{ n^{(2/3)a} \exp[-\Theta(\lambda^3)], n^{a+1/5} \exp[-\Theta(n^{1/5}) - \Omega(\lambda^{3/5})] \right\} \leq n^{(2/3)a} \exp[-\Omega(\lambda^{3/5})]. \]

    Thus this is an absolute bound for the prefactor.  We now rewrite the above equation
    \begin{align*}
        \sum_{j \geq n^{2/3} |\lambda|} j^a \sum_{\substack{k + \ell = j \\ k/\ell \in (1/5,5)}} P_{n,m,p}(k,\ell) & \leq n^{(2/3)a} \exp[-\Omega(\lambda^{3/5})] \sum_{j \geq n^{2/3} |\lambda|} \sum_{\substack{k + \ell = j \\ k/\ell \in (1/5,5)}} Q_{n,m,p}(k,\ell) \\
        & \leq n^{(2/3)a - 1/3} \exp[-\Omega(\lambda^{3/5})],
    \end{align*}
    where the last line uses \cref{lem:Qnp-bounds}.
\end{proof}

\subsection{First moment of the spine}

We now show the following result.
\begin{proposition}\label{prop:spine-expectation}
    We have
    \[ \mb{P}\paren{x \squig{n,m} \bar x} = \begin{cases} O(n^{-1/3} |\lambda|^{-2}) & \lambda < 0, \\ O(n^{-1/3} |\lambda|) & \lambda > 0. \end{cases} \]
\end{proposition}

\begin{proof}
We have
\begin{align*}
    \mb{P}(x \squig{n,m} \bar x) & = 1 - \sum_k \sum_\ell P_{n,m,p}(k,\ell) = \sum_k \sum_\ell (Q_{n,m,p}(k,\ell) - P_{n,m,p}(k,\ell)) \\
    & = \left[ \sum_{k + \ell \leq n^{2/3}/|\lambda|} + \sum_{n^{2/3}/|\lambda| \leq k + \ell \leq n^{2/3} |\lambda|} + \sum_{k+\ell \geq n^{2/3}|\lambda|} \right] (Q_{n,m,p}(k,\ell) - P_{n,m,p}(k,\ell)).
\end{align*}

We will control the sums in all three ranges.

First, for small $k+\ell$, we break up
\begin{align*}
    \sum_{k + \ell \leq n^{2/3}/|\lambda|} (Q_{n,m,p}(k,\ell) - P_{n,m,p}(k,\ell)) & = \left[ \sum_{\substack{k + \ell \leq n^{2/3}/|\lambda| \\ k/\ell \in (1/5,5)}} + \sum_{\substack{k + \ell \leq n^{2/3}/|\lambda| \\ k/\ell \notin (1/5,5)}} \right] (Q_{n,m,p}(k,\ell) - P_{n,m,p}(k,\ell)) \\
    & = \sum_{\substack{k + \ell \leq n^{2/3}/|\lambda|}} (1+o_{\eps,\lambda}(1)) \sqrt{\frac\pi8} P_{n,m,p}(k,\ell) \frac{(k\ell(k+\ell))^{1/2}p}{1-p} \\
    & \leq (1+o_{\eps,\lambda}(1)) p\sqrt{\frac\pi8} \sum_{\substack{k + \ell \leq n^{2/3}/|\lambda|}} (k+\ell)^{3/2} P_{n,m,p}(k,\ell) \\
    & \leq (1+o_{\eps,\lambda}(1)) p\sqrt{\frac\pi8} \times \frac{1+o(1)}{2pn} \frac{1}{\sqrt{2\pi}} \left[ \frac2\eps \right] \\
    & = (1+o_{\eps,\lambda}(1)) \frac{1}{4n^{1/3} \lambda^2}, \label{eq:spine-moment-j-small} \addtocounter{equation}{1}\tag{\theequation}
\end{align*}
where the second line uses \cref{lem:k/ell-bounded,eq:Qnp-minus-Pnp}.  We now specialize to each regime.

\subsubsection*{Subcritical regime (\texorpdfstring{$\lambda<0$}{negative})}  In this regime, we will combine the two regions where $k+\ell$ is larger.

\begin{align*}
    0 & \leq \sum_{k+\ell \geq n^{2/3}/|\lambda|} (Q_{n,m,p}(k,\ell) - P_{n,m,p}(k,\ell)) \leq \sum_{k+\ell \geq n^{2/3}/|\lambda|} Q_{n,m,p}(k,\ell) \\
    & \leq \frac{|\lambda|}{n^{2/3}} \sum_{k+\ell \geq n^{2/3}/|\lambda|} (k+\ell) Q_{n,m,p}(k,\ell) = O\paren{n^{-1/3} \exp\{-\Omega(|\lambda|) \}},
\end{align*}
where the last line uses \cref{lem:Qnmp-sum-k/ell-bounded}.

Combining this with \cref{eq:spine-moment-j-small} gives
\[ \mb{P}\paren{x \squig{n,m} \bar x} \leq  \frac{1+o_{\eps,\lambda}(1)}{4n^{1/3} \lambda^2} \]
when $\lambda < 0$.

\subsubsection*{Supercritical regime (\texorpdfstring{$\lambda > 0$}{positive})}

For the middle region of $k+\ell$, we bound
\[ 0 \leq \smashoperator{ \sum_{n^{2/3}/|\lambda| \leq k + \ell \leq n^{2/3} |\lambda|} } (Q_{n,m,p}(k,\ell) - P_{n,m,p}(k,\ell)) \leq \smashoperator{\sum_{n^{2/3}/|\lambda| \leq k + \ell \leq n^{2/3} |\lambda|}} Q_{n,m,p}(k,\ell) = n^{-1/3} \exp[-\Omega(\lambda)], \]
which will not contribute.

For large $k+\ell$, the proof of \cref{lem:moments-of-Pnp} shows that $P_{n,m,p}(k,\ell) = Q_{n,m,p}(k,\ell) \exp[-\Omega(\lambda^{3/5})]$.  Thus
\begin{align*}
    \sum_{k + \ell > n^{2/3} \lambda} (Q_{n,m,p}(k,\ell) - P_{n,m,p}(k,\ell)) & = (1 - \exp[-\Omega(\lambda^{3/5})]) \sum_{k + \ell > n^{2/3} \lambda} Q_{n,m,p}(k,\ell) \\
    & \leq (2 + o_{\lambda,\eps}(1)) \lambda n^{-1/3},
\end{align*}
where the second line uses \cref{lem:Qnp-bounds}.

This is the dominant term, so
\[ \mb{P} \paren{x \squig{n,m} \bar x} \leq (2 + o_{\lambda,\eps}(1)) \lambda n^{-1/3} \]
when $\lambda > 0$. 
\end{proof}

\subsection{Second moment of the spine}

Armed with these moments for $L^+_{n,m,p}(x)$, we can now approximate $\mb{P}(x \rightsquigarrow \bar x, y \rightsquigarrow \bar y)$ as desired.

\begin{remark}
    Throughout all of the previous sections, we have various results on $P_{n,m,p}(k,\ell)$.  In this section, at points, we will need to control instead $P_{m,n,p}(k,\ell)$.  Note that $P_{n,m,p}(k,\ell)$ is not symmetric in any way because the initial vertex is on the side with $n$ vertices.  However, our only assumption on $m$ and $n$ has been $|n-m| \leq C n^{2/3}$.  This condition is close to symmetric, and implies for a different constant $C'$ that $|n-m| \leq C' m^{2/3}$.  We can choose $\lambda_0 = \lambda_0(C')$ (resp.~$\eps_0 = \eps_0(C')$) so that all earlier computations hold for $P_{m,n,p}(k,\ell)$ as well.

    We will also at points need to deal with $P_{n-i,m-j,p}(k,\ell)$ where $i,j \leq n^{2/3}/\lambda_0$.  Similarly, this will only change $C$ by some (arbitrarily small) constant amount.  We may choose $\lambda_0$ (resp.~$\eps_0$) sufficiently large (resp.~small) that all earlier computations also hold for $n,m$ shifted by $n^{2/3}/\lambda_0$.
\end{remark}

The following result says that the FKG inequality is approximately tight.

\begin{lemma}\label{lem:FKG-for-Pxx}
    We have
    \[ \mb{P}(x \squig{n,m} \bar x, y \squig{m,n} \bar y) - \mb{P}(x \squig{n,m} \bar x) \mb{P}(y \squig{m,n} \bar y) = \begin{cases} O(n^{-2/3} |\lambda|^{-4}) & \lambda < 0, \\ O(n^{-2/3} |\lambda|^{-1}) & \lambda > 0. \end{cases} \]
\end{lemma}

To begin moving towards this lemma, define the event $\hat{P}_{n,m,p}(k,\ell;z)$ to mean that literal $z$ is a member of a part $Z$ with $n$ variables while the other part has $m$ variables, that $|L^+(z) \cap Z| = k$, $|L^+(z) \setminus Z| = \ell$, and that $L^+_{n,m,p}(z)$ is strictly distinct.

We will decompose
\[ 1 = \mathbbm{1}_{y \squig{m,n} \bar y} + \sum_k \sum_\ell \mathbbm{1}_{\hat{P}_{m,n,p}(k,\ell;y)}, \]
giving
\begin{multline*}
    \mb{P}(x \squig{n,m} \bar x, y \squig{m,n} \bar y) + \sum_k \sum_\ell \mb{P}(x \squig{n,m} \bar x, \hat{P}_{m,n,p}(k,\ell;y)) \\ = \mb{P}(x \squig{n,m} \bar x) \paren{ \mb{P}(y \squig{m,n} \bar y) + \sum_k \sum_\ell P_{m,n,p}(k,\ell) }, \end{multline*} which we may reorganize into\begin{multline*}
    \mb{P}(x \squig{n,m} \bar x, y \squig{m,n} \bar y) - \mb{P}(x \squig{n,m} \bar x) \mb{P}(y \squig{m,n} \bar y) \\= \sum_k \sum_\ell \paren{\mb{P}(x \squig{n,m} \bar x) - \mb{P}(x \squig{n,m} \bar x \,|\, \hat{P}_{m,n,p}(k,\ell;y))} P_{m,n,p}(k,\ell).
\end{multline*}

We may lower bound
\[ \mb{P}(x \squig{n,m} \bar x \,|\, \hat{P}_{m,n,p}(k,\ell;y)) \geq \frac{n-\ell}{n} \mb{P}(x \squig{n-\ell,m-k} \bar x). \]

This comes as there is a $\ell/n$ chance that one of $x$ and $\bar x$ is in $L^+(y)$.  In this case, any path $x \rightsquigarrow \bar x$ must avoid $L^+(y)$, so we may as well explore $G \setminus L^+(y)$ instead, which is a bipartite graph on $(n-\ell) + (m-k)$ vertices.

Thus we get
\begin{align*}
    & \qquad \mb{P}(x \squig{n,m} \bar x, y \squig{m,n} \bar y) - \mb{P}(x \squig{n,m} \bar x) \mb{P}(y \squig{m,n} \bar y) \\
    & \leq \sum_k \sum_\ell P_{m,n,p}(k,\ell) \paren{\mb{P}(x \squig{n,m} \bar x) - \frac{n-\ell}{n} \mb{P}(x \squig{n-\ell,m-k} \bar x)} \\
    & = \mb{P}(x \squig{n,m} \bar x) \sum_k \sum_\ell \frac{\ell}{n} P_{m,n,p}(k,\ell) + \sum_k \sum_\ell \frac{n-\ell}{n} P_{m,n,p}(k,\ell) \paren{\mb{P}(x \squig{n,m} \bar x) - \mb{P}(x \squig{n-\ell,m-k} \bar x)} \\
    & \leq \frac1n \mb{P}(x \squig{n,m} \bar x) \sum_k \sum_\ell P_{m,n,p}(k,\ell) \ell + \sum_k \sum_\ell P_{m,n,p}(k,\ell) \paren{\mb{P}(x \squig{n,m} \bar x) - \mb{P}(x \squig{n-\ell,m-k} \bar x)}. \addtocounter{equation}{1}\tag{\theequation}\label{eq:FKG-n-minus-ell}
\end{align*}

We already have the machinery built up to handle the first term, but the second term will be dominant in both regimes.  This lemma allows us to control the difference term.

\begin{lemma}\label{lem:inductive-step-for-x-to-x}
    We have
    \[ \mb{P}(x \squig{n+1,m} \bar x) - \mb{P}(x \squig{n,m} \bar x) = \begin{cases} O(1/(n|\lambda|^3)) & \lambda < 0, \\ O(1/n) & \lambda > 0, \end{cases} \]
    and
    \[ \mb{P}(x \squig{n,m+1} \bar x) - \mb{P}(x \squig{n,m} \bar x) = \begin{cases} O(1/(n|\lambda|^3)) & \lambda < 0, \\ O(1/n) & \lambda > 0. \end{cases} \]
\end{lemma}

\begin{proof}
We will use a coupling of these events.  Given a sample on vertex set $X \cup Y$ of $2n + 2m$ vertices, we will add vertices $x_{n}$ and $\bar x_{n+1}$ and extend to a formula on the new vertex set.  Thus the events $x \squig{n,m} \bar x$ and $x \squig{n+1,m} \bar x$ have been coupled so that
\[ \mb{P}(x \squig{n+1,m} \bar x) - \mb{P}(x \squig{n,m} \bar x) = \mb{P}(x \squig{n+1,m} \bar x, \ x \notsquig{n,m} \bar x). \] 

Suppose $x \notsquig{n,m} \bar x$.  Let $Z = L_{n,m,p}^+(x)$, which must be strictly distinct.

\textbf{Case 1.} $Z \not\to x_{n+1}$ and $Z \not\to \bar x_{n+1}$.  Then $x \notsquig{n+1,m} \bar x$.

\textbf{Case 2.} $Z \to x_{n+1}$ and $Z \to \bar x_{n+1}$.  Then $x \squig{n+1,m} \bar x$.

\textbf{Case 3.} $Z \to x_{n+1}$ and $Z \not\to \bar x_{n+1}$.  In this case, we must have $x_{n+1} \squig{V \setminus Z} \bar x_{n+1}$.  Thus here we get (if $|Z \cap X| = k$ and $|Z \cap Y| = \ell$)
\[ \mb{P}(x \squig{n+1-k,m-\ell} \bar x). \]

\textbf{Case 4.} $Z \not\to x_{n}$ and $Z \to \bar x_{n}$.  By symmetry this is Case 3 again.

Thus, for $Z$ strictly distinct with $|Z \cap X| = k$ and $|Z \cap Y| = \ell$,
\begin{align*}
    \mb{P}(x \squig{n,m} \bar x \,|\, L_{n,m,p}^+(x) = Z) & = [1 - (1-p)^\ell]^2 + 2[1 - (1-p)^\ell](1-p)^\ell \mb{P}(x \squig{n+1-k,n-\ell} \bar x) \\
    & \leq p^2 \ell^2 + 2p\ell \, \mb{P}(x \squig{n+1-k,m-\ell} \bar x).
\end{align*}

We now get
\begin{align*}
    \mb{P}(x \squig{n+1,m} \bar x) - \mb{P}(x \squig{n,m} \bar x) & = \sum_{Z} \mb{P}(L_{n,m,p}^+(x) = Z) \mb{P}(x \squig{n+1,m} \bar x \,|\, L_{n,m,p}^+(x) = Z) \\
    & \leq \sum_k \sum_\ell P_{n,m,p}(k,\ell) \left[ p^2 \ell^2 + 2p\ell \, \mb{P}(x \squig{n+1-k,m-\ell} \bar x) \right] \\
    & \leq p^2 \sum_k \sum_\ell P_{n-1,m,p}(k,\ell) \ell^2 + 2p\, \mb{P}(x \squig{n,m} \bar x) \sum_k \sum_\ell P_{n,m,p}(k,\ell) \ell \\
    & \leq \frac{1+o_\eps(1)}{n^2} \frac{1}{|\eps|^3} + \frac2n \mb{P}(x \squig{n,m} \bar x) \frac{1+o_\eps(1)}{|\eps|} \\
    & = \frac{1+o_\eps(1)}{n |\lambda|^3} + \frac{2+o_{\eps,\lambda}(1)}{n^{2/3} |\lambda|}\mb{P}(x \squig{n,m} \bar x),
\end{align*}
where the fourth line uses \cref{lem:moments-of-Pnp}.

Now, \cref{prop:spine-expectation} gives us
\[ \mb{P}(x \squig{n+1,m} \bar x) - \mb{P}(x \squig{n,m} \bar x) = \begin{cases} O(1/(n|\lambda|^3)) & \lambda < 0, \\ O(1/n) & \lambda > 0. \end{cases} \]

    The other argument is almost identical.  For this case, we will be adding $y_{m+1}$ and $\bar y_{m+1}$ to $X \cup Y$, coupling the distributions, and examining $\mb{P}(x \squig{n,m+1} \bar x \setminus x \squig{n,m} \bar x)$.

    \textbf{Case 1.} $Z \not\to y_{m+1}$ and $Z \not\to \bar y_{m+1}$.  Then $x \notsquig{n,m+1} \bar x$.

    \textbf{Case 2.} $Z \to y_{m+1}$ and $Z \to \bar y_{m+1}$.  Then $x \squig{n,m+1} \bar x$.

    \textbf{Case 3.} $Z \to y_{m+1}$ and $Z \not\to \bar y_{m+1}$.  Then we get, for $|Z \cap X| = \ell$ and $|Z \cap Y| = k$,
    \[ \mb{P}(x \squig{n-\ell,m+1-k} \bar x). \]

    \textbf{Case 4.} $Z \not\to y_{m+1}$ and $Z \to \bar y_{m+1}$.  This is the same as Case 3.
    
    Thus for $|Z \cap X| = \ell$ and $|Z \cap Y| = k$,
    \begin{align*}
        \mb{P}(x \squig{n,m+1} \bar x \,|\, L_{n,m,p}^+(x) = Z) & = [1 - (1-p)^k]^2 + 2[1 - (1-p)^k](1-p)^k \mb{P}(x \squig{n-\ell,m+1-k} \bar x) \\
        & \leq p^2 k^2 + 2pk \, \mb{P}(x \squig{n-\ell,m+1-k} \bar x) \\
        \mb{P}(x \squig{n,m+1} \bar x) - \mb{P}(x \squig{n,m} \bar x) & \leq p^2 \sum_k \sum_\ell P_{m,n,p}(k,\ell) k^2 + 2p\, \mb{P}(x \squig{n,m} \bar x) \sum_k \sum_\ell P_{m,n,p}(k,\ell) k \\
        & \leq \begin{cases} O(1/(n|\lambda|^3)) & \lambda < 0, \\ O(1/n) & \lambda > 0, \end{cases}
    \end{align*}
    using the same arguments as for the previous.
\end{proof}

\begin{corollary}
    If $k,\ell \leq n^{2/3}/|\lambda|$, then
    \[ \mb{P}(x \squig{n,m} \bar x) - \mb{P}(x \squig{n-\ell,m-k} \bar x) = \begin{cases} O((k+\ell)/(n|\lambda|^3)) & \lambda < 0, \\ O((k+\ell)/n) & \lambda > 0. \end{cases} \]
\end{corollary}

\begin{proof}
    We write the telescoping sum
    \begin{multline*} \mb{P}(x \squig{n,m} \bar x) - \mb{P}(x \squig{n-\ell,m-k} \bar x) = \sum_{i=n-\ell}^{n-1} \paren{\mb{P}(x \squig{i+1,m-k} \bar x) - \mb{P}(x \squig{i,n-k} \bar x)} \\ + \sum_{j=m-k}^{m-1} \paren{\mb{P}(x \squig{n,j+1} \bar x) - \mb{P}(x \squig{n,j} \bar x) }. \end{multline*}

    Fix some pair $n',m'$.  Define
    \[ \lambda' = \frac{(m'n')^{1/2} - (mn)^{1/2}}{(mn)^{1/2}} (m'n')^{1/6} + \lambda \paren{\frac{m'n'}{mn}}^{2/3} \]
    so that
    \[ p = \frac{1+\lambda (mn)^{-1/6}}{2(mn)^{1/2}} = \frac{1 + \lambda' (m'n')^{1/6}}{2(m'n')^{1/2}}. \]

    Observe that, for pairs $n',m'$ that will appear in the telescoping sum, we have $\lambda' = (1+o_{\lambda,\eps}(1)) \lambda$.  Thus, using our remark that our earlier work is sufficiently strong to handle $P_{n',m',p}(k,\ell)$, we bound each term in the summation by \cref{lem:inductive-step-for-x-to-x} to give the result.
\end{proof}

Returning to \cref{eq:FKG-n-minus-ell}, for $\lambda < 0$,
\begin{align*}
    & \qquad \mb{P}(x \squig{n,m} \bar x, y \squig{n,m} \bar y) - \mb{P}(x \squig{n,m} \bar x) \mb{P}(y \squig{n,m} \bar y) \\
    & \leq \frac1n \mb{P}(x \squig{n,m} \bar x) \sum_k \sum_\ell P_{m,n,p}(k,\ell) \ell + \sum_k \sum_\ell P_{m,n,p}(k,\ell) \paren{\mb{P}(x \squig{n,m} \bar x) - \mb{P}(x \squig{n-\ell,m-k} \bar x)} \\
    & \leq \frac1n O(n^{-1/3} |\lambda|^{-2}) \cdot O(|\eps|^{-1}) + \sum_{k + \ell \leq n^{2/3}/|\lambda|} P_{m,n,p}(k,\ell) \paren{\mb{P}(x \squig{n,m} \bar x) - \mb{P}(x \squig{n-\ell,m-k} \bar x)} \\ & \qquad \qquad + \sum_{k + \ell \geq n^{2/3}/|\lambda|} P_{m,n,p}(k,\ell) \mb{P}(x \squig{n,m} \bar x) \\
    & \leq O(n^{-1} |\lambda|^{-3}) + \frac{O(1)}{n|\lambda|^3} \sum_{k,\ell} P_{m,n,p}(k,\ell) (k+\ell) + |\eps|^{-1} \exp[-\Omega(|\lambda|^{3/5})] \cdot O(n^{-1/3} |\lambda|^{-2}) \\
    & \leq O(n^{-1} |\lambda|^{-3}) + \frac{O(1)}{n|\lambda|^3} O(|\eps|^{-1}) + n^{-1/3} |\lambda|^{-2} |\eps|^{-1} \exp[-\Omega(|\lambda|^{3/5})] \\
    & \leq O(n^{-1} |\lambda|^{-3}) + O(n^{-2/3} |\lambda|^{-4}) + n^{-2/3} \exp[-\Omega(|\lambda|^{3/5})] \\
    & = O(n^{-2/3} |\lambda|^{-4}).
\end{align*}

For $\lambda > 0$,
\begin{align*}
    & \qquad \mb{P}(x \squig{n,m} \bar x, y \squig{m,n} \bar y) - \mb{P}(x \squig{n,m} \bar x) \mb{P}(y \squig{m,n} \bar y) \\
    & \leq \frac1n \mb{P}(x \squig{n,m} \bar x) \sum_k \sum_\ell P_{m,n,p}(k,\ell) \ell + \sum_k \sum_\ell P_{m,n,p}(k,\ell) \paren{\mb{P}(x \squig{n,m} \bar x) - \mb{P}(x \squig{n-\ell,m-k} \bar x)} \\
    & \leq \frac1n O(n^{-1/3} |\lambda|) \cdot O(|\eps|^{-1}) + \sum_{k + \ell \geq n^{2/3}/|\lambda|} P_{m,n,p}(k,\ell) \paren{\mb{P}(x \squig{n,m} \bar x) - \mb{P}(x \squig{n-\ell,m-k} \bar x)} \\ & \qquad \qquad + \sum_{k + \ell \le n^{2/3}/|\lambda|} P_{m,n,p}(k,\ell) \mb{P}(x \squig{n,m} \bar x) \\
    & \leq O(n^{-1}) + O(1/n) \sum_{k + \ell \geq n^{2/3}/|\lambda|} P_{m,n,p}(k,\ell) (k+\ell) + |\eps|^{-1} \exp[-\Omega(|\lambda|^{3/5})] \cdot O(n^{-1/3} |\lambda|) \\
    & = O(n^{-1}) + O(n^{-2/3} \lambda^{-1}) + n^{-2/3} \exp[-\Omega(\lambda^{3/5})] = O(n^{-2/3} \lambda^{-1}).
\end{align*}

Thus
\[ \mb{P}(x \squig{n,m} \bar x, y \squig{m,n} \bar y) - \mb{P}(x \squig{n,m} \bar x) \mb{P}(y \squig{m,n} \bar y) = \begin{cases} O(n^{-2/3} |\lambda|^{-4}) & \lambda < 0, \\ O(n^{-2/3} |\lambda|^{-1}) & \lambda > 0, \end{cases} \]
completing the proof of \cref{lem:FKG-for-Pxx}.

\begin{proof}[Proof of \cref{lemma:second-moment-bound}]
    Using \cref{prop:spine-expectation} and \cref{lem:FKG-for-Pxx},
    \begin{align*}
        \mb{P}\paren{x \squig{n,m} \bar x, y \squig{m,n} \bar y} & = \left[ \mb{P}\paren{x \squig{n,m} \bar x, y \squig{m,n} \bar y} - \mb{P}\paren{x \squig{n,m} \bar x} \mb{P}\paren{y \squig{m,n} \bar y} \right] + \mb{P}\paren{x \squig{n,m} \bar x} \mb{P}\paren{y \squig{n,m} \bar y} \\
        & = \begin{cases} O(n^{-2/3} |\lambda|^{-4}) + O(n^{-2/3} |\lambda|^{-4}) & \lambda < 0, \\ O(n^{-2/3} |\lambda|^{-1}) + O(n^{-2/3} |\lambda|^2) & \lambda > 0. \end{cases} \qedhere
    \end{align*}
\end{proof}

\section{Upper bounds}\label{sec:SAT-upper-bounds}

In this section, we will bound the probability of satisfiability from above.

\begin{theorem}\label{thm:upper-bound}
    For all $C > 0$, there exist constants $\lambda_0$ sufficiently large and $\eps_0$ sufficiently small such that for two integers $n,m$ with $n \to \infty$, $|n-m| < Cn^{2/3}$, and $p = \frac{1}{2\sqrt{mn}} (1+\lambda(mn)^{-1/6})$ with $\lambda_0 < |\lambda| < \eps_0 (mn)^{1/6}$, we get the following upper bounds on bipartite random 2-SAT.
    \[ \mb{P}[\SAT(F_{n,m,p})] = \begin{cases} \exp(-\Omega(\lambda^3)) & \lambda > 0, \\ \exp(-\Omega(|\lambda|^{-3})) & \lambda < 0. \end{cases} \]
\end{theorem}

Our main tool is a bipartite version of the so-called ``hourglass theorem.''

\begin{definition}[{\cite[Definition 4.2]{bollobas2001scaling}}]\label{def:hourglass}
    An \emph{hourglass} is a triple $ (v, I_v, O_v) $ where $ v $ is a literal, and $ I_v $ and $ O_v $ are two disjoint sets of literals not containing $ v $, such that for each $ x \in I_v $, there exists a path $ x \squig{} v $ in $ I_v \cup \{v\} $, and for each $ y \in O_v $, there exists a path $ v \squig{} y $ in $ O_v \cup \{v\} $. Furthermore, we require that $ \{v\} \cup I_v \cup O_v $ is strictly distinct. We refer to $ v $ as the \emph{central literal}, $ I_v $ as the \emph{in-portion}, and $ O_v $ as the \emph{out-portion} of the hourglass.
\end{definition}

In the bipartite setting, we say that an hourglass $ (v, I_v, O_v) $ is \emph{balanced} if \[ \frac{|I_v \cap X|}{|I_v \cap Y|} \in \left(\frac{1}{5},5\right) \qquad \text{and} \qquad \frac{|O_v \cap X|}{|O_v \cap Y|} \in \left(\frac{1}{5},5\right). \]

We begin with the subcritical case, where we will have many hourglasses.

\begin{lemma}\label{lem:hourglass-subcritical}
    Let $m,n,p,\lambda$ be as in \cref{thm:upper-bound} with $\lambda < 0$.  Then with probability $1 - \exp(-\Omega(|\lambda|^3))$, there are at least $\Theta(|\lambda|^3)$ disjoint, mutually strictly distinct, balanced hourglasses with in-portion and out-portion each of size at least $(mn)^{1/3}/\lambda^2$.
\end{lemma}

\begin{proof}
    Recall $\lambda_0 < |\lambda| < \eps_0 (mn)^{1/6}$ for $\lambda_0$ an arbitrarily large constant and $\eps_0$ an arbitrarily small constant to be chosen later.  We will explore the in- and out-graphs to grow hourglasses one by one.

    To avoid changing of the distribution of hourglasses during the process of depleting variables, we keep a reservoir of variables and only check the others so that we may replace them as we go.  Let $N = \min\{n,m\} - |\lambda| n^{2/3}$ and define $\lambda'$ by
    \[
       \frac{1}{2\sqrt{mn}}(1+\lambda(mn)^{-1/6}) = p = \frac{1}{2N}(1+\lambda'N^{-1/3}).
    \]
    Clearly $\lambda' = (1+o(1))\lambda$.  Place arbitrary total orders on the variables $\{x_1,\ldots,x_n\}$ and $\{y_1,\ldots,y_m\}$.  We will now search for an hourglass in only the first $N$ variables from each side.

    Let $z$ be the first literal in the larger part of $X$ and $Y$, breaking ties arbitrarily.  Explore the graph from $z$ by exposing only its neighbors whose literal uses one of the first $N$ variables from its component.  Iterate this on the neighbors to find only those vertices $u$ such that there is a path $z \squig{} u$ that does not use any $x_i$ or $y_i$ for $i > N$.  Let $T$ be the induced graph on these vertices.  Recall $R_{{n},{m},p}(k,\ell)$ denotes the probability that for an arbitrary $x \in X$, $|L^+_{n,m,p}(x) \cap X| = k$, $|L^+_{n,m,p}(x) \cap Y| = \ell$, and the induced subgraph on $L^+_{n,m,p}(x)$ is a (directed) tree.  By \cref{lem:Qnmp-expanded-formula,lem:k/ell-bounded}, with probability $(c+o(1))|\lambda| n^{-1/3}$ for $c>0$ an absolute constant, $T$ is a tree with $|V(T)| \in [2|\lambda|^{-2}N^{2/3}, 4|\lambda|^{-2}n^{2/3}]$ and $T$ is balanced, i.e.~${|T \cap X|}/{|T \cap Y|} \in (1/5,5)$.  For $v \in V(T)$, we say that $v$ is ``promising'' if the path from $z$ to $v$ has length at least $\sqrt{|T|/2} \geq |\lambda|^{-1}n^{\frac{1}{3}}$ and $v$ has at least $|T|/2$ descendants in $T$.

    We will find $v$ by choosing $w \ne z$ uniformly at random and letting $v$ be the midpoint of the unique path $w \squig{} z$.  Note that conditioned on being a tree, $T$ is uniformly distributed among spanning trees on $V(T)$ respecting the bipartition, so with probability at least $(1-o_{|T|}(1))e^{-1}$, the path $z \squig{} w$ has length at least $\sqrt{2|T|}$ (see \cite[Section 6]{aldous1990random}).  By symmetry and uniform randomness, with probability at least $1/2$, at least half of the remaining vertices of the tree will be connected to $v$ via the path from $v$ to $w$ and so be in the out-graph of $v$.  Thus
   \[
    \mb{P}[v\text{ is promising}] = (1+o(1))\frac{c|\lambda|}{2en^{1/3}}.
   \]

    Given that $v$ is promising, we check the first $|\lambda|^{-1}n^{1/3}$ vertices on the path $z \squig{} v$ to find one with a large in-graph.  Again by \cref{lem:Qnmp-expanded-formula,lem:k/ell-bounded}, each individual in-graph will independently be balanced with size at least $|\lambda|^{-2}n^{2/3}$ with probability at least $(c+o(1))(|\lambda| n^{-1/3})$.  Thus the probability all of these vertices fail is at most $(1+o(1))e^{-c}$.  
   If any one of them does not fail, then we have found an hourglass centered at $v$ with out-portion and in-portion both of size at least $\lambda^{-2}n^{2/3}$.

   Therefore whenever we select a literal $z$ and search for a large, balanced hourglass as outlined above, we find one with probability at least
   \[
        (1+o(1))\frac{c(1-e^c)|\lambda|}{2en^{1/3}}.
   \]

   Once this process is complete, we will discard both literals from all variables that we have explored, replace them from the reservoir and repeat.

    We now compute the number of variables expended in one search.  For any in- or out-graph explored in this process, by \cref{lem:Qnmp-sum-k/ell-bounded}, the expected number of variables searched is $(1+o(1)) n^{1/3}/|\lambda|$.  We always explore only one out-graph and with probability $(1+o(1))[c/(2e)]|\lambda| n^{-1/3}$ (if $v$ is promising) we explore at most $\left\lceil |\lambda| n^{1/3}\right\rceil$ in-graphs. Hence, the total expected number of variables explored is
   \[
        (1+o(1)) \frac{2e+c}{4e} \cdot |\lambda|^{-1}n^{1/3}.
   \]
   
   If we search for an hourglass $\frac{4e}{c(1-e^c)} |\lambda|^{-1}n^{1/3}$ times, the probability of failing to find one is $(1+o(1))e^{-2}$, and the expected number of variables that have been used up is
   \[
        (1+o(1)) \frac{2e+c}{c(1-e^c)}|\lambda|^{-2}n^{2/3}. 
   \]
   
   The probability that we use up more than $3(1 + o(1))$ times the expected number of variables is at most $\frac{1}{3} +o(1)$. Therefore, with probability at least $1-e^{-2}-\frac{1}{3}-o(1) \geq \frac{1}{2}$ we find an hourglass, and do not use up more than $b|\lambda|^{-2}n^{2/3}$ variables, where $K = \frac{3(2e+c)}{c(1-e^c)}$.

   Now consider the following modification to the procedure described above: as before, we employ the local search method to explore the trimmed out-graphs and in-graphs in an attempt to locate an hourglass. However, we terminate each search as soon as we exhaust $b|\lambda|^{-2}n^{2/3}$ variables. This process can be repeated $|\lambda|^3/b$ times while ensuring that no variables beyond the first $n$ are used.  Furthermore, the number of disjoint hourglasses found follows a distribution that stochastically dominates the binomial distribution $\on{Bin}(|\lambda|^3/b, 1/2)$. By the Chernoff bound, the probability of finding fewer than half the expected number of hourglasses is at most $\exp[-|\lambda|^3/(8b)]$.
\end{proof}

In the supercritical regime, we instead have a giant hourglass.

\begin{lemma}\label{lem:hourglass-supercritical}
    Let $m,n,p,\lambda$ be as in \cref{thm:upper-bound} with $\lambda > 0$.  Then with probability $1 - \exp(-\Omega(|\lambda|^3))$, there is an hourglass with in-portion and out-portion each of size $\Theta(\lambda(mn)^{1/3})$.
\end{lemma}

\begin{proof}
    We will follow the previous proof in the subcritical regime and then show that, after increasing $p$, the hourglasses will percolate into one giant hourglass whp.

    Let $q = (1 - \lambda (mn)^{-1/6}) / (2 \sqrt{mn})$.  By the previous section, with probability at least $1-\exp(-\Theta(\lambda^3)) $, there will be $\Theta(\lambda^3)$ balanced hourglasses in $G_{F_q}$, each with an in-portion and out-portion of size at least $\lambda^{-2}(mn)^{1/3}$.  Create a directed graph $\Gamma$ with $V(\Gamma)$ the collection of $\Theta(\lambda^3)$ hourglasses and $E(\Gamma) = \{ ((v,I_v,O_v), (u,I_u,O_u) ) : v \squig{F_p} u \}$.

    We claim $E(\Gamma)$ stochastically dominates independent random choice with probability $\frac1{75} M/\lambda^3$.  Indeed, for each pair $(v,I_v,O_v)$ and $(u,I_u,O_u)$, if there is an edge introduced between $O_v$ and $I_u$, then $v \squig{} u$.  Thus
    \begin{multline*} \mb{P}[v \squig{F_p} u] \geq \mb{P}[e(O_v,I_u) > 0] = 1 - (1-(p-q))^{|I_u \cap X| |O_v \cap Y| + |I_u \cap Y| |O_v \cap X|} \\ \geq 1 - \exp\paren{-M \lambda (mn)^{-2/3} \paren{\frac{\lambda^{-2} (mn)^{1/3}}{5}}^2 } \geq \frac1{75} M \lambda^{-3}. \end{multline*}

    Further, since all sets $\{I_u\}$ and $\{O_u\}$ are mutually disjoint across all $u$, these are independent events.  Thus $\Gamma$ has $\Theta(\lambda^3)$ vertices and edge probability at least $\frac13 M \lambda^{-3}$, and so by choosing $M$ sufficiently large based on the implicit constants, we can ensure $\Gamma$ has expected average degree $6$.  By \cite[Lemma 9.1]{bollobas2001scaling}, this means $\Gamma$ has a vertex $u$ with in-neighborhood and out-neighborhood both of constant size, i.e.~$\Theta(\lambda^3)$ vertices.  Since each of these vertices corresponds to an hourglass with in- and out-portions of size $(mn)^{1/3} \lambda^{-2}$, this means that the central vertex of $u$ has both in- and out-portions of size $\Theta((mn)^{1/3} \lambda)$ in the original graph $G_{F_p}$.
\end{proof}

\begin{proof}[Proof of \cref{thm:upper-bound}]
    The proof here stems from the fact that given an hourglass $(v,I_v,O_v)$, for any $u,u' \in I_v$, $w,w' \in O_v$, adding the clauses $u \vee u'$ and $\bar w \vee \bar w'$ guarantees unsatisfiability.

    \subsection*{Subcritical regime (\texorpdfstring{$\lambda < 0$}{negative lambda})}
    Let $q = (1 + 2\lambda(mn)^{-1/6})/(2(mn)^{1/2})$.  By \cref{lem:hourglass-subcritical}, in $G_{F_q}$, there are $\Theta(|\lambda|^3)$ hourglasses each of size $\frac14 (mn)^{1/3}|\lambda|^{-2}$.  We now increase from $q$ to $p$ by adding each missing edge with probability $(q-p)/(1-p) = \Theta(|\lambda| (mn)^{-2/3})$.  For any fixed hourglass found by \cref{lem:hourglass-subcritical}, we have added both $u \vee u'$ and $\bar w \vee \bar w'$ for $u,u' \in I_v$, $w,w' \in O_v$ with probability $\Theta(|\lambda|^{-6})$.  By inclusion-exclusion, we have created a contradictory cycle via one of the hourglasses with probability at least $\Theta(|\lambda|^{-3})$, and so we are satisfiable with probability only
    \[ 1 - \Omega(|\lambda|^{-3}) = \exp(-\Omega(|\lambda|^{-3})). \]

    \subsection*{Supercritical regime (\texorpdfstring{$\lambda > 0$}{positive lambda})}
    Let $q = (1 + \frac12\lambda(mn)^{-1/6})/(2(mn)^{1/2})$.  By \cref{lem:hourglass-supercritical}, in $G_{F_q}$, we have a giant hourglass except with probability $\exp(-\Theta(\lambda^3))$.  We now increase from $q$ to $p$ by adding each missing edge with probability $(q-p)/(1-p) = \Theta(\lambda (mn)^{-2/3})$.  The probability we do not add $u \vee u'$ for any $u,u' \in I_v$ is at most
    \[ \paren{1 - \Theta(\lambda (mn)^{-2/3})}^{\Theta(\lambda (mn)^{1/3})^2} = \exp(-\Omega(\lambda^3)), \]
    and similarly we add $\bar v \vee \bar v'$ for some $v,v' \in O_v$ except with probability $\exp(-\Omega(\lambda^3))$, completing the proof.
\end{proof}

\section*{Acknowledgments}

CM supported in part by a student fellowship from the IDEAL Institute for Data, Econometrics, Algorithms, and Learning and by a Simons Dissertation Fellowship.  WP supported in part by NSF grant DMS-2348743.  

\providecommand{\bysame}{\leavevmode\hbox to3em{\hrulefill}\thinspace}
\providecommand{\MR}{\relax\ifhmode\unskip\space\fi MR }
\providecommand{\MRhref}[2]{%
  \href{http://www.ams.org/mathscinet-getitem?mr=#1}{#2}
}
\providecommand{\href}[2]{#2}

\appendix

\section{Proof of \texorpdfstring{\cref{lem:expander}}{expander lemma}}\label{app:expander}

\begin{proof}
    By \cref{lem:strongly-balanced}, we have $|A|,|B| \geq n/3$.  For any $v \in A$, recall $d_\mbf{G}(v,B) = |E_\mathrm{cr} \cap (\{v\} \times B)|$ with $E_\mathrm{cr}$ a sample from the hard-core model.  
    Thus by \cref{lem:hardcore-quasirandom},
    \[ \mb{P} \left[ d_\mbf{G}(v,B) \leq \frac{\lambda n}{30}\right] \leq \Pr\left[d_\mbf{G}(v,B) \leq \frac{\lambda |B|}{10}\right] \leq e^{-\lambda |B|/8} \leq e^{-\lambda n/24} \]
    and
    \[ \mb{P} \left[ d_\mbf{G}(v,B) \geq 5\lambda n \right] \leq \Pr\left[d_\mbf{G}(v,B) \geq 5 \lambda |B| \right] \leq e^{-\lambda|B|} \leq e^{-\lambda n/3}. \]
    
    By a union bound and $n/3 \leq |B| \leq n$,
    \[ \Pr\left[\frac{\lambda n}{30} \leq d_\mbf{G}(v,B) \leq 5\lambda n  \text{ for all }v \in A\right] \geq 1 - n(e^{-\lambda n/24}+e^{-2\lambda n/3}) \geq 1 - e^{-\lambda n/25}. \]

    Fix $X$ and $Y$ as in the statement of \cref{lem:expander}.  We will prove a stronger statement.  As $d_\mbf{G}(X,Y) = |E_\mathrm{cr} \cap X \times Y|$, by \cref{lem:hardcore-quasirandom}, we get
    \[ \Pr\left[d_\mbf{G}(X,Y) \leq \frac{\lambda |X| |Y|}{10}\right] \leq e^{-\lambda |X||Y|/8} \leq e^{-\lambda^2n^2/7200}. \]

    By a union bound over all possible choices of $X$ and $Y$, we get
    \[ \Pr\left[d_\mbf{G}(X,Y) \geq \frac{\lambda |X||Y|}{10} \text{ for all such }X,Y\right] \geq 1 - 2^n e^{-\lambda^2n^2/7200} = 1 - e^{-\Omega(n \log n)}, \]
    where the last equality uses $\lambda = \Theta((\log n/n)^{1/2})$.  As $\lambda |X||Y|/10 = \Omega(1)$, this gives the desired result on $d_\mbf{G}(X,Y)$.
\end{proof}

\section{Approximately counting connected bipartite graphs}\label{sec:proof-of-lower-bound-bipartite}

Here we show a matching lower bound for $C(k,\ell,s)$ to the upper bound of Do, Erde, Kang, and Missethan~\cite[Theorem 3.5]{do2021component}.  The two bounds are off by an exponential factor in $s$, but this is sufficient for our purposes.

\begin{proof}[Proof of \cref{prop:lower-bound-on-bipartite-connected-graphs}]
    Let $m = s-1$ so that we are counting graphs with $k+\ell+m$ edges.  We will use the so-called kernel-core method (see \cite{bollobas1984evolutionsparse,luczak1990component,luczak1990number}).  Given a graph $G$, the \textit{core}, or 2-core, refers to the unique maximal subgraph of minimum degree 2.  This can be attained by recursively removing isolated vertices and leaves.  The \textit{kernel} is the multigraph obtained by recursively removing vertices of degree 2 from the core and connecting their two neighbors with an edge.  The vertices of the kernel are exactly the vertices of the core of degree $> 2$.

    We will count only graphs whose kernel is a $2m$-vertex 3-regular bipartite graph.  This core will have $3m$ edges and we will connect the remaining $k+\ell-2m$ vertices as a forest (one edge each) guaranteeing $k+\ell+m$ edges.  We first choose the vertices of the kernel, for which there are
    \[ \binom km \binom \ell m \]
    choices.  Next, we choose the 3-regular bipartite graph on the kernel.  Using the configuration model (see \cite{mckay1984asymptotics}), the number of labelled 3-regular bipartite graphs with $m$ vertices per side is
    \[ \exp[-2 + o(1)] {\frac{(3m)!}{(3!)^{2m}}}. \]

    Now, we must choose the remaining vertices in the core.  Let $u$ be the number of vertices on each side in the core of $G$ (including the $m$ vertices in the kernel).  Thus we have
    \[ \binom{k-m}{u-m} \binom{\ell-m}{u-m} \]
    choices to finish the core.

    We now must enumerate ways to add these remaining core vertices while respecting the kernel and the bipartition.  First, order the $u-m$ new vertices on each side.  This pairs up vertices on opposite sides of the bipartition, so that they can be added together, turning paths of length 1 into paths of length 3 and so respecting the bipartition.  Finally, we must choose how much we wish to extend each edge from the kernel.  This can be counted via ``stars and bars,'' giving
    \[ ((u-m)!)^2 \binom{u+2m-1}{3m-1} \]
    choices.

    Finally, we must assign a forest on the non-core vertices.  We will use the following result of Moon~\cite{moon1967enumerating}.

    \begin{lemma}[\cite{moon1967enumerating}]
        Given $k,\ell,a,b \in \mb{N}$ satisfying $a \leq k$ and $b \leq \ell$, let $F(k,\ell,a,b)$ denote the number of bipartite forests with partition classes $I = \{x_1,\ldots,x_k\}$ and $J = \{y_1,\ldots,y_\ell\}$ with $a+b$ components where the vertices $x_1,\ldots,x_a,y_1,\ldots,y_b$ belong to distinct components.  Then
        \[ F(k,\ell,a,b) = k^{\ell-b-1} \ell^{k-a-1} (a\ell + bk - ab). \]
    \end{lemma}

    We will use this with $a=b=u$.  Putting this all together, we get
    \begin{align*}
        C(k,\ell,m) & \gtrsim \sum_u \binom km \binom \ell m \frac{(3m)!}{36^m} \binom{k-m}{u-m} \binom{\ell-m}{u-m} ((u-m)!)^2 \binom{u+2m-1}{3m-1} F(k,\ell,u,u) \\
        & = \frac{(3m)!}{(m!)^2 36^m} k^{\ell-1} \ell^{k-1} \sum_u \frac{(k)_u (\ell)_u}{k^u \ell^u} \binom{u+2m-1}{3m-1} \paren{u\ell + uk - u^2} \\
        & \gtrsim m^m c^m k^{\ell-1} \ell^{k-1} \sum_{u=\sqrt{(k+\ell)m}}^{2\sqrt{(k+\ell)m}} \exp\left[ -\frac{u^2}{(k+\ell)} \right] \paren{\frac{u+2m-1}{3m-1}}^{3m-1} {u(k+\ell)} \\
        & \gtrsim m^m c^m k^{\ell-1} \ell^{k-1} \sqrt{(k+\ell)m} \exp[-4m] \paren{\frac{k+\ell}{c'm}}^{(3m-1)/2} {(k+\ell)^{3/2} m^{1/2}} \\
        & \gtrsim k^{\ell-1} \ell^{k-1} (c'')^m \paren{\frac{(k+\ell)^3}{m}}^{(m+1)/2}.
    \end{align*}

    Recalling $s = m+1$ (and paying a factor of $1/e$), we have the result.
\end{proof}

\section{Proof of \texorpdfstring{\cref{claim:variable-partial-assignment}}{triangle claim}}\label{app:variable-partial-assignment}
This proof will follow from the machinery developed in \cref{sec:SAT-lower-bounds}.  We will use the language of that section; i.e.~$n$, $\lambda$, $p$, $X$, $Y$ refer to the quantities defined in \cref{sec:2-SAT-intro}.

The content of the claim is that if $z \in X \cup Y$ has a variable $v$ with $v \to z$ and $\bar v \to z$, then whp, $z \notin S(F)$.  

\begin{proof}[Proof of \cref{claim:variable-partial-assignment}]
Let
\[ Z = \{ z \in X \cup Y : \exists v \text{ s.t. } v \to z, \ \bar v \to z \}. \]

First, we compute that $\mb{E}[Z] = O(1)$.  Indeed, let $x \in X$.  Then by a union bound,
\[ \mb{P}[x \in Z] \leq m p^2 = \frac{(1+\eps)^2}{2n} \leq \frac2{3n}, \]
so $\mb{E}|Z \cap X| \leq 2/3$.  Similarly, $\mb{E}|Z \cap Y| \leq 2/3$, so $\mb{E}|Z| \leq 4/3$.  We will union bound over the elements of $Z$.

Let $x \in Z \cap X$.  Then there exists a variable $y \in Y$ with $y \to x, \bar y \to x$.  For $x$ to cause unsatisfiability, we must have $x \in S(F)$.  Consider $\widetilde{X} = X, \widetilde{Y} = Y \setminus \{y,\bar y\}$.  In this set, $x$ is generic.  Then
\begin{align*}
    \mb{P}[x \squig{X,Y} \bar x] & \leq \mb{P}[x \squig{X,Y \setminus \{y,\bar y\}} \bar x] + \mb{P}[L^+_{m,n,p}(x) \to y] + \mb{P}[L^+_{m,n,p}(x) \to \bar y] \\
    & \leq \mb{P}[x \squig{n,m-2} \bar x] + 2\sum_k \sum_\ell P_{n,m,p}(k,\ell) pk \\
    & \leq O(n^{-1/3} |\lambda|) + p \cdot O(n^{1/3}/|\lambda|)
    = O(|\eps|).
\end{align*}

The second line uses a union bound and the fact that $x$ is generic in $X \cup (Y \setminus \{y,y'\})$.  The third uses \cref{prop:spine-expectation,lem:moments-of-Pnp}.
The same argument with $m$ and $n$ exchanged will work for the case $x \in Y$.

Thus by a union bound,
\[ \mb{P}[Z \cap S(F) \ne \emptyset] \leq \mb{E}|Z \cap S(F)| = \mb{E}|Z| \mb{P}[z \in S(F) \,|\, z \in Z] = O(|\eps|). \]

$Z \cap S(F) = \emptyset$ whp and so the existence of $Z$ does not affect satisfiability.
\end{proof}

\section{Proof of \texorpdfstring{\cref{lem:Qnp-bounds}}{branching process lemma}}\label{app:branching-process}

\begin{proof}
We begin with the first inequality.  Recall as mentioned in the proof of \cref{lem:Qnp-minus-Pnp} that $Q_{n,m,p}(k,\ell)$ is exactly the distribution the cluster size in $G(n,m,\wt p)$ where $\wt p := 2p-p^2$.  The cluster size distribution in $G(n,m,\wt p)$ is stochastically dominated by a $\on{Binom}(\max\{m,n\},\wt p)$ birth process, which is in turn stochastically dominated by a a Poisson birth process with parameter
\[ \hat{\kappa} := \max\{m,n\} \log(1/(1-(2p-p^2))) = 2(mn)^{1/2} p(1+O(n^{-1/3})) = (1+\eps)(1+O(n^{-1/3})). \]

Thus it suffices to analyze this Poisson birth process instead.  By \cite[eq.~(B.1)--(B.8)]{bollobas2001scaling}, this gives
\[ \sum_{k+\ell \geq \lambda n^{2/3}} Q_{n,m,p}(k,\ell) \leq (2+o_{\lambda,\eps}(1))\eps. \]

We now show the second inequality, studying cluster size with the method of Karp~\cite{karp1990transitive}.

Let $N_0 = n-1$ and $N_t \sim \on{Binom}(N_{t-1},1-\wt{p})$.  Let $M_0 = m$ and $M_t \sim \on{Binom}(M_{t-1},1-\wt{p})$.  Let $X_{r,s} = n - N_r - s$ and $Y_{r,s} = m - M_s - r$.  Finally, let $(r^*,s^*)$ be the $r$ and $s$ minimizing $r+s$ such that $X_{r,s} = Y_{r,s} = 0$. 

\begin{claim}\label{app:claim:cluster-size}
    Let $A,B$ be sets of size $n,m$ respectively and let $v \in A$.
    The law of $(r^*,s^*)$ is exactly the law of $(|C(v) \cap A|, |C(v) \cap B|)$ in $G(A,B,\wt{p})$.
\end{claim}

\begin{proof}
    At step $(r,s)$, we have viewed the (out)neighborhoods of $s$ vertices on the LHS and $r$ vertices on the RHS.  There are $N_r$ RHS vertices and $M_s$ LHS vertices that have gone unviewed, and so $n-N_r$ and $m-M_s$ vertices in our component thus far.  Thus we need $n-N_r \geq s$ and $m-M_s \geq r$ so that we have only explored vertices in our component.

    Equivalently, we divide each side into 3 components: explored, seen, unseen.  The sizes are then
    \begin{center}
    \begin{tabular}{c|cc}
         & LHS & RHS \\ \hline 
        Explored & $s$ & $r$ \\
        Seen & $n - N_r - s$ & $m - M_s - r$ \\
        Unseen & $N_r$ & $M_s$
    \end{tabular}
    \end{center}

    This elucidates the definitions $X_{r,s} := n - N_r - s$ and $Y_{r,s} := m - M_s - r$.  We may keep exploring as long as there is a ``seen'' but not ``explored'' vertex.
\end{proof}

We will also need the following claim.

\begin{claim}\label{app:claim:XY-boundary}
    Suppose there exist $r,s \in (n^{2/3}/\lambda,n^{2/3}\lambda)$ with $X_{r,s}, Y_{r,s} \leq 0$.  Then one of the following is true:
    \begin{enumerate}
        \item there exist $r,s \in (n^{2/3}/\lambda,n^{2/3}\lambda)$ with $X_{r,s} = Y_{r,s} = 0$,
        \item $X_{n^{2/3}/\lambda,n^{2/3}/\lambda} \leq 0$,
        \item $Y_{n^{2/3}/\lambda,n^{2/3}/\lambda} \leq 0$.
    \end{enumerate}
\end{claim}

\begin{proof}
    Choose the $r,s \in (n^{2/3}/\lambda,n^{2/3}\lambda)$ minimizing $r+s$ among all those $r,s \geq n^{2/3}/\lambda$ with $X_{r,s},Y_{r,s} \leq 0$.  If $X_{r,s} = Y_{r,s} = 0$, then we are done.

    Otherwise, suppose $X_{r,s} < 0$.  Then $X_{r,s-1} = X_{r,s} + 1 \leq 0$ and $Y_{r,s-1} \leq Y_{r,s} \leq 0$, which would contradict that $r,s$ minimizes $r+s$ unless $s = n^{2/3}/\lambda$.  Thus $X_{r,n^{2/3}/\lambda} \leq 0$.  However, $X_{r,s}$ is monotonically increasing in $r$ and so $X_{n^{2/3}/\lambda,n^{2/3}/\lambda} \leq X_{r,n^{2/3}/\lambda} \leq 0$.

    Similarly, if $Y_{r,s} < 0$, then $r = n^{2/3}/\lambda$ and so by monotonicity $Y_{n^{2/3}/\lambda,n^{2/3}/\lambda} \leq 0$.
\end{proof}

We can now evaluate
\begin{align*}
    & \qquad \mb{P}[n^{2/3}/\lambda \leq r^*,s^* \leq n^{2/3}\lambda] \\ & \leq \mb{P}[n^{2/3}/\lambda \leq r^*,s^*] \mb{P}[\exists r,s \in (n^{2/3}/\lambda,n^{2/3}\lambda) : X_{r,s}, Y_{r,s} \leq 0] \\
    & = O(\eps) \paren{ \mb{P}[\exists r,s \in (n^{2/3}/\lambda,n^{2/3}\lambda) : X_{r,s} = Y_{r,s} = 0] + \mb{P}[X_{n^{2/3}/\lambda,n^{2/3}/\lambda} \leq 0] + \mb{P}[Y_{n^{2/3}/\lambda,n^{2/3}/\lambda} \leq 0] } \\
    & = O(\lambda n^{-1/3}) \paren{ \mb{P}[\exists r,s \in (n^{2/3}/\lambda,n^{2/3}\lambda) : X_{r,s} = Y_{r,s} = 0] + e^{-\Omega(\lambda)} \} }. \label{app:eq:master} \stepcounter{equation} \tag{\theequation}
\end{align*}

The first inequality is FKG.  The second follows from the first inequality in \cref{lem:Qnp-bounds} (which we have already proven) and a union bound via \cref{app:claim:XY-boundary}.  The third is a Chernoff bound; this is also shown in \cite[(B.23)]{bollobas2001scaling}, as when $r=s$ our random variable has identical distribution to theirs.  We now must show this last probability is exponentially small in $\lambda$.

Let $A_{r,s}$ denote the event $\{ X_{r,s} = Y_{r,s} = 0\}$ and $B_{r,s}$ that $A_{r,s}$ holds and there is no $n^{2/3/\lambda} \leq r' \leq r$ and $n^{2/3}/\lambda \leq s' \leq s$ with $A_{r',s'}$ holding unless $r'=r$ and $s'=s$.  Let
\[ S = \sum_{n^{2/3}/\lambda \leq r+s \leq \lambda n^{2/3} + n^{2/3}/\lambda^2} \mathbbm{1}_{A_{r,s}} \geq \sum_{n^{2/3}/\lambda \leq r+s \leq \lambda n^{2/3}} \mathbbm{1}_{B_{r,s}} \sum_{0 \leq \Delta_r, \Delta_s \leq n^{2/3}/\lambda^2} \mathbbm{1}_{A_{r + \Delta_r, s + \Delta_s}} \]
so that
\begin{align*}
    \mb{E}[S] & \geq \sum_{n^{2/3}/\lambda \leq r+s \leq \lambda n^{2/3}} \mb{P}[B_{r,s}] \mb{E}\left[ \sum_{0 \leq \Delta_r, \Delta_s \leq n^{2/3} / \lambda^2} \mathbbm{1}_{A_{r+\Delta_r,s+\Delta_s}} \,|\, B_{r,s} \right] \\
    & \geq \sum_{n^{2/3}/\lambda \leq r+s \leq \lambda n^{2/3}} \mb{P}[B_{r,s}] \min_{r,s} \sum_{0 \leq \Delta_r,\Delta_s \leq n^{2/3}/\lambda^2} \mb{P}[A_{r+\Delta_r, s+\Delta_s} \,|\, A_{r,s}] \\
    & \geq \mb{P}[\exists r,s \in (n^{2/3}/\lambda,\lambda n^{2/3}) : X_{r,s} = Y_{r,s} = 0] \min_{r,s} \sum_{0 \leq \Delta_r,\Delta_s \leq n^{2/3}/\lambda^2} \mb{P}[A_{r+\Delta_r, s+\Delta_s} \,|\, A_{r,s}]. \stepcounter{equation} \tag{\theequation} \label{app:eq:ES-with-min}
\end{align*}

This used the fact that the events $\{A_{r,s}\}$ are Markovian.

We now control this final probability using a central limit theorem.  Fix $r,s,\Delta_r,\Delta_s$.  Define $Z_1,Z_2$ by
\[ X_{r+\Delta_r,s+\Delta_s} - X_{r,s} = Z_1 - \Delta_s, \qquad Y_{r+\Delta_r,s+\Delta_s} - Y_{r,s} \equiv Z_2 - \Delta_r. \]

Note that
\begin{align*}
Z_1 & = Z_1(r,\Delta_r) \sim \on{Binom}(N_r,1-(1-\wt p)^{\Delta_r}), \\
Z_2 & = Z_2(s,\Delta_s) \sim \on{Binom}(M_s,1-(1-\wt p)^{\Delta_s}).
\end{align*}

Thus conditioning on the values of $X_{r,s}$ and $Y_{r,s}$ and using independence of $Z_1$ and $Z_2$, we get
\begin{equation}\label{app:eq:XYZ}
\mb{P}[X_{r+\Delta_r,s+\Delta_s} = 0 = Y_{r+\Delta_r,s+\Delta_s} \,|\, X_{r,s}, Y_{r,s}] = \mb{P}[Z_1 = \Delta_s - X_{r,s}] \mb{P}[Z_2 = \Delta_r - Y_{r,s}]. \end{equation}

Note that we may rewrite the deviation from the expectation
\[ \Delta_s - X_{r,s} - \mb{E}Z_1 = -X_{r,s} + (\Delta_s - \Delta_r) + \Delta_r(1 - \wt pn + \wt ps + \wt pX_{r,s} + O(\Delta_r n \wt p^2)). \]

Let $\Delta = (\Delta_s + \Delta_r)/2$ and let $\delta = (\Delta_s - \Delta_r)/2$.  
Assuming $X_{r,s} = O(1)$, $s = O(\lambda n^{2/3})$, this becomes
\[ \Delta_s - X_{r,s} - \mb{E}Z_1 = -X_{r,s} + 2\delta + O(\Delta \eps), \]
and similarly
\[ \Delta_r - Y_{r,s} - \mb{E}Z_2 = -Y_{r,s} - 2\delta + O(\Delta \eps). \]

If $\delta = \Omega(\Delta^{7/12})$, then by a Chernoff bound, $\mb{P}[Z_1 - \mb{E}Z_1 \geq 2\delta] \leq \exp[-\Omega(\Delta^{1/6})]$.

Otherwise, we will control each term in \cref{app:eq:XYZ} with a normal approximation.  This is valid when
\[ \Delta_s - X_{r,s} - \mb{E}Z_1 \ll \on{Var}(Z_1)^{2/3}. \]

Thus since $\Delta_s - X_{r,s} - \mb{E}Z_1 = O(\Delta^{7/12}) = O(\on{Var}(Z_1)^{7/12})$, valid for $\Delta = O(1/\eps^2) = O(n^{2/3}/\lambda^2)$, we have
\[ \mb{P}[Z_1 = \Delta_s + X_{r,s}] \sim \frac{1}{\sqrt{2\pi\Delta}} \exp\left\{ - \frac{(2\delta -X_{r,s} + O(\Delta \eps))^2}{(2+o(1)) \Delta} \right\} \]
and
\[ \mb{P}[Z_2 = \Delta_r + Y_{r,s}] \sim \frac{1}{\sqrt{2\pi\Delta}} \exp\left\{ - \frac{(2\delta + Y_{r,s} + O(\Delta \eps))^2}{(2+o(1)) \Delta} \right\} \]
and so
\begin{align*}
    & \qquad \sum_{0 \leq \Delta_r \leq n^{2/3}/\lambda^2} \sum_{0 \leq \Delta_s \leq n^{2/3}/\lambda^2} \mb{P}[X_{r+\Delta_r,s+\Delta_s} = 0 \,|\, X_{r,s} = 0] \mb{P}[Y_{r+\Delta_r,s+\Delta_s} = 0 \,|\, Y_{r,s} = 0] \\
    & = \sum_{0 \leq \Delta \leq n^{2/3}/\lambda^2} \sum_{-\Delta^{7/12} \leq \delta \leq \Delta^{7/12}} \mb{P}[X_{r+\Delta-\delta,s+\Delta+\delta} = 0 \,|\, X_{r,s} = 0] \mb{P}[Y_{r+\Delta-\delta,s+\Delta+\delta} = 0 \,|\, Y_{r,s} = 0] \\
    & = \sum_{0 \leq \Delta \leq n^{2/3}/\lambda^2} \sum_{-\Delta^{7/12} \leq \delta \leq \Delta^{7/12}} \frac{O(1)}{\sqrt{\Delta^2-\delta^2}} \exp \left\{ - \frac{(2\delta + O(\Delta^{1/2}))^2}{\Delta - \delta} - \frac{(-2\delta + O(\Delta^{1/2}))^2}{\Delta + \delta} \right\} \\
    & = \sum_{0 \leq \Delta \leq n^{2/3}/\lambda^2} \frac{O(1)}{\Delta} \sum_{-\Delta^{7/12} \leq \delta \leq \Delta^{7/12}} \exp\left\{ - \frac{8\delta^2}{\Delta} + O(\delta \Delta^{-1/2}) + O(1) \right\} \\
    & = \sum_{0 \leq \Delta \leq n^{2/3}/\lambda^2} \Theta(1/\Delta) \cdot \Theta(\Delta^{1/2}) = \sum_{0 \leq \Delta \leq n^{2/3}/\lambda^2} \Theta(\Delta^{-1/2}) = \Theta(n^{1/3}/\lambda). \stepcounter{equation} \tag{\theequation} \label{eq:bound-on-Xrs-Yrs}
\end{align*}

Finally, we must lower bound $\mb{E}[S]$.  By much the same logic except plugging in $X_{0,0} = 1, Y_{0,0} = 0$, we get
\begin{align*}
    \mb{E}[S] & = \sum_{n^{2/3}/\lambda \leq r \leq \lambda n^{2/3} + n^{2/3}/\lambda^2} \ \sum_{n^{2/3}/\lambda \leq s \leq \lambda n^{2/3} + n^{2/3}/\lambda^2} \mb{P}[A_{r,s}] \\
    & \leq \sum_{n^{2/3}/\lambda \leq \Delta \leq \lambda n^{2/3} + n^{2/3}/\lambda^2} \ \sum_{-\Delta \leq \delta \leq \Delta} \mb{P}[A_{r,s}] \\
    & = \sum_{n^{2/3}/\lambda \leq \Delta \leq \lambda n^{2/3} + n^{2/3}/\lambda^2} \ \sum_{-\Delta^{7/12} \leq \delta \leq \Delta^{7/12}} \mb{P}[A_{r,s}] + \exp\{-\Delta^{\Omega(1)} \} \\
    & = e^{-n^{\Omega(1)}} + \sum_{n^{2/3}/\lambda \leq \Delta \leq \lambda n^{2/3} + n^{2/3}/\lambda^2} \ \sum_{-\Delta^{7/12} \leq \delta \leq \Delta^{7/12}} O(\Delta^{-1}) \\ & \qquad \qquad \exp\left\{ - \frac{(\Delta\wt\eps + 2\delta + 1 + O(\delta\eps))^2 + (\Delta\wt\eps - 2\delta + 1 + O(\delta\eps))^2}{(2+o(1))\Delta} \right\}  \\
    & = e^{-n^{\Omega(1)}} + \sum_{n^{2/3}/\lambda \leq \Delta \leq \lambda n^{2/3} + n^{2/3}/\lambda^2} \frac1\Delta \sum_{-\Delta^{7/12} \leq \delta \leq \Delta^{7/12}} \exp\left\{ -\Theta(\delta^2 \Delta^{-1} + \Delta\wt\eps^2) \right\} \\
    & = e^{-n^{\Omega(1)}} + e^{-\Omega(\lambda)} \sum_{n^{2/3}/\lambda \leq \Delta \leq \lambda n^{2/3} + n^{2/3}/\lambda^2} \frac1\Delta \sum_{-\Delta^{7/12} \leq \delta \leq \Delta^{7/12}} \exp\left\{ -\Theta(\delta^2 \Delta^{-1}) \right\} \\
    & = e^{-n^{\Omega(1)}} + e^{-\Omega(\lambda)} \sum_{n^{2/3}/\lambda \leq \Delta \leq \lambda n^{2/3} + n^{2/3}/\lambda^2} \Theta(\Delta^{-1/2}) = e^{-\Omega(\lambda)} n^{-1/3}. \stepcounter{equation} \tag{\theequation} \label{app:eq:ES}
\end{align*}

Thus combining several equations \cref{app:eq:ES-with-min}, \cref{eq:bound-on-Xrs-Yrs}, and \cref{app:eq:ES} we get
\[ \mb{P}[\exists r,s \in (n^{2/3}/\lambda,\lambda n^{2/3}) : X_{r,s} = Y_{r,s} = 0] = \exp(-\Omega(\lambda)). \]

Now plugging it all back into \cref{app:eq:master} and recalling \cref{app:claim:cluster-size},
\[ \mb{P}[n^{2/3}/\lambda \leq |C(v)| \leq \lambda n^{2/3}] = O(n^{-1/3} \exp(-\Omega(\lambda))). \]

This completes the proof of the second inequality.
\end{proof}

\end{document}